\documentclass[12pt]{article}

\usepackage{latexsym}
\usepackage{a4wide}
\usepackage{amscd}
\usepackage{graphics}
\usepackage{amsmath}
\usepackage{amssymb}
\usepackage{soul,xcolor}
\usepackage{esint}
\usepackage{mathrsfs}
\usepackage{amsthm}
\usepackage{bbm}
\usepackage{color}
\usepackage{enumerate}
\usepackage{hyperref}
\usepackage{ulem}
\input xy
\xyoption{all}

\newcommand{\Z}{\mathbb{Z}}                     
\newcommand{\R}{\mathbb{R}}                     
\newcommand{\C}{\mathbb{C}}                     
\newcommand{\T}{\mathbb{T}}                     
\newcommand{\CP}{\mathbb{CP}}                   
\newcommand{\RP}{\mathbb{RP}}                   

\newcommand{\di}{\mathrm{d}}                     

\newtheorem{mainthm}{\sc Theorem}           
\newtheorem{maincor}{\sc Corollary}           
\newtheorem{mainprop}{\sc Proposition}           
\newtheorem{thm}{\sc Theorem}[section]               
\newtheorem*{thm*}{\sc Theorem}               
\newtheorem{cor}[thm]{\sc Corollary}        
\newtheorem*{cor*}{\sc Corollary}        
\newtheorem{lem}[thm]{\sc Lemma}            
\newtheorem{prop}[thm]{\sc Proposition}     
\newtheorem{rem}[thm]{\sc Remark}           

\begin{document}

\title{On the local systolic optimality of Zoll contact forms}

\author{Alberto Abbondandolo\footnote{Ruhr-Universit\"at Bochum, Fakult\"at f\"ur Mathematik, Universit\"atsstrasse 150, 44801 Bochum, Germany} {}\ and Gabriele Benedetti\footnote{Vrije Universiteit Amsterdam, Department of Mathematics, De Boelelaan 1111, 1081 HV Amsterdam, The Netherlands}}
\date{}

\maketitle

\begin{abstract}
We prove a normal form for contact forms close to a Zoll one and deduce that Zoll contact forms on any closed manifold are local maximizers of the systolic ratio. Corollaries of this result are: (i) sharp local systolic inequalities for Riemannian and Finsler metrics close to Zoll ones, (ii) the perturbative case of a conjecture of Viterbo on the symplectic capacity of convex bodies, (iii) a generalization of Gromov's non-squeezing theorem in the intermediate dimensions for symplectomorphisms that are close to linear ones.
\end{abstract}

\section*{Introduction}

\subsection*{Metric systolic geometry} 

A classical problem in Riemannian geometry consists in bounding from above the length of the shortest closed geodesic on a closed Riemannian manifold $(W,g)$ by the volume of the  manifold. In other terms, one asks if the systolic ratio of $(W,g)$, i.e.\ the scaling invariant quantity
\[
\rho_{\mathrm{sys}}(W,g) := \frac{\ell_{\min}(g)^n}{\mathrm{vol}(W,g)},
\]
where $n=\dim W$ and $\ell_{\min}(g)$ denotes the length of the shortest closed geodesic on $(W,g)$, is bounded from above on the space of all Riemannian metrics on $W$. The first investigations on this problem go back to Loewner, who in a course given at Syracuse University in 1949 proved that the systolic ratio of the two-torus is maximized by the flat torus that is obtained as the quotient of $\R^2$ by a lattice generated by two sides of an equilateral triangle (see \cite[Section 7.2.1.1]{ber03} for two different proofs of Loewner's result). Shortly afterwards, Pu \cite{pu52} showed that the systolic ratio of the projective plane is maximized by the round metric. A very general result, still in the framework of non-simply-connected manifolds, for which one can obtain closed geodesics by minimizing the length of non-contractible closed curves, was obtained by Gromov \cite{gro83}: The systolic ratio of any essential manifold is bounded from above by a constant depending only on the dimension. Here, a closed manifold $W$ is called essential if its fundamental class is non-zero in the  Eilenberg--MacLane space $K(\pi_1(W), 1)$ of its fundamental group.

The first result about simply connected manifolds is due to Croke \cite{cro88}, who showed that the systolic ratio of the two-sphere is bounded from above. Interestingly, the round metric does not maximize $\rho_{\mathrm{sys}}(S^2,\cdot)$, whose supremum is currently  unknown, but it is a local maximizer. More generally, all Zoll metrics on $S^2$, i.e.\ metrics all of whose geodesics are closed and have the same length, are local maximizers of the systolic ratio (see \cite{abhs17} for the local maximality of Zoll metrics among suitably pinched metrics on $S^2$ and \cite{abhs18} for the case of an arbitrary Zoll metric on $S^2$). The question whether the systolic ratio of a simply connected manifold of dimension at least three is bounded from above is open, even for spheres. Equally open is the boundedness of the systolic ratio of non-simply-connected non-essential manifolds, such as for instance $S^2 \times S^1$: The minimal length of a non-contractible closed curve can be arbitrarily large on any non-essential manifold of unit volume, see \cite{bab93,bru08}, but this does not exclude the existence of short contractible closed geodesics.

Consider now a Finsler metric on the closed $n$-dimensional manifold $W$, i.e.\ a positively 1-homogeneous function $F: TW \rightarrow [0,+\infty)$ that is smooth and positive outside of the zero section and such that the second fiberwise differential of $F^2$ is positive definite outside of the zero section. The systolic ratio of $(W,F)$ is the quantity
\[
\rho_{\mathrm{sys}}(W,F) := \frac{\ell_{\min}(F)^n}{\mathrm{vol}(W,F)},
\]
where $\ell_{\min}(F)$ denotes the length of the shortest closed geodesic on $(W,F)$ and $\mathrm{vol}(W,F)$ is the Holmes--Thompson volume of $(W,F)$, which we normalize so that it coincides with the usual Riemannian volume when $F=\sqrt{g}$ is Riemannian. Several other notions of volume, such as the Busemann--Hausdorff volume, can be defined on a Finsler manifold, which yield corresponding systolic ratios and reduce to the Riemannian volume when $F=\sqrt{g}$, see e.g.\ \cite{apt04}.  As we will see, the Holmes--Thompson volume is the natural one when generalizing to Reeb flows.

Both Gromov's and Croke's results about the boundedness of the systolic ratio in the Riemannian setting extend to the Finsler setting. Indeed, bounds on the Riemannian systolic ratio imply bounds on the Finsler one by a combined use of Loewner ellipsoids and the Rogers--Shephard inequality in convex geometry, see \cite{apbt16}. 

\subsection*{Contact systolic geometry}

In \cite{apb14}, \'Alvarez-Paiva and Balacheff proposed to extend questions from metric systolic geometry to the broader setting of contact geometry and Reeb dynamics, in which one can take advantage of a larger symmetry group. We recall that a co-oriented contact structure $\xi$ on the closed $(2n-1)$-dimensional manifold $M$ is a maximally non-integrable, co-oriented hyperplane distribution $\xi\subset TM$. We call any one-form $\alpha$ on $M$ such that $\xi=\ker\alpha$ a contact form supported by the contact structure $\xi$. In this case, the top-degree form $\alpha\wedge \di\alpha^{n-1}$ is nowhere vanishing. Therefore, $\alpha\wedge \di\alpha^{n-1}$ is a volume form on $M$, and the volume of $M$ with respect to it is denoted by
\[
\mathrm{vol}(M,\alpha) := \int_M \alpha\wedge \di \alpha^{n-1}.
\]
Moreover, the contact form $\alpha$ induces the Reeb vector field $R_{\alpha}$ on $M$, which is defined by the conditions
\[
\imath_{R_{\alpha}} \di\alpha = 0, \qquad \imath_{R_{\alpha}} \alpha = 1.
\]
It is then natural to define the systolic ratio of $(M,\alpha)$ as
\[
\rho_{\mathrm{sys}}(M,\alpha) :=  \frac{T_{\min}(\alpha)^n}{\mathrm{vol}(M,\alpha)}\in(0,+\infty],
\]
where $T_{\min}(\alpha)$ denotes the minimum of the periods of all closed orbits of $R_{\alpha}$. Here, $T_{\min}(\alpha)$ is defined to be $+\infty$ if $R_{\alpha}$ does not have any closed orbit. Note, however, that the Weinstein conjecture, which has been confirmed for many contact manifolds, asserts that any Reeb vector field on a closed manifold has closed orbits, so $\rho_{\mathrm{sys}}(M,\alpha)$ is expected to be always a finite number.

An important source of examples is given by starshaped hypersurfaces in the cotangent bundle $T^*W$ of any closed $n$-dimensional manifold $W$. Here, a hypersurface $M\subset T^*W$ is said to be starshaped if every ray in each cotangent fiber emanating from the origin meets $M$ transversally at exactly one point, and we take as contact form on $M$ the restriction of the Liouville form $p\, \di q$ of the cotangent bundle $T^*W$. If such a hypersurface is fiberwise strictly convex, then it can be seen as the unit cotangent sphere bundle $S^*_FW$ of a Finsler metric on $W$. Moreover, the Reeb flow of the associated contact form $\alpha_F$ is precisely the geodesic flow of $F$. In particular, $T_{\min}(\alpha_F)$ coincides with $\ell_{\min}(F)$ and the two volumes are related by the identity
\[
\mathrm{vol}(S^*_F W,\alpha_F) = n!\, \omega_n \,\mathrm{vol}(W,F),
\]
where $\omega_n$ denotes the volume of the Euclidean $n$-ball. Therefore, the Finsler systolic ratio of $(W,F)$ coincides up to a multiplicative constant with the contact systolic ratio:
\[
\rho_{\mathrm{sys}}(S^*_FW,\alpha_F) = \frac{1}{n!\, \omega_n} \rho_{\mathrm{sys}}(W,F).
\]
While in the metric case one considers the systolic ratio on $W$ as a function of the metric $F$, in the contact case it is natural to study the systolic ratio on $(M,\xi)$ as a function of the contact form $\alpha$ supported by $\xi$. This is indeed an interesting problem, as the space of such contact forms is infinite dimensional, being parametrised by positive smooth functions $f$ on $M$ via $f\mapsto\alpha=f\alpha_*$, where $\alpha_*$ is a fixed contact form. At the same time, the dynamics of $R_{f \alpha_*}$ is highly dependent on the positive function $f$, and the class of Reeb flows of contact forms supported by a given contact structure is extremely rich: For instance, all Reeb flows on a starshaped hypersurface $M\subset T^*W$ can be seen as Reeb flows on the same contact manifold $(S^*W,\xi)$, where $S^*W$ denotes the abstract unit cotangent bundle of $W$.

In investigating the systolic ratio on the space of contact forms supporting $\xi$, we distinguish between \textit{global} and \textit{local} properties. As far as \textit{global} properties are concerned, \'Alvarez-Paiva and Balacheff asked whether the systolic ratio is bounded from above. This question was given a negative answer: Any closed contact manifold $(M,\xi)$ admits contact forms of arbitrarily large systolic ratio. This was first proven for the tight three-sphere in \cite{abhs18}, for arbitrary contact three-manifolds in \cite{abhs19}, and in full generality in \cite{sag21}. In particular, without the convexity assumption a starshaped hypersurface in $T^* W$ can have an arbitrarily high systolic ratio, for every closed manifold $W$.

As far as \textit{local} properties are concerned, a special role is played by Zoll contact forms, that is, contact forms such that all Reeb orbits are closed and have the same minimal period. \'Alvarez-Paiva and Balacheff showed that if $\alpha$ is a critical point of $\rho_{\mathrm{sys}}$, then it is Zoll. Indeed, if the Reeb flow of a contact form $\alpha$ has an orbit that does not close up within the minimal period $T_{\min}(\alpha)$, then all nearby orbits do not close up before $T_{\min}(\alpha)$, and one can modify $\alpha$ near this orbit and change the volume at first order while keeping $T_{\min}(\alpha)$ constant. See \cite[Theorem 3.4]{apb14} for more details.

Zoll contact forms were introduced by Reeb in \cite{ree52} under the name of ``fibered dynamical systems with an integral invariant'' and are also called ``regular'' in the subsequent literature, but we prefer the term ``Zoll'', which we borrow from metric geometry: As recalled above, Zoll metrics are those Riemannian or Finsler metrics all of whose geodesics are closed and have the same length. 

Zoll contact forms have an easy description that is due to Boothby and Wang \cite{bw58} (see also \cite[Section 7.2]{gei08}): If $\alpha$ is a Zoll contact form on $M$ and $T$ is the common period of all its Reeb orbits, then the quotient of $M$ by the free $S^1$-action given by the Reeb flow is a symplectic manifold $(B,\omega)$, and the pull-back of $\omega$ by the projection map is $(1/T)\di\alpha$. Moreover, the cohomology class $[\omega]$ is integral and is the Euler class of the circle bundle $M\rightarrow B$. It follows that the systolic ratio of a Zoll contact form $\alpha$ is the inverse of a positive integer:
\[
\rho_{\mathrm{sys}}(M,\alpha) = \frac{1}{N},
\]
where $N=\langle [\omega]^{n-1},[B]\rangle$ is the Euler number of the circle bundle $M\rightarrow B$. For instance, the standard contact form on $S^{2n-1}$ is Zoll with common period $\pi$ and systolic ratio 1, and the corresponding circle bundle is the Hopf fibration $S^{2n-1} \rightarrow \CP^{n-1}$. Actually, the Hopf fibration gives a universal model for all Zoll contact forms: The restriction of it to the inverse image of any closed symplectic submanifold of $\CP^{n-1}$ defines a Zoll contact form, and any Zoll contact form with common period $\pi$ can be produced in this way, by choosing $n$ large enough (see \cite[Theorem 3.2]{apb14} and references therein).

\subsection*{The main results}

Knowing that critical points of the systolic ratio are Zoll contact forms it is natural to wonder if the converse is also true, and if so, what is the local behavior of the systolic ratio in a neighborhood of a Zoll contact form. The main result of \cite{apb14} goes in this direction: It says that if $\alpha_t$ is a one-parameter deformation of the Zoll contact form $\alpha_0$, then either the function $t\mapsto \rho_{\mathrm{sys}}(\alpha_t)$ has a local maximum at $t=0$, or $\alpha_t$ is tangent up to infinite order to the space of Zoll contact forms at $t=0$. See [Theorem 2.9]\cite{apb14} for the precise statement.

Therefore, we are led to ask: Are Zoll contact forms local maximizers of the systolic ratio, with respect to some reasonable topology on the space of contact forms? The aim of this paper is to give an affirmative answer to this question.
 
 \begin{mainthm}[Local systolic maximality of Zoll contact forms] 
 \label{main1}
 Let $\alpha_0$ be a Zoll contact form on a closed manifold $M$. For all $C>0$ there exists $\delta_C>0$ such that, if we define the $C^3$-neighborhood $\mathscr N_C$ of $\alpha$ by
\[
 \mathscr N_C:=\Big\{\alpha\text{ contact form on }M\ \Big|\ \Vert\alpha-\alpha_0\Vert_{C^2}<\delta_C,\ \Vert\alpha-\alpha_0\Vert_{C^3}<C \Big\},
 \]
 then there holds
 \[
 \rho_{\mathrm{sys}}(\alpha) \leq  \rho_{\mathrm{sys}}(\alpha_0) \qquad \forall \alpha\in  \mathscr{N}_C,
 \]
with equality if and only if $\alpha$ is Zoll. In the case of equality, there is a diffeomorphism $u: M \rightarrow M$ such that $u^* \alpha = \frac{T}{T_0} \, \alpha_0$, where $T$ and $T_0$ denote the period of the orbits of $R_{\alpha}$ and $R_{\alpha_0}$.
 \end{mainthm}
 
The local systolic maximality of Zoll contact forms in the $C^3$-topology is already known in dimension three: It was first proven for $M=S^3$ in \cite{abhs18} and then for arbitrary three-manifolds in \cite{bk21} (see also \cite{bk20} for a generalization to odd symplectic forms on three-manifolds and \cite{bk19c} for an application to magnetic flows on surfaces). The proofs in \cite{abhs18} and \cite{bk21} build on the fact that a closed orbit with minimal period of a contact form that is close to a Zoll one is the boundary of a global surface of section for the Reeb flow, provided that the manifold has dimension three. Global surfaces of section bounded by closed orbits are peculiar to three-manifolds, and we do not see a way of applying this approach to the higher dimensional case. 

The proof of Theorem \ref{main1} will be based instead on a normal form for contact forms close to Zoll ones. More precisely, it will use the following theorem, that is the second main result of this paper.

 \begin{mainthm}[Normal Form]
 \label{main2}
Let $\alpha_0$ be a Zoll contact form on a closed manifold $M$. There is $\delta_0>0$ such that if $\alpha$ is a contact form on $M$ with $\Vert \alpha-\alpha_0\Vert_{C^2}<\delta_0$, then there exists a diffeomorphism $u: M \rightarrow M$ such that
\[
u^*\alpha = S \alpha_0 + \eta + \di f,
\]
where:
\begin{enumerate}[(i)]
\item $S$ is a smooth positive function on $M$ that is invariant under the Reeb flow of $\alpha_0$;
\item $f$ is a smooth function on $M$ with average zero along each orbit of $R_{\alpha_0}$;
\item $\eta$ is a smooth one-form on $M$ satisfying $\imath_{R_{\alpha_0}} \eta=0$;
\item $\imath_{R_{\alpha_0}} \di\eta= \mathscr{F} [\di S]$ for a smooth endomorphism $\mathscr{F}:T^*M \rightarrow T^*M$ lifting the identity;
\item $\imath_{R_{\alpha_0}} \di f = \imath_Z\, \di S$ for a smooth vector field $Z$ on $M$ taking values in the contact distribution $\ker \alpha_0$ and having average zero along each orbit of $R_{\alpha_0}$.
\end{enumerate}
Moreover, for every integer $k\geq 0$ there is a monotonically increasing continuous function $\omega_k:[0,\infty)\to[0,\infty)$ with $\omega_k(0)=0$, such that
\[
\begin{split}
\max\Big\{\mathrm{dist}_{C^{k+1}} (u,\mathrm{id}),\  \|S-1\|_{C^{k+1}},\ \|f\|_{C^{k+1}},\ \|\eta\|_{C^k},\ \|\di\eta\|_{C^{k}},\ \|\mathscr{F}\|_{C^k}, \|Z\|_{C^k} \Big\}\\ \leq \omega_k\big(\|\alpha-\alpha_0\|_{C^{k+2}}\big).
\end{split}
\]
\end{mainthm}

The averages of a real function $f$ or a vector field $Z$ on $M$ along the orbits of $R_{\alpha_0}$ which are mentioned in (ii) and (v) are defined as
\[
\overline{f}(x) := \frac{1}{T_0} \int_0^{T_0} f\bigl( \phi_{\alpha_0}^t(x) \bigr)\, \di t, \quad
\overline{Z}(x) := \frac{1}{T_0} \int_0^{T_0} \di \phi_{\alpha_0}^{-t}(x) \bigl[ Z(\phi_{\alpha_0}^t(x) \bigr]\, \di t, \qquad \forall x\in M,
\]
where $\phi_{\alpha_0}^t$ denotes the flow of $R_{\alpha_0}$ and $T_0$ is the period of its orbits.

The proof of Theorem \ref{main2} is based on a normal form for vector fields close to vector fields inducing a free $S^1$-action that is due to Bottkol \cite{bot80}, which we include, in the form that is needed here, as Theorem \ref{tbot}. In Appendix \ref{appbottkol}, we exhibit a proof of Bottkol's theorem following an idea we learned in \cite[Proposition 3.4]{ker99}.

Note that any one-form $\beta$ can be decomposed as
\[
\beta = S  \alpha_0 + \eta + \di f,
\]
with $S$, $\eta$ and $f$ as in (i), (ii) and (iii): Define $S(x)$ to be the integral of $\beta$ on the closed orbit of $R_{\alpha_0}$ through $x$, so that the one-form $\beta - S \alpha_0$ has zero integral on every orbit of $R_{\alpha_0}$ and hence differs from a one-form vanishing on $R_{\alpha_0}$ by the differential of a function with zero average on the orbits of $R_{\alpha_0}$ (see Lemma \ref{splitting}). Therefore, the relevant statements in the above Theorem are (iv) and (v), which establish a further relationship between the forms appearing in the  above splitting. Statement (iv) will be crucial in this paper, whereas knowing that also (v) holds will simplify the proof of Proposition \ref{rigid-zoll} below.

Being invariant under the flow of $R_{\alpha_0}$, the function $S$ descends to a smooth function
\[
\widehat{S}: B \rightarrow \R
\]
on the quotient $B$ of $M$ by the free $S^1$-action defined by this flow. Condition (iv) implies that the function $\widehat{S}$ is a variational principle for detecting closed orbits of $R_\alpha$ of short period, that is, those closed orbits that bifurcate from the $(2n-2)$-dimensional manifold $B$ of closed orbits of $R_{\alpha_0}$. Indeed, we have the following result (see Section \ref{varprinsec}).

\begin{mainprop}[Variational principle]
\label{mainprop1}
Let $\alpha_0$ be a Zoll contact form on a closed manifold $M$ and let $\pi: M \rightarrow B$ be the corresponding $S^1$-bundle. Let $\beta$ be a contact form on $M$ of the form
\[
\beta = S \alpha_0 + \eta + \di f,
\]
where $S$ and $\eta$ satisfy the conditions (i), (iii), (iv) of Theorem \ref{main2} and $f$ is any smooth function on $M$. Denote by $\widehat{S}:B \rightarrow \R$ the function that is defined by $S=\widehat{S}\circ \pi$. Then for every critical point $b$ of $\widehat{S}$ the circle $\pi^{-1}(b)$ is a closed orbit of $R_{\beta}$ of period $\widehat{S}(b) T_{\min}(\alpha_0)$. Moreover, $\beta$ is Zoll if and only if the function $S$ - or equivalently the function $\widehat{S}$ - is constant.
\end{mainprop}

Theorem \ref{main2} and Proposition \ref{mainprop1} immediately imply that any contact form $\alpha$ that is $C^2$-close to the Zoll contact form $\alpha_0$ has at least as many closed orbits as the minimal number of critical points of a smooth function on $B$. Indeed, the image by the diffeomorphism $u$ of Theorem \ref{main2} of the circles $\pi^{-1}(b)$ corresponding to critical points $b\in B$ of $\widehat{S}$ are closed orbits of $R_{\alpha}$.

For instance, if $\alpha$ is a contact form on $S^{2n-1}$ that is $C^2$-close to the standard Zoll contact form whose Reeb trajectories define the Hopf fibration $S^{2n-1}\to \C\mathbb P^{n-1}$, then $R_\alpha$ has at least $n$ closed orbits of period close to $\pi$. Proving this and more general multiplicity results for closed orbits bifurcating from manifolds of closed orbits was Bottkol's original motivation for his normal form. See also \cite{wei73b}, \cite{wei77b}, \cite{mos76}, \cite{gin87}, \cite{gin90}, \cite{ban90} and \cite{br94} for other approaches to this question. 

Besides for producing a finite dimensional variational principle, the power of the normal form appearing in Theorem \ref{main2} lies in the fact that it yields the following useful formula for the volume.

\begin{mainprop}[Volume formula]\label{mainprop2}
Assume that $\alpha_0$ is a Zoll contact form on a $(2n-1)$-dimensional closed manifold $M$ and let $\beta$ be a one-form on $M$ of the kind 
\[
\beta = S \alpha_0 + \eta+\di f,
\]
where $S$ and $f$ are smooth functions on $M$ and $\eta$ is a one-form satisfying
\[
\imath_{R_{\alpha_0}} \eta = 0, \qquad \imath_{R_{\alpha_0}} \di \eta = \mathscr{F} [\di S],
\]
for some endomorphism $\mathscr{F}: T^* M \rightarrow T^*M$ lifting the identity. Then 
\[
\int_M \beta \wedge \di \beta^{n-1} = \int_M p(x,S(x)) \, \alpha_0 \wedge \di\alpha_0^{n-1},
\]
where $p: M \times \R \rightarrow \R$ is a smooth function of the form
\[
p(x,s) = s^n + \sum_{j=1}^{n-1} p_j(x) s^j,
\]
whose coefficients $p_j$ are smooth functions on $M$  satisfying
\[
\int_M p_j \, \alpha_0 \wedge \di\alpha_0^{n-1} = 0 \qquad  \forall j=1,\dots,n-1.
\]
Moreover, for every $c>0$ and $\epsilon>0$ there exists $\delta>0$ such that if
\[
\max\{\|\eta\|_{C^0}, \|\di\eta\|_{C^0}, \|\mathscr{F}\|_{C^0}\} < \delta, \qquad \max\{ \|\eta\|_{C^1}, \|\di\eta\|_{C^1}, \|\mathscr{F}\|_{C^1}\} < c,
\]
then $\|p_j\|_{C^0} < \epsilon$ for every $j=1,\dots,n-1$.
\end{mainprop}

It is now easy to see how Theorem \ref{main2}, Proposition \ref{mainprop1} and Proposition \ref{mainprop2} lead to the proof of the sharp systolic inequality of Theorem \ref{main1}. Indeed, for every $C>0$ we can find a positive number $\delta_C$ such that if $\alpha$ belongs to the neighborhood $\mathscr{N}_C$ defined in Theorem \ref{main1}, then $\alpha$ can be put in the normal form $\beta=u^*\alpha$ of Theorem \ref{main2} by a diffeomorphism $u$, and furthermore the function $s\mapsto p(x,s)$ of Proposition \ref{mainprop2} is strictly increasing on the interval $[\min S,\max S]$ for every $x\in M$. This fact, together with the fact that the principal coefficient of the polynomial map $p$ is 1 and all the other coefficients have vanishing integral, implies the estimate
\[
\begin{split}
\mathrm{vol}(M,\alpha) &= \mathrm{vol}(M,u^*\alpha) = \int_M p(x,S(x))\, \alpha_0\wedge \di\alpha_0^{n-1}\\ & \geq \int_M p(x,\min S)\, \alpha_0\wedge \di\alpha_0^{n-1}
= (\min S)^n\mathrm{vol}(M,\alpha_0).
\end{split}
\]
By Proposition \ref{mainprop1}, the Reeb vector field of $\alpha$ has a closed orbit of period $(\min S) T_{\min}(\alpha_0)$, and hence 
\[
T_{\min}(\alpha) \leq (\min S) T_{\min}(\alpha_0).
\]
The above two inequalities imply the desired sharp systolic bound
\[
\frac{T_{\min}(\alpha)^n}{\mathrm{vol}(M,\alpha)} \leq \frac{T_{\min}(\alpha_0)^n}{\mathrm{vol}(M,\alpha_0)}.
\]
The fact that $\alpha$ is Zoll if and only if the function $S$ is constant, see again Proposition \ref{mainprop1}, implies that the equality holds in the above estimate if and only if $\alpha$ is Zoll. In the latter case, the fact that $\alpha$ is strictly contactomorphic to $\alpha_0$ up to a multiplicative constant follows from Moser's homotopy argument, see Proposition \ref{rigid-zoll} below.

We refer to Section \ref{systolicsec} for a detailed proof. At the end of that section we also discuss a lower bound for the maximal period of ``short'' closed orbits. 

\subsection*{Three applications of Theorem \ref{main1}}

We conclude this introduction with three corollaries of the local systolic maximality of Zoll contact forms.

\paragraph{Finsler geodesic flows.} The first corollary is immediate and consists in applying Theorem \ref{main1} to the contact form $\alpha_F$ on $S_F^*W$ that is induced by a Finsler metric $F$ on $W$.

\begin{maincor}
\label{cor1}
Let $F_0$ be a Zoll Finsler metric on the closed manifold $W$. Then $F_0$ has a $C^3$-neighborhood $\mathscr{U}$ in the space of all Finsler metrics on $W$ such that
\[
\rho_{\mathrm{sys}}(W,F) \leq \rho_{\mathrm{sys}}(W,F_0) \qquad \qquad \forall F\in \mathscr{U},
\]
with equality if and only if $F$ is Zoll. In the equality case, the geodesic flow of $F$ is smoothly conjugated to that of $F_0$, up to a linear time reparametrization.
\hfill\qed
\end{maincor}

In the above result, the $C^3$-topology on the space of smooth Finsler metrics on the closed manifold $W$ is defined by restricting these functions $F: T W \rightarrow \R$ to the unit tangent sphere bundle of some fixed Riemannian metric on $W$ (the resulting topology is eventually independent of the choice of this metric).

In dimension two, this theorem follows from known results: The only surfaces admitting Zoll Finsler metrics are $S^2$, for which this result was proven in the already mentioned articles \cite{abhs17} and \cite{abhs18}, and $\RP^2$, for which the result immediately follows by lifting the metric to $S^2$. Actually, reversible Zoll Finsler metrics on $\RP^2$ are global maximizers of the systolic ratio among reversible Finsler metrics, as proven by Ivanov in \cite{iva11}. In higher dimensions, the local sharp systolic inequality of Corollary \ref{cor1} appears to be a new result, even for Riemannian perturbations of simple rank-one symmetric spaces, such as the round $S^n$ or the round $\RP^n$. In particular, this corollary gives a positive answer to the local version of Question 5.3 in Berger's survey paper \cite{ber70}.

\paragraph{Symplectic capacity of convex domains.} Our next corollary concerns the behavior of symplectic capacities on convex domains in $\R^{2n}$. Recall that a (normalized) symplectic capacity on the vector space $\R^{2n}$, endowed with its standard symplectic structure $\omega_0$, is a function $c: \{ \mbox{open subsets of } \R^{2n} \} \rightarrow [0,+\infty]$ that satisfies the following conditions:
\begin{enumerate}[(c1)]
\item Monotonicity: $c(A_1)\leq c(A_2)$ if $A_1\subset A_2$.
 \item Symplectic invariance: $c(\varphi(A)) = c(A)$ if $\varphi: A \hookrightarrow \R^{2n}$ is a symplectomorphism.
 \item Homogeneity: $c(\lambda A) = \lambda^2 c(A)$ for all $\lambda>0$.
 \item Normalization: $c(B^{2n})=c(Z)=\pi$, where $B^{2n}$ is the unit ball in $\R^{2n}$ and $Z$ is the cylinder $B^2\times \R^{2n-2}$.
 \end{enumerate}
 By definition, every symplectic capacity $c$ satisfies
 \begin{equation}
 \label{sandwich}
 c_{\mathrm{in}} \leq c \leq c_{\mathrm{out}},
 \end{equation}
 where  $c_{\mathrm{in}}$ and $c_{\mathrm{out}}$ are the functions
 \[
 \begin{split}
 c_{\mathrm{in}} (A) &:= \sup\{ \pi r^2 \mid \mbox{there exists a symplectic embedding of } r B^{2n} \mbox{ into }A \}, \\
 c_{\mathrm{out}} (A) &:= \inf \{ \pi r^2 \mid \mbox{there exists a symplectic embedding of } A  \mbox{ into } r Z \}.
\end{split}
\]
By Gromov's non squeezing theorem, $c_{\mathrm{in}}$ and $c_{\mathrm{out}}$ are themselves capacities; $c_{\mathrm{in}}$ is called Gromov width, whereas $c_{\mathrm{out}}$ is called cylindrical capacity.
 
Many other non-equivalent symplectic capacities have been constructed in this and more general settings, but for convex domains many of them have been shown to coincide: This is the case for the first of the Ekeland--Hofer capacities (see \cite{eh89}), for the Hofer--Zehnder capacity (see \cite{hz90}), for the Viterbo capacity (see \cite{her04}) and for the capacity coming from symplectic homology (see \cite{ak19} and \cite{iri19}). Following a common usage, we shall refer to the common value of these capacities on convex domains as Ekeland--Hofer--Zehnder capacity and denote it by $c_{\mathrm{EHZ}}$. This capacity is related to contact systolic geometry.  Indeed, when $C\subset \R^{2n}$ is a convex bounded open set containing the origin and having a smooth boundary, then
\begin{equation}
\label{c=T}
c_{\mathrm{EHZ}}(C) = T_{\min}(\alpha_C),
\end{equation}
the minimal period of closed orbits on $\partial C$ with respect to the Reeb flow induced by the contact form $\alpha_C := \lambda_0|_{\partial C}$, where $\lambda_0$ is the homogeneous primitive of $\omega_0$, that is the one-form
\[
\lambda_0 := \frac{1}{2} \sum_{j=1}^n (x_j \di y_j - y_j \di x_j).
\]

In \cite{vit00}, Viterbo formulated a challenging conjecture relating symplectic capacities and volume: If $c$ is any symplectic capacity and $C\subset \R^{2n}$ is a non-empty convex bounded open set, then
\begin{equation}
\label{viterbo}
c(C)^n \leq \mathrm{vol}(C,\omega_0^n),
\end{equation}
with equality if and only if $C$ is symplectomorphic to a ball. Note that $\mathrm{vol}(C,\omega_0^n)$ is $n!$ times the Euclidean volume of $C$. Note also that (\ref{viterbo}) is trivially true for the Gromov width $c_{\mathrm{in}}$. Proving that all symplectic capacities agree on convex domains - a long standing open question - would then imply (\ref{viterbo}) in general.

The bound (\ref{viterbo}) has been shown to be asymptotically true, that is, valid up to a multiplicative constant that is independent of the dimension, in \cite{amo08}. Moreover, its validity in the sharp form for the Ekeland--Hofer--Zehnder capacity would imply the Mahler conjecture in convex geometry, see \cite{ako14}. 

Thanks to Theorem \ref{main1}, we can prove the sharp version of Viterbo's conjecture assuming the convex domain $C$ to be $C^3$-close to a ball (see Section \ref{viterbo-sec} below for the precise definition of $C^k$-closedness for convex domains with smooth boundary).

\begin{maincor}
\label{cor2}
There is a $C^3$-neighborhood $\mathscr{B}$ of the ball in the space of smooth convex bounded open subsets of $\R^{2n}$ such that every symplectic capacity $c$ satisfies
\[
c(C)^n \leq \mathrm{vol}(C,\omega_0^n) \qquad \forall C\in \mathscr{B},
\]
with equality if and only if $C$ is symplectomorphic to a ball.
\end{maincor}

For $n=2$ and $c=c_{\mathrm{EHZ}}$, this is proven in \cite{abhs18}. For general $n$ and $c$, we shall deduce the above corollary from Theorem \ref{main1} and the following two results.

\begin{mainprop}
\label{EHZ=cyl}
There is a $C^3$-neighborhood $\mathscr{B}$ of the ball  $B^{2n}$ in the space of smooth convex bounded open subsets of $\R^{2n}$ on which the Ekeland--Hofer--Zehnder capacity and the cylindrical capacity coincide:
\[
c_{\mathrm{EHZ}}(C) = c_{\mathrm{out}}(C) \qquad \forall C\in \mathscr{B}.
\]
\end{mainprop}

\begin{mainprop}
\label{symplectic}
There is a $C^3$-neighborhood $\mathscr{B}$ of the unit ball $B^{2n}$ in the space of smooth convex bounded open subsets of $\R^{2n}$ such that if $C$ belongs to $\mathscr{B}$ and $\alpha_C = \lambda_0|_{\partial C}$ is Zoll, then there exists a symplectomorphism of $(\R^{2n},\omega_0)$ mapping a ball onto $C$.
\end{mainprop}

In \cite[Proposition 4.3]{abhs18}, the result of Proposition \ref{symplectic} is proven for $n=2$ in full generality for all starshaped domains $C$. In higher dimension, many of the ingredients of that proof break down and we do not know if the result holds true for all starshaped domains, but we are able to recover it for domains that are $C^3$-close to the ball by a combined use of Moser's homotopy argument and generating functions. The proof of Proposition \ref{EHZ=cyl} also uses generating functions.

These propositions are proven in Section \ref{viterbo-sec} below. Here we show how they can be used to deduce Corollary \ref{cor2} from Theorem \ref{main1}. The contact form 
$\alpha_{B^{2n}} = \lambda_0|_{\partial B^{2n}}$
is Zoll on the sphere $S^{2n-1}= \partial B^{2n}$ with orbits of period $\pi$, contact volume $\pi^n$ and hence systolic ratio 1. The radial projection $S^{2n-1} \rightarrow \partial C$ pulls back the contact form $\alpha_C$ to a contact form $\widetilde{\alpha}_C$ on $S^{2n-1}$ which is $C^k$-close to $\alpha_{B^{2n}}$ when $C$ is $C^k$-close to $B^{2n}$ (see Section \ref{viterbo-sec} below for more about this). We choose the $C^3$-neighborhood $\mathscr{B}$ of $B^{2n}$ in such a way that $\widetilde{\alpha}_C$ belongs to the $C^3$-neighborhood $\mathscr{N}_1$ of the Zoll contact form $\alpha_{B^{2n}}$ from Theorem \ref{main1} and the conclusions of Propositions \ref{EHZ=cyl} and \ref{symplectic} hold for every $C\in \mathscr{B}$. Thanks to the identities (\ref{c=T}) and
\[
\mathrm{vol}(S^{2n-1},\widetilde{\alpha}_C) = \mathrm{vol}(\partial C,\alpha_C) = \mathrm{vol}(C,\omega_0^n),
\]
the inequality
\begin{equation}
\label{dateo1}
T_{\min}(\widetilde{\alpha}_C)^n \leq \mathrm{vol}(S^{2n-1},\widetilde{\alpha}_C)
\end{equation}
which is ensured by Theorem \ref{main1} implies the bound
\[
c_{\mathrm{EHZ}}(C)^n \leq \mathrm{vol}(C,\omega_0^n) \qquad \forall C\in \mathscr{B}.
\]
If $c$ is an arbitrary symplectic capacity, then (\ref{sandwich}) and Proposition \ref{EHZ=cyl} imply 
\[
c(C)^n \leq c_{\mathrm{EHZ}}(C)^n \leq \mathrm{vol}(C,\omega_0^n) \qquad \forall C\in \mathscr{B}.
\]
If $c(C)^n=\mathrm{vol}(C,\omega_0^n)$, then also $c_{\mathrm{EHZ}}(C)^n$ coincides with $\mathrm{vol}(C,\omega_0^n)$, so (\ref{dateo1}) is an equality and by Theorem \ref{main1} the contact form $\alpha_C$ is Zoll. Then Proposition \ref{symplectic} implies that $C$ is symplectomorphic to a ball.

\paragraph{Symplectic non-squeezing in the intermediate dimensions.}
Our last corollary concerns a local generalization to intermediate dimensions of Gromov's non-squeezing theorem \cite{gro85}. Recall that this theorem can be stated in the following way: If $P_V$ is the symplectic linear projection onto a symplectic two-dimensional subspace $V\subset \R^{2n}$ (i.e.\ linear projection along the symplectic orthogonal) and  $\varphi: \R^{2n} \rightarrow \R^{2n}$ is  a symplectomorphism, then
\[
\mathrm{area}(P_V(\varphi(B^{2n})),\omega_0|_V) \geq \pi.
\]
In other words, the two-dimensional shadow of a symplectic ball has a large area, see \cite{eg91}. In \cite{am13} it was shown that higher dimensional shadows of symplectic balls can have arbitrarily small volume: If $P_V$ is the symplectic linear projection onto a symplectic $2k$-dimensional subspace $V\subset \R^{2n}$ with $1<k<n$ and $\epsilon$ is any positive number, then there exists a symplectomorphism $\varphi: \R^{2n} \rightarrow \R^{2n}$ such that
\[
\mathrm{vol}(P_V(\varphi(B^{2n})),\omega_0^k|_V) < \epsilon.
\]
On the other hand, if $\Phi :\R^{2n} \rightarrow \R^{2n}$ is a \textit{linear} symplectomorphism, then the volume of the shadow of the image of the ball $B^{2n}$ by $\Phi$ is given by the identity
\[
\mathrm{vol}(P_V \Phi (B^{2n}),\omega_0^k|_V) = \frac{\pi^k}{w(\Phi^{-1}(V))},
\]
where the function $w$ associates to any $2k$-dimensional real subspace $W\subset \R^{2n}\cong \C^n$ the number
\[
w(W) := \frac{|\omega_0^k[w_1,\dots,w_{2k}]|} {k!\, |w_1 \wedge \dots \wedge w_k|}, \qquad \mbox{with } w_1,\dots,w_{2k} \mbox{ a basis of } W.
\]
By the Wirtinger inequality, $w(W)\leq 1$ and $w(W)= 1$ if and only if $W$ is a complex subspace, so the above identity implies the sharp inequality
\[
\mathrm{vol}(P_V \Phi (B^{2n}),\omega_0^k|_V) \geq \pi^k,
\]
for the linear symplectomorphism $\Phi$ and tells us that equality holds if and only if $\Phi^{-1}(V)$ is a complex subspace. See \cite{am13} and Theorem \ref{linnonsque} below.

In \cite{am13}, some evidence to the conjecture that the above sharp inequality should hold also for nonlinear symplectomorphisms that are close enough to linear ones was given. Thanks to Theorem \ref{main1}, we can confirm this conjecture for $C^3$-closeness.

\begin{maincor}
\label{cor3}
There is a $C^3_{\mathrm{loc}}$-neighborhood $\mathscr{W}$ of the set of linear symplectomorphisms in the space of all smooth symplectomorphisms of $\R^{2n}$ such that the following holds:  If $1\leq k \leq n$ and $P_V$ is the symplectic linear projection onto a symplectic $2k$-dimensional subspace $V\subset \R^{2n}$ then
\[
\mathrm{vol}(P_V(\varphi(B^{2n})),\omega_0^k|_V) \geq \pi^k
\]
for every $\varphi\in \mathscr{W}$.
\end{maincor}

For $k=2$, a slightly weaker version of this result was proven in \cite{abhs18} (there, the order of quantifiers is different, and the neighborhood $\mathscr{W}$ depends on the choice of the linear symplectic subspace $V$). In the analytic category, a related result for arbitrary $k$ is proven in \cite{rig19}. 

It is interesting to observe that, in contrast to the above result, other inequalities of a similar flavor are known to fail in the intermediate dimensions, even locally. For instance, Gromov studied the higher homological systoles of metrics on $\CP^n$ having the same volume as the Fubini-Study metric $g_0$ and showed that the 2-systole of $\CP^2$ is locally maximized by $g_0$ (and all its quasi-K\"ahler deformations), whereas for $2\leq k \leq n-1$ there are metrics on $\CP^n$ that are arbitrary close to $g_0$ and have a strictly larger $2k$-systole. See \cite[Section 4]{gro96}.

Corollary \ref{cor3} is proven in Section \ref{shadowssec} below. Here we wish to remark that the validity of the Viterbo conjecture for the Ekeland--Hofer--Zehnder capacity would imply the conclusion of Corollary \ref{cor3} for all symplectomorphisms $\varphi$ such that $\varphi(B^{2n})$ is convex. Indeed, this follows from the fact that the Ekeland--Hofer--Zehnder capacity of the image of a convex domain $C\subset \R^{2n}$ with respect to the linear symplectic projection $P_V$ is not smaller than the Ekeland--Hofer--Zehnder capacity of $C$:
\[
c_{\mathrm{EHZ}}(P_V(C)) \geq c_{\mathrm{EHZ}}(C),
\]
where the capacity on the left-hand side is acting on subsets of the symplectic vector space $(V,\omega_0|_V)$. The above inequality follows from the characterization of the Ekeland--Hofer--Zehnder capacity via Clarke duality, see e.g.\ \cite[Theorem 4.1 (v)]{ama15}.

\paragraph{Acknowledgements.}
We would like to thank Viktor Ginzburg for making us aware of Bottkol's theorem and for several discussions about it. Our gratitude goes also to Barney Bramham, Umberto Hryniewicz, Jungsoo Kang and Pedro Salom\~ao, with whom we have discussed the topic of this paper so many times that their contribution goes surely beyond what we even realize. 

A.\ Abbondandolo is partially supported by the Deutsche Forschungsgemeinschaft under the Collaborative Research Center SFB/TRR 191 - 281071066 (Symplectic Structures in Geometry, Algebra and Dynamics). G.\ Benedetti is
partially supported by the Deutsche Forschungsgemeinschaft 
under Germany's Excellence Strategy
EXC2181/1 - 390900948 (the Heidelberg STRUCTURES Excellence Cluster), the Collaborative Research Center SFB/TRR 191 - 281071066 (Symplectic Structures in Geometry, Algebra and Dynamics), and the Research Training Group RTG 2229 - 281869850 (Asymptotic Invariants and Limits of Groups and Spaces).

\numberwithin{equation}{section}

\section{A few facts about differential forms}

In this section, we fix some notation and we discuss some results about differential forms that will be used in the proof of the normal form of Theorem \ref{main2}. 

We denote by $\Lambda^kM$ the vector bundle of alternating $k$-forms on the manifold $M$ and by $\Omega^k(M)$ the space of smooth sections of this bundle, i.e.\ differential $k$-forms on $M$. The vector bundle $\Lambda^1M$ is the cotangent bundle $T^*M$.

The $C^k$-norms of differential forms on $M$ are induced by the choice of some arbitrary but fixed Riemannian metric on $M$. When estimating such norms, we will use the symbol ``$\lesssim$'' to mean ``less or equal up to a multiplicative constant depending on $k$''.

Alternatively, bounds will be given in terms of moduli of continuity. By modulus of continuity we mean here a monotonically increasing continuous function 
\[
\omega:[0,+\infty) \rightarrow [0,+\infty)
\]
such that $\omega(0)=0$. Giving bounds in terms of moduli of continuity has the advantage that we can conclude the smallness of the output from the smallness of the input and the boundedness of the output from the boundedness of the input at the same time.

The first lemma allows us to bound the pullback of differential forms. Its proof is standard and is contained in Appendix \ref{appbound}.
\begin{lem}
	\label{boundsonforms}
	Let $M$ be a closed Riemannian manifold of dimension $d$. Then there exists a positive number $r>0$ such that for every smooth map $u: M \rightarrow M$ with the property that
	\begin{equation}
	\label{rvicini}
	\mathrm{dist}_{C^0}(u,\mathrm{id}) \leq r, 
	\end{equation}
	and for every $\alpha\in \Omega^j(M)$, $0\leq j \leq d$, the following bounds hold:
	\begin{equation}
	\label{bdof1}
	\|u^* \alpha\|_{C^k} \lesssim  \|\alpha\|_{C^k} \|\di u\|_{C^k}^j ( 1 + \|\di u\|_{C^{k-1}}^{k}), 
	\end{equation}
	\begin{equation}
	\label{bdof2}
	\|u^* \alpha - \alpha\|_{C^k} \lesssim \|\alpha\|_{C^{k+1}} \mathrm{dist}_{C^{k+1}}(u,\mathrm{id}) ( 1 + \|\di u\|_{C^k}^{k+j}),
	\end{equation}
	for every integer $k\geq 0$, where for $k=0$ the term $\|\di u\|_{C^{k-1}}$ in (\ref{bdof1}) is set to be zero.
\end{lem}

The second lemma allows us to bound the distance of two Reeb vector fields in terms of the corresponding contact forms. The proof is given in Appendix \ref{appbound}.

\begin{lem}
\label{l:estreeb}
Let $M$ be a closed manifold of dimension $2n-1$ with contact form $\alpha_0$. Then there exists $\delta>0$ and a sequence of moduli of continuity $\omega_k$ such that	
\begin{equation*}
	\Vert R_{\alpha}-R_{\alpha_0}\Vert_{C^k}\leq \omega_k\big(\Vert\alpha-\alpha_0\Vert_{C^{k+1}}\big)\qquad\forall\,k\geq0,
	\end{equation*}
for every contact form $\alpha$ on $M$ such that $\|\alpha-\alpha_0\|_{C^1}<\delta$.
\end{lem}

The last lemma of this section is a splitting result for one-forms on $M$ whose integrals over the Reeb orbits of a Zoll contact form vanish.

\begin{lem}
	\label{splitting}
	Let $\alpha_0$ be a Zoll contact form on $M$ with associated $S^1$-bundle $\pi: M \rightarrow B$.
	Let $\beta$ be a one-form on $M$ such that
	\[
	\int_{\pi^{-1}(b)} \beta = 0 \qquad \forall\, b\in B.
	\]
	Then $\beta$ splits uniquely as
	\begin{equation}
	\label{splid}
	\beta=\eta + \di f,
	\end{equation}
	where $\eta\in \Omega^1(M)$ satisfies $\imath_{R_{\alpha_0}} \eta=0$ and $f\in \Omega^0(M)$ has average zero along each orbit of $R_{\alpha_0}$. Moreover, for every integer $k\geq 0$ the following bounds hold:
	\[
	\|\eta\|_{C^k} \lesssim  \|\beta\|_{C^k} + \|\di\imath_{R_{\alpha_0}} \beta\|_{C^k}, \qquad \|f\|_{C^{k+1}} \lesssim \|\imath_{R_{\alpha_0}} \beta\|_{C^k} + \|\di \imath_{R_{\alpha_0}} \beta\|_{C^k}. 
	\]
\end{lem}

\begin{proof}
If 
\[
\beta = \eta + \di f = \eta'+ \di f'
\]
are two splittings as above, then 
\[
\imath_{R_{\alpha_0}} \di (f'-f) = \imath_{R_{\alpha_0}} (\eta-\eta') =0,
\]
so $f'-f$ is constant on each orbit of $R_{\alpha_0}$ and by the zero average assumption $f'-f=0$. This proves the uniqueness of the splitting.

We now prove its existence and the bounds on $\eta$ and $f$. By assumption, the function $h:= \imath_{R_{\alpha_0}} \beta$ has average zero along each orbit of $R_{\alpha_0}$. This implies the existence of a function $f\in \Omega^0(M)$ having average zero along every orbit of $R_{\alpha_0}$ and such that $\imath_{R_{\alpha_0}} \di f = h$. This claim can be proven in the following way. Let $\{\rho_j\}_{j=1,\dots,N}$ be a smooth partition of unity on $B$, where each $\rho_j$ is supported in an open set $B_j$ that is a trivializing domain for the $S^1$-bundle $\pi$. Then the function $h_j := (\rho_j\circ \pi) h$ is supported in $\pi^{-1}(B_j)$. We identify $\pi^{-1}(B_j)$ with $B_j \times \R/T_0 \Z$ in such a way that $R_{\alpha_0}$ is identified with the vector field $\partial_{\theta}$, $T_0$ denoting the minimal period of the orbits of $R_{\alpha_0}$ and $\theta$ being the variable in $\R/T_0 \Z$. Then the assumption on $h$ implies 
	\[
	\int_0^{T_0} h_j(b,\theta) \, d\theta = \rho_j(b) \int_0^{T_0} h(b,\theta) \, d\theta = 0 \qquad \forall\, b\in B_j.
	\]
	Now the formula
	\begin{equation}
	\label{primitive}
	f_j(b,\theta) := \int_0^{\theta} h_j(b,\vartheta)\, d\vartheta + \frac{1}{T_0} \int_0^{T_0} \vartheta\,h_j(b,\vartheta)\, d\vartheta, \qquad \forall (b,\theta)\in B_j \times \R/T_0 \Z,
	\end{equation}
	defines a smooth function $f_j$ on $M$ that is supported in $\pi^{-1}(B_j)$, has average zero on every orbit of $R_{\alpha_0}$ and satisfies
	\[
	\imath_{R_{\alpha_0}} \di f_j = \frac{\partial f_j}{\partial \theta} = h_j.
	\]
	Since the sum of the functions $h_j$ is $h$, we see that the function $f:= \sum_{j=1}^N f_j$, which has average zero on every orbit of $R_{\alpha_0}$, satisfies
	\[
	\imath_{R_{\alpha_0}} \di f = h =  \imath_{R_{\alpha_0}} \beta,
	\]
	proving our claim. As a consequence, the one-form $\eta:= \beta-\di f$ satisfies the desired relation
	\[
	\imath_{R_{\alpha_0}} \eta = \imath_{R_{\alpha_0}} \beta -  \imath_{R_{\alpha_0}} \di f = 0.
	\]
	To prove the bounds on $\eta$ and $f$ let $k\geq 0$ be an integer. By differentiating (\ref{primitive}) we get
	\[
	\|f_j\|_{C^{k+1}} \leq  \Bigl( 1 + \frac32T_0\Bigr)  \|h_j\|_{C^{k+1}}.
	\]
	Together with the identities
	\[
	\di h_j = \di(\rho_j\circ \pi) \imath_{R_{\alpha_0}} \beta +(\rho_j\circ \pi) \di \imath_{R_{\alpha_0}} \beta,
	\]
	we deduce the bound
	\[
	\|f\|_{C^{k+1}} \leq \sum_{j=1}^N \|f_j\|_{C^{k+1}} \lesssim \| \imath_{R_{\alpha_0}} \beta \|_{C^k} + \|\di \imath_{R_{\alpha_0}} \beta\|_{C^k}.
	\]
	By the definition of $\eta$ and the above bound, we have
	\[
	\|\eta\|_{C^k} = \|\beta-\di f\|_{C^k} \leq \|\beta\|_{C^k} + \|\di f\|_{C^k} \lesssim \|\beta\|_{C^k} + \| \imath_{R_{\alpha_0}} \beta \|_{C^k} + \|\di \imath_{R_{\alpha_0}} \beta\|_{C^k}.
	\]
	Since the $C^k$-norm of $\imath_{R_{\alpha_0}} \beta$ can be bounded by the $C^k$-norm of $\beta$, the above inequality implies the bound
	\[
	\|\eta\|_{C^k} \lesssim \|\beta\|_{C^k} + \|\di\imath_{R_{\alpha_0}} \beta\|_{C^k},
	\]
	which concludes the proof.
\end{proof}

\section{Normal form for contact forms close to a Zoll one}
\label{normalsec}

In \cite{bot80}, Bottkol constructed a normal form for vector fields $X$ on a manifold $M$ which are close to a vector field $X_0$ having a submanifold of periodic orbits with the same minimal period and satisfying a suitable non-degeneracy assumption. In the proof of Theorem \ref{main2}, we shall use the following version of Bottkol's theorem concerning the case in which  the manifold of periodic orbits of $X_0$ is the whole $M$.

\begin{thm}
	\label{tbot}
Let $M$ be a closed manifold and $X_0$ a vector field on $M$ all of whose orbits are periodic and with the same minimal period $T_0$. Then there exists $\delta>0$ such that for every vector field $X$ on $M$ with $\|X-X_0\|_{C^1} < \delta$ there is a diffeomorphism $u:M\to M$, a smooth vector field $V$ on $M$, a smooth function $h: M \rightarrow \R$, and a linear automorphism $\mathscr Q:TM\to TM$ lifting the identity such that:
	\begin{enumerate}[(a)]
		\item $h\,u^* X = X_0 - \mathscr{Q} [V]$;
		\item $\mathcal{L}_{X_0} V=0$;
		\item $g(V,X_0)=0$;
		\item $\mathcal{L}_{X_0} h=0$.
	
	\end{enumerate}
Moreover, for every $k\geq 0$, there is a modulus of continuity $\omega_k$ such that
\begin{equation}
\label{lestime}
\begin{split}
		\max \Big\{ \mathrm{dist}_{C^{k+1}}(u,\mathrm{id}),\  \|V\|_{C^{k+1}},\ & \|\mathscr{Q}-\mathrm{id}\|_{C^k},\ \mathrm{dist}_{C^{k+1}}(\di u\circ \mathscr{Q},\mathrm{id}),\ \|h-1\|_{C^{k+1}} \Big\}\\  &\leq \omega_k( \|X-X_0\|_{C^{k+1}}); 
	\end{split}
\end{equation}
where $\mathrm{dist}_{C^{k+1}}(\di u\circ \mathscr{Q},\mathrm{id})$ is calculated at points of the unit sphere bundle of $M$.
\end{thm}

Here, $\mathcal{L}_{X_0}$ denotes the Lie derivative along $X_0$ and $g$ is an arbitrary Riemannian metric on $M$ that is invariant under the $S^1$-action defined by $X_0$.
In Appendix \ref{appbottkol}, we give a complete proof of the above version of Bottkol's theorem and we discuss it further.

This section is devoted to the proof of the normal form for contact forms that are close to a Zoll one stated in Theorem \ref{main2} from the Introduction.

\begin{proof}[Proof of Theorem \ref{main2}] 
Let $\alpha_0$ be a Zoll contact form on $M$ with associated $S^1$-bundle denoted by $\pi:M\to B$. Let $\delta>0$ be the number obtained in Theorem \ref{tbot} taking $X_0=R_{\alpha_0}$. By Lemma \ref{l:estreeb}, there exists $\delta_0>0$ such that
\begin{equation}\label{epsilon1}
\|\alpha-\alpha_0\|_{C^2}<\delta_0\quad\Longrightarrow\quad \|R_\alpha-R_{\alpha_0}\|_{C^1}<\delta
\end{equation}
and we can apply Theorem \ref{tbot} to $X=R_\alpha$. We get a smooth diffeomorphism $u:M \rightarrow M$, a smooth vector field $V$ on $M$ satisfying 
\[
\mathcal{L}_{R_{\alpha_0}} V=0, \qquad g(V,R_{\alpha_0}) =0,
\]
a linear bundle morphism $\mathscr{Q}:TM \rightarrow TM$ lifting the identity and a smooth function $h: M \rightarrow \R$ satisfying $\mathcal{L}_{R_{\alpha_0}} h =0$ such that
\begin{equation}
\label{campi}
h\, u^*R_{\alpha} = R_{\alpha_0} - \mathscr{Q} [V].
\end{equation}
By choosing the $S^1$-invariant metric $g$ so that $R_{\alpha_0}$ is orthogonal to the contact distribution $\ker \alpha_0$, we obtain that $V$ takes values in $\ker \alpha_0$. Thanks to (\ref{lestime}) and Lemma \ref{l:estreeb}, $u$, $V$, $\mathscr{Q}$ and $h$ satisfy the bounds
\begin{equation}
\label{omega1}
\begin{split}
\max \{ \mathrm{dist}_{C^{k+1}}(u,\mathrm{id}), \ \|V\|_{C^{k+1}},\ & \|\mathscr{Q}-\mathrm{id}\|_{C^k},\ \mathrm{dist}_{C^{k+1}}(\di u\circ \mathscr{Q},\mathrm{id}),\ \|h-1\|_{C^{k+1}} \}\\  &\leq \omega_k( \|\alpha-\alpha_0\|_{C^{k+2}}), 
\end{split}
\end{equation}
for every integer $k\geq 0$, where the $\omega_k$'s are suitable moduli of continuity. In the following argument, we will need to successively replace the $\omega_k$'s by larger and larger moduli of continuity, but in order to keep the notation simple we will denote these new functions by the same symbol $\omega_k$. 

By \eqref{omega1}, $u$ is $C^1$-close to the identity when $\|\alpha-\alpha_0\|_{C^2}$ is small. In particular, up to reducing the size of the positive number $\delta_0$ in \eqref{epsilon1}, we may assume that
\begin{equation}
\label{rrvicini}
\mathrm{dist}_{C^0}(u,\mathrm{id}) \leq r,
\end{equation}
where $r$ is the positive number given by Lemma \ref{boundsonforms}. 

Let us consider now the one-form $\beta:= u^* \alpha$, so that $R_{\beta}=u^* R_{\alpha}$ and (\ref{campi}) can be rewritten as
\begin{equation}
\label{campi2}
h\, R_{\beta} = R_{\alpha_0} - \mathscr{Q} [V].
\end{equation}
For every $k\geq 0$, we can bound the $C^k$-norm of the difference $\beta-\alpha_0$ using Lemma \ref{boundsonforms} by
\[
\begin{split}
\|\beta - \alpha_0\|_{C^k} &\leq \| u^* (\alpha - \alpha_0)\|_{C^k} + \|u^* \alpha_0 - \alpha_0\|_{C^k} \\ &\lesssim \|\alpha-\alpha_0\|_{C^k}  \|\di u\|_{C^k}(1 + \|\di u\|_{C^{k-1}}^k) \\ & \quad + \|\alpha_0\|_{C^{k+1}} \,\mathrm{dist}_{C^{k+1}}(u,\mathrm{id}) ( 1 +  \|\di u\|_{C^k}^{k+1}),
\end{split}
\]
where for $k=0$ the undefined term $\|\di u\|_{C^{k-1}}^k$ is set to be zero.
Using (\ref{omega1}), we then get a bound of the form
\begin{equation}
\label{omega4}
\|\beta - \alpha_0\|_{C^k} \leq \omega_k (\|\alpha-\alpha_0\|_{C^{k+2}}) \qquad \forall k\geq 0.
\end{equation}
Similarly, Lemma \ref{boundsonforms} implies that the $C^k$-norm of the two-form
\[
\di\beta - \di\alpha_0= u^* (\di\alpha - \di\alpha_0)+ u^* \di\alpha_0 - \di\alpha_0
\]
has a bound of the form
\begin{equation}
\label{omega5}
\|\di\beta - \di\alpha_0\|_{C^k} \leq \omega_k (\|\alpha - \alpha_0\|_{C^{k+2}})\qquad \forall k\geq 0.
\end{equation}

We define the function $S\in \Omega^0(M)$ by
\[
S(x) := \frac{1}{T_0} \int_{\pi^{-1}(\pi(x))} \beta,
\]
where $T_0$ is the common period of the orbits of $R_{\alpha_0}$. By construction, the function $S$ is invariant under the action of the Reeb flow of $\alpha_0$, i.e.\ $\mathcal{L}_{R_{\alpha_0}} S=0$. From (\ref{omega4}) we obtain that $S$ is close to the constant function 1:
\begin{equation}
\label{omega7}
\|S-1\|_{C^k} \leq \omega_k (\|\alpha - \alpha_0\|_{C^{k+2}}) \qquad \forall k\geq 0.
\end{equation}
Denote by $\phi^t_{\alpha_0}$ the  flow of $R_{\alpha_0}$ and by 
\[
\gamma_x: \R/T_0 \Z \rightarrow M, \qquad \gamma_x(t) := \phi^t_{\alpha_0}(x),
\]
its orbit through $x\in M$. Then the function $S$ has the form
\[
S(x) = \frac{1}{T_0} \mathscr{S}(\gamma_x),
\]
where $\mathscr{S}$ is the action functional defined by the one-form $\beta$, i.e.\
\[
\mathscr{S}: C^{\infty}(  \R/T_0 \Z,M) \rightarrow \R, \qquad \mathscr{S}(\gamma) := \int_{\R/T_0 \Z} \gamma^* \beta.
\]
The Gateaux differential of $\mathscr{S}$ at the curve $\gamma$ is
\[
\mathscr{S}(\gamma)[\xi]  = \int_{\R/T_0 \Z} \gamma^* (\imath_{\xi} \di\beta) = \int_{\R/T_0 \Z} \di\beta[\xi(t),\gamma'(t)]\, \di t,
\]
for every tangent vector field $\xi$ along $\gamma$. The chain rule implies that the differential of $S$ has the form
\begin{equation}
\label{diffS}
\di S(x)[w] = \frac{1}{T_0} \int_{\R/T_0 \Z} \di\beta \bigl[ \di\phi_{\alpha_0}^t(x)[w] , R_{\alpha_0}(\phi_{\alpha_0}^t(x)) \bigr]\, \di t, \qquad \forall x\in M, \; w\in T_x M.
\end{equation}
The above integrand vanishes if $\beta=\alpha_0$, so this identity and (\ref{omega5}) imply the bound
\begin{equation}
\label{omega8}
\|\di S\|_{C^k} \leq \omega_k (\|\alpha - \alpha_0\|_{C^{k+2}}) \qquad \forall k\geq 0,
\end{equation}
which, together with (\ref{omega7}) for $k=0$, implies
\begin{equation}
\label{omega9}
\|S-1\|_{C^{k+1}} \leq \omega_k (\|\alpha - \alpha_0\|_{C^{k+2}}) \qquad \forall k\geq 0.
\end{equation}
By the definition of $S$, the one-form $\beta - S \alpha_0$ satisfies
\[
\int_{\pi^{-1}(b)} ( \beta - S \alpha_0 ) = 0 \qquad \forall b\in B,
\]
so by Lemma \ref{splitting} it splits as
\[
\beta - S \alpha_0 = \eta + \di f,
\]
where $\eta\in \Omega^1(M)$ satisfies $\imath_{R_{\alpha_0}} \eta=0$ and $f\in \Omega^0(M)$ has average zero on every orbit of $R_{\alpha_0}$. Moreover, the same lemma gives us the estimates
\begin{equation}
\label{bdssplit}
\begin{split}
\|\eta\|_{C^k} &\lesssim \|\beta-S\alpha_0\|_{C^k} + \| \di\imath_{R_{\alpha_0}} (\beta - S \alpha_0)\|_{C^k}, \\
\|f\|_{C^{k+1}} &\lesssim \|\imath_{R_{\alpha_0}} (\beta-S\alpha_0)\|_{C^k} + \| \di \imath_{R_{\alpha_0}} (\beta - S \alpha_0)\|_{C^k},
\end{split}
\end{equation}
for every $k\geq 0$. From (\ref{omega4}) and (\ref{omega7})  we obtain bounds of the following form for the $C^k$-norm of $\beta-S\alpha_0$, for every $k\geq 0$:
\begin{equation}
\label{omega11}
\|\beta-S\alpha_0\|_{C^k} \leq \|\beta-\alpha_0\|_{C^k} + \|(1-S) \alpha_0\|_{C^k} \leq \omega_k (\|\alpha - \alpha_0\|_{C^{k+2}}).
\end{equation}
Now we wish to estimate the $C^k$-norm of the one-form $\di\imath_{R_{\alpha_0}} (\beta-S\alpha_0)$. We have
\begin{equation}
\label{pride}
\imath_{R_{\alpha_0}} \beta  = \imath_{R_{\alpha_0}} u^* \alpha = u^* \bigl(  \imath_{u_* R_{\alpha_0}} \alpha \bigr).
\end{equation}
Applying the push-forward operator by $u$ to (\ref{campi}) we obtain
\[
u_* R_{\alpha_0} = (h\circ u^{-1}) R_{\alpha} + Y,
\]
where $Y$ is the vector field
\[
Y:= \di u\circ \mathscr{Q}[V\circ u^{-1}],
\]
and hence
\[
\imath_{u_* R_{\alpha_0}} \alpha = h\circ u^{-1} +  \imath_{Y} \alpha.
\]
By plugging this formula into (\ref{pride}) we obtain the identity
\begin{equation}
\label{ide}
\imath_{R_{\alpha_0}} (\beta-S\alpha_0) =  \imath_{R_{\alpha_0}} \beta -  S = h + u^* (\imath_Y \alpha) - S.
\end{equation}
By (\ref{omega1}), the vector field $Y$ has the bound
\[
\|Y\|_{C^{k+1}} \leq \omega_k (\|\alpha - \alpha_0\|_{C^{k+2}}) \qquad \forall k\geq 0,
\]
and hence we have
\begin{equation}
\label{omega12}
\|\imath_Y \alpha\|_{C^{k+1}} \leq \|\alpha\|_{C^{k+1}} \, \omega_k (\|\alpha - \alpha_0\|_{C^{k+2}})\leq\omega_k (\|\alpha - \alpha_0\|_{C^{k+2}}) \qquad \forall k\geq 0.
\end{equation}
Since $\imath_Y \alpha$ is a zero-form, by Lemma \ref{boundsonforms} we have 
\[
\|u^*(\imath_Y \alpha)\|_{C^{k+1}} \lesssim \|\imath_Y \alpha\|_{C^{k+1}} ( 1 + \|\di u\|_{C^k}^{k+1}) \qquad \forall k\geq 0,
\] 
so (\ref{omega1}) and (\ref{omega12}) imply a bound of the form
\[
\|u^*(\imath_Y \alpha)\|_{C^{k+1}} \leq  \omega_k (\|\alpha - \alpha_0\|_{C^{k+2}}) \qquad \forall k\geq 0.
\]
The above estimate, together with the identity (\ref{ide}) and the bounds (\ref{omega1}) for $h$ and (\ref{omega8}) for $\di S$, implies a bound of the form
\[
\|\di \imath_{R_{\alpha_0}} (\beta - S \alpha_0)\|_{C^k} \leq  \omega_k (\|\alpha - \alpha_0\|_{C^{k+2}}) \qquad \forall k\geq 0.
\]
Thanks to the above estimate, (\ref{bdssplit}) and (\ref{omega11}) yield the following bounds for the one-form $\eta$ and the function $f$ in the splitting of $\beta-S \alpha_0$:
\begin{equation}
\label{omega13}
\|\eta\|_{C^k} \leq \omega_k (\|\alpha - \alpha_0\|_{C^{k+2}}),\qquad \|f\|_{C^{k+1}}\leq \omega_k (\|\alpha - \alpha_0\|_{C^{k+2}})\qquad \forall k\geq 0.
\end{equation}
The differential of $\eta$ is the two-form
\begin{equation}\label{e:dieta}
\di\eta = \di(\beta - S\alpha_0) = \di \beta - \di S \wedge \alpha_0 - S \di\alpha_0,
\end{equation}
and its $C^k$-norm can be estimated by the triangle inequality as follows:
\[
\|\di\eta\|_{C^k} \leq \|\di\beta - \di\alpha_0\|_{C^k} + \|(S-1) \di\alpha_0\|_{C^k} + \|\di S \wedge \alpha_0\|_{C^k}.
\]
The above expression, together with (\ref{omega5}), (\ref{omega7}) and (\ref{omega8}), shows that the $C^k$-norm of $\di\eta$ satisfies
\begin{equation}
\label{omega14}
\|\di\eta\|_{C^k} \leq\omega_k (\|\alpha - \alpha_0\|_{C^{k+2}})\qquad \forall k\geq 0. 
\end{equation}

So far, we have proven that the diffeomorphism $u$ puts $\alpha$ into the desired normal form
\[
u^* \alpha = \beta = S \alpha_0 + \eta + \di f,
\]
so that the statements (i), (ii) and (iii) in Theorem \ref{main2} hold. Moreover, all the bounds of the theorem involving the objects appearing in these statements have been proven, see (\ref{omega1}), (\ref{omega9}), (\ref{omega13}) and (\ref{omega14}).

We now turn to the proof of (iv) and of the bound for $\mathscr F$. Contracting equation \eqref{e:dieta} by the vector field $R_{\alpha_0}$ and using (\ref{campi2}), we find
\begin{equation}
\label{ideta}
\imath_{R_{\alpha_0}} \di\eta = \imath_{R_{\alpha_0}} \di\beta + \di S = \imath_{\mathscr{Q} [V]} \di\beta + \di S.
\end{equation}
Now we wish to show that $V(x)$, which we recall belongs to $\ker \alpha_0(x)$, depends linearly on $\di S(x)$, for every $x\in M$. From (\ref{diffS}) and (\ref{campi2}) we obtain
\[
\di S(x)[w] = \frac{1}{T_0} \int_{\R/T_0 \Z} \di\beta \bigl[ \di\phi_{\alpha_0}^t(x)[w] , \mathscr{Q} [V(\phi_{\alpha_0}^t(x)) ]\bigr]\, \di t, \qquad \forall x\in M, \; w\in T_x M,
\]
and hence, using the fact that $V$ is invariant under the action of the flow $\phi_{\alpha_0}$, because $\mathcal{L}_{R_{\alpha_0}}V=0$,
\[
\di S(x)[w] = \frac{1}{T_0} \int_{\R/T_0 \Z} \di\beta \bigl[\di \phi_{\alpha_0}^t(x)[w] , \mathscr{Q}\circ \di\phi_{\alpha_0}^t(x) [V(x)] \bigr]\, \di t.
\]
We conclude that
\begin{equation}
\label{diS}
\di S(x)[w] = -B_x[V(x),w] \qquad \forall x\in M, \; w\in T_x M,
\end{equation}
where $B$ is the following bilinear form on $TM$:
\[
B_x[v,w] :=   \frac{1}{T_0} \int_{\R/T_0 \Z}  \di \beta \bigl[  \mathscr{Q}\circ \di\phi_{\alpha_0}^t(x) [v], \di\phi_{\alpha_0}^t(x)[w]  \bigr]\, \di t.
\]
Note that the above expression gives us the alternating bilinear form $\di\alpha_0$ if $\di\beta=\di\alpha_0$ and $\mathscr{Q}=\mathrm{id}$. Therefore, (\ref{omega1}) and (\ref{omega5}) imply that $B$ is close to $B_0=\di\alpha_0$:
\begin{equation}
\label{omega15}
\|B-\di\alpha_0\|_{C^k} \leq \omega_k (\|\alpha-\alpha_0\|_{C^{k+2}}) \qquad \forall k\geq 0.
\end{equation}
Consider now the restriction of $B$ to $\ker \alpha_0 \times \ker \alpha_0$ and let $\mathscr{B}: \ker \alpha_0 \rightarrow (\ker \alpha_0)^*$ be the corresponding bundle morphism, which is defined by
\[
B[v,w] = \langle \mathscr{B}[v],w \rangle \qquad \forall\,v,w\in  \ker \alpha_0, 
\]
where $\langle \cdot,\cdot\rangle$ denotes the duality pairing. Now observe that the morphism $\mathscr B_0$ associated to $B_0=\di\alpha_0$ is invertible as $\di\alpha_0$ is non-degenerate on $\ker\alpha_0$. Then, \eqref{omega15} tells us that, up to reducing the size of the positive number $\delta_0$ from (\ref{epsilon1}), the morphism $\mathscr B$ is invertible with
\begin{equation}\label{omega16}
\|\mathscr{B}^{-1} - \mathscr{B}_0^{-1} \|_{C^k} \leq \omega_k (\|\alpha-\alpha_0\|_{C^{k+2}}) \qquad \forall k\geq 0.
\end{equation}
Identifying $(\ker \alpha_0)^*$ with the subspace of $T^*M$ consisting of one-forms vanishing on $R_{\alpha_0}$, we infer from (\ref{diS}) that
\begin{equation}
\label{VvS}
V(x) = -\mathscr{B}^{-1}[\di S(x)].
\end{equation}
From (\ref{ideta}) we conclude that
\[
\imath_{R_{\alpha_0}} \di\eta = \di S - \imath_{\mathscr{Q} \circ \mathscr{B}^{-1} [\di S]} \di\beta.
\]
We can therefore uniquely define the endomorphism $\mathscr{F} : T^*M \rightarrow T^*M$ by setting
\[
\mathscr F[\alpha_0]:=0,\qquad \mathscr{F} [\xi] := \xi - \imath_{\mathscr{Q} \circ \mathscr{B}^{-1} [\xi]} \di\beta,\quad\forall\,\xi\in(\ker\alpha_0)^*,
\] 
and we obtain the desired identity
\[
\imath_{R_{\alpha_0}} \di\eta = \mathscr{F} [\di S].
\]
If $\mathscr F_0$ is the endomorphism corresponding to $\mathscr B_0$, the tautological identity
\[
\imath_{\mathscr{B}^{-1}_0 [\xi]} \di\alpha_0 = \xi,\quad\forall\,\xi\in(\ker\alpha_0)^*
\]
implies that $\mathscr F_0$ is the zero endomorphisms since
\[
\mathscr{F}_0 [\xi] = \xi - \imath_{\mathscr{B}^{-1}_0 [\xi]} \di\alpha_0=0 \quad \forall\,\xi \in (\ker\alpha_0)^*.
\]
From the definition of $\mathscr F$, we see that the bounds on $\mathscr Q$, $\mathscr B^{-1}$ and $\di\beta$ established in \eqref{omega1}, \eqref{omega16}, \eqref{omega5} imply that
\[
\|\mathscr{F} \|_{C^k}=\|\mathscr F-\mathscr F_0\|_{C^k} \leq  \omega_k (\|\alpha-\alpha_0\|_{C^{k+2}}) \qquad \forall k\geq 0,
\]
concluding the proof of (iv) and of the bound for $\mathscr F$.

There remains to prove (v) and the bound for $Z$. By applying the one-form $\beta$ to (\ref{campi2}) and using (\ref{VvS}) we find
\[
h = S + \imath_{R_{\alpha_0}} \di f - \imath_{\mathscr Q [V]} \beta = S + \imath_{R_{\alpha_0}} \di f + \imath_{\mathscr{Q} \circ \mathscr{B}^{-1} [\di S]} \beta.
\]
Defining the section $W$ of $\ker \alpha_0$ dually by
\[
\iota_W\xi = - \imath_{\mathscr{Q} \circ \mathscr{B}^{-1} [\xi]} \beta, \qquad \forall \xi\in (\ker \alpha_0)^*,
\]
we rewrite the above identity as
\begin{equation}
\label{hanum}
\imath_{R_{\alpha_0}} \di f = \imath_W \di S + h - S.
\end{equation}
By averaging along the orbits of the flow of $R_{\alpha_0}$ and using the fact that $h-S$ is invariant under this flow, we obtain the identity
\[
0 = \imath_{\overline{W}} \di S + h - S,
\]
where $\overline{W}$ denotes the averaged vector field
\[
\overline{W}(x) := \frac{1}{T_0} \int_{\R/T_0 \Z} \di \phi_{\alpha_0}^{-t}(x) \bigl[ W(\phi_{\alpha_0}^t(x) \bigr]\, \di t, \qquad \forall x\in M.
\]
Defining the vector field $Z:= W - \overline{W}$, which is also a section of $\ker \alpha_0$ and has average zero along the orbits of $R_{\alpha_0}$, (\ref{hanum}) becomes
\[
\imath_{R_{\alpha_0}} \di f = \imath_Z \di S.
\]
The identity
\[
- \iota_W\xi = \imath_{\mathscr Q \circ \mathscr{B}^{-1} [\xi] } \beta - \imath_{\mathrm{id}\circ\mathscr{B}_0^{-1}[\xi]} \alpha_0 = \imath_{\mathscr{Q} \circ \mathscr{B}^{-1} [\xi] } (\beta - \alpha_0) + \imath_{(\mathscr Q - \mathrm{id}) \circ \mathscr{B}_0^{-1} [\xi]} \alpha_0 + \imath_{\mathscr{Q} \circ (\mathscr{B}^{-1} - \mathscr{B}_0^{-1})[\xi]} \alpha_0,
\]
together with (\ref{omega1}), (\ref{omega4}) and (\ref{omega16}), implies the bound
\[
\|W\|_{C^k} \leq \omega_k(\|\alpha-\alpha_0\|_{C^{k+2}}).
\]
The same bounds holds also for $\overline{W}$ and hence for $Z$. This concludes the proof of (v) and of the bound for $Z$. The proof of Theorem \ref{main2} is complete.

\end{proof}

\section{The variational principle}
\label{varprinsec}

In this section, we prove Proposition \ref{mainprop1} from the Introduction, namely the variational principle for contact forms in normal form, and we discuss some consequences of it and Theorem \ref{main2}.

\begin{proof}[Proof of Proposition \ref{mainprop1}]
Assume that $\alpha_0$ is a Zoll contact form on $M$ and $\beta$ is a contact form on $M$ of the form
\begin{equation}
\label{labeta}
\beta = S\, \alpha_0 + \eta + \di f,
\end{equation}
where $S\in \Omega^0(M)$ is positive and invariant under the Reeb flow of $\alpha_0$, $\eta\in \Omega^1(M)$ satisfies $\imath_{R_{\alpha_0}} \eta=0$ and $\imath_{R_{\alpha_0}} \di \eta = \mathscr{F} [\di S]$ for some endomorphism $\mathscr{F} : T^*M \rightarrow T^*M$ lifting the identity, and $f\in \Omega^0(M)$. We denote by $\pi: M \rightarrow B$ the $S^1$-bundle determined by the flow of $R_{\alpha_0}$ and by $\widehat{S}: B \rightarrow \R$ the function defined by $S=\widehat{S}\circ \pi$.

By differentiating (\ref{labeta}) and contracting along $R_{\alpha_0}$ we obtain the identity
\[
\imath_{R_{\alpha_0}} \di \beta = \imath_{R_{\alpha_0}} ( \di S \wedge \alpha_0 + S\, \di \alpha_0 + \di \eta ) = - \di S + \mathscr{F}[\di S].
\]
Let $b\in B$ be a critical point of $\widehat{S}$. Then the circle $\pi^{-1}(b)$ consists of critical points of $S$, and the above identity  shows that $\imath_{R_{\alpha_0}} \di \beta$ vanishes on this circle. Therefore, $R_{\beta}$ is parallel to $R_{\alpha_0}$ on $\pi^{-1}(b)$, and hence $\pi^{-1}(b)$ is a closed orbit of $R_{\beta}$. Its period is
\[
\int_{\pi^{-1}(b)} \beta = \int_{\pi^{-1}(b)} ( S\, \alpha_0 + \eta + \di f ) = \widehat{S}(b)  \int_{\pi^{-1}(b)} \alpha_0 = \widehat{S}(b) T_{\min}(\alpha_0).
\]
Assume now that $\beta$ is Zoll. Therefore, all its closed orbits have the same period, and in particular this is true for the closed orbits corresponding to the maxima and minima of $\widehat{S}$ on $B$. The above formula for the periods then forces $\max \widehat{S} = \min \widehat{S}$, i.e.~$\widehat{S}$ - or equivalently $S$ - is constant.

Conversely, assume that $\widehat{S}$ and $S$ are constantly equal to a positive number $S_0$. Then all the points in $B$ are critical for $\widehat{S}$ and hence each circle $\pi^{-1}(b)$ is a closed orbit of $R_{\beta}$ of period $S_0 T_{\min}(\alpha_0)$. This shows that $\beta$ is Zoll.
\end{proof}

Together with Moser's argument, Theorem \ref{main2} and Proposition \ref{mainprop1} can be used to prove the $C^2$-local rigidity of Zoll contact form, i.e.\ the following statement.

\begin{prop}
\label{rigid-zoll}
Let $\alpha_0$ be a Zoll contact form on a closed manifold $M$ with closed orbits of common minimal period $T_0$. Then $\alpha_0$ has a $C^2$-neighborhood $\mathcal{N}$ such that if $\alpha\in \mathcal{N}$ is a Zoll contact form with closed orbits of common minimal period $T$, then there exists a diffeomorphism $v: M \rightarrow M$ such that 
\[
v^* \alpha = {\textstyle \frac{T}{T_0}} \alpha_0.
\]
Moreover, for every integer $k\geq 0$ there is a modulus of continuity $\omega_k$ such that
\[
\mathrm{dist}_{C^k}(v,\mathrm{id}) \leq \omega_k ( \|\alpha-\alpha_0\|_{C^{k+2}}).
\]
\end{prop}

\begin{proof}
Let $\mathcal{N}$ be the set of contact forms $\alpha$ such that $\|\alpha-\alpha_0\|_{C^2}< \delta_0$, with $\delta_0$ as in Theorem \ref{main2}. Let $\alpha\in \mathcal{N}$ be a Zoll contact form with closed orbits of common minimal period $T$. Consider a diffeomorphism $u: M \rightarrow M$ such that $u^* \alpha$ has the normal form
\[
u^* \alpha = S \, \alpha_0 + \eta + \di f,
\]
where $S$, $\eta$ and $f$ satisfy all the requirements stated in Theorem \ref{main2}. By Proposition \ref{mainprop1}, the function $S$ is constant and equal to $T/T_0$, so statements (ii), (iv) and (v) of 
Theorem \ref{main2} imply that $f$ is identically zero and $\imath_{R_{\alpha_0}} \di\eta=0$. We then have
\[
u^* \alpha = {\textstyle \frac{T}{T_0}} \, \alpha_0 + \eta,
\]
where the one-form $\eta$ satisfies $\imath_{R_{\alpha_0}} \eta =0$ and $\imath_{R_{\alpha_0}} \di \eta=0$. Moreover, the bounds of Theorem \ref{main2} imply
\begin{equation}
\label{thebound1}
\max\{ \mathrm{dist}_{C^{k+1}}(u,\mathrm{id}), |T-T_0|,  \|\eta\|_{C^k}, \|\di \eta\|_{C^k} \} \leq \omega_k (\|\alpha-\alpha_0\|_{C^{k+2}}),
\end{equation}
for some modulus of continuity $\omega_k$. This modulus of continuity will be replaced by larger ones in the following argument, but in order to keep the notation simple we use the same symbol $\omega_k$ for all these moduli of continuity.

It is now enough to find a diffeomorphism $w: M \rightarrow M$ which satisfies
\begin{equation}
\label{dovev1}
w^*\Bigl(  {\textstyle \frac{T}{T_0}} \, \alpha_0 + \eta \Bigr) =  {\textstyle \frac{T}{T_0}} \, \alpha_0,
\end{equation}
and 
\begin{equation}
\label{uuu}
\mathrm{dist}_{C^k}(v,\mathrm{id}) \leq \omega_k ( \|\alpha-\alpha_0\|_{C^{k+2}}).
\end{equation}
Such a diffeomorphism can be found by Moser's homotopy argument. Indeed, we set
\[
\beta_t := {\textstyle \frac{T}{T_0}} \, \alpha_0 + t  \eta,
\]
and get from (\ref{thebound1}):
\begin{equation}
\label{thebound2}
\max_{t\in [0,1]} \|\beta_t - \alpha_0\|_{C^k} \leq  \omega_k (\|\alpha-\alpha_0\|_{C^{k+2}}), \quad \max_{t\in [0,1]} \|\di \beta_t - \di \alpha_0\|_{C^k} \leq  \omega_k (\|\alpha-\alpha_0\|_{C^{k+2}}).
\end{equation}
The above bounds for $k=0$ imply that, up to replacing $\mathcal{N}$ by a smaller $C^2$-neighborhood of $\alpha_0$, $\beta_t$ is a contact form for every $t\in [0,1]$. Note that the fact that $\imath_{R_{\alpha_0}} \di \eta$ is identically zero implies that the Reeb vector field of $\beta_t$ is parallel to $R_{\alpha_0}$ for every $t\in [0,1]$. 
The contact structure $\ker \beta_t$ depends smoothly on $t\in [0,1]$ and since $\di\beta_t$ is non-degenerate on it, we can find a smooth family of vector fields $\{Y_t\}_{t\in [0,1]}$ on $M$ such that
\[
Y_t\in \ker \beta_t, \qquad \imath_{Y_t} \di\beta_t |_{\ker \beta_t} = -\eta|_{\ker \beta_t}.
\]
Since $\eta$ vanishes on the line $\R R_{\beta_t} = \R R_{\alpha_0}$, we actually have
\[
\imath_{Y_t} \di\beta_t = -\eta
\]
on the whole tangent bundle of $M$. Thanks to the bounds (\ref{thebound1}) and  (\ref{thebound2}) we have
\begin{equation}
\label{thebound3}
\max_{t\in [0,1]} \|Y_t\|_{C^k} \leq \omega_k ( \|\alpha-\alpha_0\|_{C^{k+2}}).
\end{equation}
Let $\phi_t$ be the path of diffeomorphisms of $M$ that is defined by integrating the non-autonomous vector field $Y_t$, i.e.\
\[
\frac{d}{dt} \phi_t = Y_t(\phi_t), \quad \phi_0 = \mathrm{id}.
\]
From Cartan's identity and from the properties of $Y_t$ we find
\[
\frac{\di}{\di t} \phi_t^* \beta_t = \phi_t^* \Bigl( \mathcal{L}_{Y_t} \beta_t + \frac{\di \beta_t}{\di t} \Bigr) =  \phi_t^* \Bigl( \imath_{Y_t} \di\beta_t + \di \imath_{Y_t} \beta_t - \eta \Bigr) = \phi_t^* (\eta-\eta) = 0,
\]
which together with the identity $\phi_0^* \beta_0 = \beta_0$ implies
\[
\phi_t^* \beta_t = \beta_0 \qquad \forall t\in [0,1].
\]
Thus, the diffeomorphism $w:= \phi_1$ satisfies (\ref{dovev1}). Finally, (\ref{uuu}) holds because of (\ref{thebound3}).
\end{proof}

\begin{rem}
Let $\alpha_0$ be a Zoll contact form on a closed manifold $M$ with closed orbits of common minimal period $T_0$. The first part of the proof of the above proposition shows that if $\alpha$ is a Zoll contact form which is $C^2$-close to $\alpha_0$ and its orbits have minimal period $T$, then there is a diffeomorphism $u: M \rightarrow M$ 
such that
\[
u^* R_{\alpha} = {\textstyle \frac{T_0}{T}} R_{\alpha_0}
\]
and is $C^{k+1}$-close to the identity when $\alpha$ is $C^{k+2}$-close to $\alpha_0$. In other words, we obtain a better bound on the distance of the diffeomorphism $u$ from the identity if we just require $u$ to conjugate the Reeb flows (up to a linear time reparametrization). We have proven this using the normal form from Theorem \ref{main2}, but one can also deduce it from the structural stability of free $S^1$-actions, which is a more elementary fact (see e.g.\ \cite[Lemma 4.7]{bk20}).
\end{rem}

As observed in the Introduction, Theorem \ref{main2} and Proposition \ref{mainprop1} immediately imply a multiplicity result for closed orbits of Reeb flows close to Zoll ones that goes back to Weinstein \cite{wei73b}. 
Denoting by $\sigma_{\mathrm{prime}} (\alpha)$ the prime spectrum of $\alpha$, i.e.\ the set of periods of the non-iterated closed orbits of $R_{\alpha}$, we can complement this multiplicity result with a spectral rigidity result and state it as follows.

\begin{cor} \label{vpina}
Let $\alpha_0$ be a Zoll contact form on a closed manifold $M$ with closed orbits of common minimal period $T_0$, and let $\pi: M \rightarrow B$ be the corresponding $S^1$-bundle. For every $\epsilon>0$ there exists $\delta>0$ such that every contact form $\alpha$ with $\|\alpha-\alpha_0\|_{C^2} < \delta$ has at least as many closed Reeb orbits with period in the interval $(T_0-\epsilon,T_0+\epsilon)$ as the minimal number of critical points of a smooth function on $B$. Moreover, if for such a contact form $\alpha$ the set
\[
\sigma_{\mathrm{prime}} (\alpha) \cap (T_0-\epsilon,T_0+\epsilon)
\]
contains only one element $T$, then $\alpha$ is Zoll and there exists a diffeomorphism $v: M \rightarrow M$ such that $v^* \alpha = \frac{T}{T_0} \alpha_0$.
\end{cor}

\begin{proof}
If $\|\alpha-\alpha_0\|_{C^2}<\delta$ with $\delta$ small enough,  Theorem \ref{main2} gives us a diffeomorphism $u: M \rightarrow M$ such that $u^*\alpha=\beta$, with $\beta$ of the form (\ref{labeta}). Up to choosing $\delta$ small enough, we also obtain
\[
\|S-1\|_{C^0} < \frac{\epsilon}{T_0}.
\]
Denote by $\widehat{S}: B \rightarrow \R$ the induced function on $B$. By Proposition \ref{mainprop1}, for every critical point of $\widehat{b}$ of $\widehat{S}$ the circle $\pi^{-1}(b)$ is a closed orbit of $R_{\beta}= u^* R_{\alpha}$ of period $\widehat{S}(b) T_0 \in (T_0-\epsilon,T_0 + \epsilon)$, and hence $u(\pi^{-1}(b))$ is a closed orbit of $R_{\alpha}$ of the same period. This proves the first statement. If the prime spectrum of $\alpha$ has just one element in the interval $(T_0-\epsilon,T_0+\epsilon)$ then 
$\widehat{S}$ must be constant, and hence $\alpha$ is Zoll. The last statement follows from Proposition \ref{rigid-zoll}.
\end{proof}

The second statement in the corollary above is a local version, in arbitrary dimension, of a spectral rigidity phenomenon that has been proven by Cristofaro-Gardiner and Mazzucchelli in dimension three, see \cite[Corollary 1.2]{cgm20}: Any contact form $\alpha$ on a closed three-manifold whose prime spectrum consists of a single element is Zoll. The proof of the latter result uses embedded contact homology.

\begin{rem} 
The vector field $R_{\alpha}$ might of course have many other closed orbits of very large period, but it is natural to ask whether all the closed orbits of $R_{\alpha}$ of period close to $T_0$ are determined by the variational principle $\widehat{S}$. This is indeed true, provided that $\alpha$ is $C^3$-close to $\alpha_0$: For every $\epsilon>0$ there exists $\rho>0$ such that if $\|\alpha-\alpha_0\|_{C^3} < \rho$ then every non-iterated closed orbit of $R_{\alpha}$ has either period larger than $1/\epsilon$ or contained in the interval $(T_0-\epsilon,T_0+\epsilon)$, and in the latter case it is of the form $u(\pi^{-1}(b))$ for some critical point $b$ of $\widehat{S}$. Thanks to identity (\ref{VvS}), this follows from the more general Proposition \ref{Vvedetutto} that is proved in Appendix \ref{appbottkol}.
\end{rem}

\section{The volume formula}

In this section, we wish to prove Proposition \ref{mainprop2} from the Introduction.  In the proof we need the notion of dual endomorphism on the space of alternating forms. 
If $M$ is a $d$-dimensional manifold, then the vector bundle $\Lambda^{d}M$ is one-dimensional and the wedge product induces a non-degenerate pairing 
\[
\Lambda^k M \times \Lambda^{d-k} M \rightarrow \Lambda^{d} M, \qquad (\gamma_1,\gamma_2) \mapsto \gamma_1 \wedge \gamma_2,
\]
for every $k=0,1,\dots,d$. Therefore, every endomorphism $\mathscr{F} : \Lambda^k M \rightarrow \Lambda^k M$ has a dual endomorphism 
\[
\mathscr F^\vee:\Lambda^{d-k} M \to \Lambda^{d-k} M
\]
such that
\[
\mathscr F[\gamma_1]\wedge\gamma_2=\gamma_1\wedge\mathscr F^{\vee}[\gamma_2],\qquad \forall\,(\gamma_1,\gamma_2)\in \Lambda^k M\times \Lambda^{d-k} M.
\]
Moreover,
\[
\|\mathscr F^\vee\|_{C^k}\lesssim \|\mathscr F\|_{C_k}\qquad\forall\,k\geq0.
\]
We now proceed with the proof of the volume formula.

\begin{proof}[Proof of Proposition \ref{mainprop2}]
Let $\alpha_0$ be a Zoll contact form on the $(2n-1)$-dimensional closed manifold $M$. Our first aim is to compute the integral
\[
\int_M \beta \wedge \di \beta^{n-1}
\]
for the one-form
\[
\beta := S\,\alpha_0 + \eta + \di f,
\]
where $S,f\in \Omega^0(M)$ and $\eta\in \Omega^1(M)$ satisfies 
\[
\imath_{R_{\alpha_0}} \eta = 0, \qquad \imath_{R_{\alpha_0}}\di  \eta = \mathscr{F}[\di S],
\]
for some endomorphism $\mathscr{F}: T^*M \rightarrow T^*M$ lifting the identity. 

An elementary computation, involving only the identity $\imath_{R_{\alpha_0}} \eta=0$ and Stokes theorem, shows that
\begin{equation}
\label{forvolume}
\begin{split}
\int_M \beta \wedge \di\beta^{n-1} =  &\int_M \Bigl( S^n \alpha_0 \wedge \di\alpha_0^{n-1} + \sum_{j=1}^{n-1} \binom{n}{j} \di(S^j) \wedge \alpha_0 \wedge \di\alpha_0^{j-1} \wedge \eta \wedge \di\eta^{n-1-j} \\ &+ \sum_{j=1}^{n-1} \binom{n}{j-1} S^{j-1} \di\alpha_0^{j-1} \wedge \eta \wedge \di\eta^{n-j} \Bigr).
\end{split} 
\end{equation}
For the reader's convenience, this computation is carried out explicitly at the end of this subsection, see Lemma \ref{computation} below. 

Observe that the operator $\xi\mapsto\alpha_0 \wedge \imath_{R_{\alpha_0}}\xi$ acts as the identity on $(2n-1)$-forms. Therefore, the forms appearing in the last sum of (\ref{forvolume}) can be manipulated as follows:
\[
\begin{split}
\di\alpha_0^{j-1} \wedge \eta  \wedge \di\eta^{n-j} &= \alpha_0 \wedge \imath_{R_{\alpha_0}} ( \di\alpha_0^{j-1} \wedge \eta \wedge \di\eta^{n-j}) \\ &= - \alpha_0 \wedge \di\alpha_0^{j-1} \wedge \eta \wedge \imath_{R_{\alpha_0}}  (\di\eta^{n-j}) \\ &= - (n-j) \alpha_0 \wedge \di\alpha_0^{j-1} \wedge \eta \wedge  (\imath_{R_{\alpha_0}} \di\eta )\wedge \di\eta^{n-1-j}.
\end{split}
\]
Here we have used the fact that $\eta$ vanishes on $R_{\alpha_0}$. Now we can use the assumption on $\di\eta$ and replace $\imath_{R_{\alpha_0}} \di\eta$ in the above expression by $\mathscr{F}[\di S]$. Using also the definition of the dual operator $\mathscr{F}^{\vee}$ at the beginning of this section, we can go on with the chain of identities and obtain
\[
\begin{split}
\di\alpha_0^{j-1} \wedge \eta  \wedge \di\eta^{n-j} &= -(n-j) \alpha_0 \wedge \di\alpha_0^{j-1} \wedge \eta \wedge \mathscr{F}[\di S] \wedge \di\eta^{n-1-j} \\ &= -(n-j)  \mathscr{F}[\di S]  \wedge \alpha_0 \wedge \di \alpha_0^{j-1} \wedge \eta \wedge \di\eta^{n-1-j}  \\ &= -(n-j)  \di S \wedge \mathscr{F}^{\vee} [ \alpha_0 \wedge \di\alpha_0^{j-1} \wedge \eta  \wedge \di\eta^{n-1-j}  ].
\end{split}
\]
Multiplication of the above form by $S^{j-1}$ gives us
\[
S^{j-1} \di\alpha_0^{j-1} \wedge \eta  \wedge \di\eta^{n-j} = - \frac{n-j}{j} \di(S^j) \wedge \mathscr{F}^{\vee} [ \alpha_0 \wedge \di\alpha_0^{j-1} \wedge \eta  \wedge \di\eta^{n-1-j}  ].
\]
By plugging the above identities into the last sum of (\ref{forvolume}) we obtain the following expression:
\[
\int_M \beta \wedge \di\beta^{n-1} =  \int_M \Bigl( S^n \alpha_0 \wedge \di\alpha_0^{n-1} + \sum_{j=1}^{n-1} \binom{n}{j} \di(S^j) \wedge \tau_j - \sum_{j=1}^{n-1} \frac{n-j}{j} \binom{n}{j-1} \di(S^j) \wedge  \mathscr{F}^{\vee} [ \tau_j  ] \Bigr),
\]
where $\tau_j$ is the $(2n-2)$-form
\[
\tau_j :=  \alpha_0 \wedge \di\alpha_0^{j-1} \wedge \eta  \wedge \di\eta^{n-1-j}.
\]
By Stokes theorem we can turn this formula into
\[
\int_M \beta \wedge \di\beta^{n-1} =  \int_M \Bigl( S^n \alpha_0 \wedge \di\alpha_0^{n-1} - \sum_{j=1}^{n-1} \binom{n}{j} S^j \di \tau_j + \sum_{j=1}^{n-1} \frac{n-j}{j} \binom{n}{j-1} S^j \di \bigl( \mathscr{F}^{\vee} [ \tau_j  ] \bigr) \Bigr).
\]
This formula can be rewritten as
\[
\int_M \beta \wedge \di\beta^{n-1} = \int_M p(x,S(x)) \, \alpha_0 \wedge \di\alpha_0^{n-1},
\]
where
\[
p(x,s) := s^n + \sum_{j=1}^{n-1} p_j(x) s^j
\]
and the functions $p_j\in \Omega^0(M)$ are defined by
\[
p_j \, \alpha_0 \wedge \di\alpha_0^{n-1} = - \binom{n}{j} \di\tau_j + \frac{n-j}{j} \binom{n}{j-1} \di \bigl( \mathscr{F}^{\vee}[\tau_j] \bigr).
\]
Since the right-hand side is an exact $(2n-1)$-form, the function $p_j$ integrates to zero when multiplied by $\alpha_0 \wedge \di \alpha_0^{n-1}$, as stated in Proposition \ref{mainprop2}.

There remains to check the last statement about the $C^0$-norm of the functions $p_j$. Namely, we must prove that for any $\epsilon>0$ there exists $\delta>0$ such that if
\begin{equation}
\label{asspicc}
\max\{\|\eta\|_{C^0}, \|\di\eta\|_{C^0}, \|\mathscr{F}\|_{C^0}\} < \delta, \qquad \max\{ \|\eta\|_{C^1}, \|\di\eta\|_{C^1}, \|\mathscr{F}\|_{C^1}\} < c,
\end{equation}
then $\|p_j\|_{C^0}< \epsilon$ for every $j=1,\dots,n-1$.

Assume that (\ref{asspicc}) holds for some positive number $\delta$, whose size will be specified in due time. Then the $(2n-2)$-form $\tau_j$ and its differential
\[
\di\tau_j = \di\alpha_0^j \wedge \eta \wedge \di\eta^{n-1-j} - \alpha_0 \wedge \di\alpha_0^{j-1} \wedge \di\eta^{n-j}
\]
have the $C^0$-bounds
\begin{equation}
\label{bound1}
\|\tau_j\|_{C^0} \leq b_0 \delta^{n-j}, \qquad 
\|\di\tau_j\|_{C^0} \leq b_0 \delta^{n-j},
\end{equation}
for a suitable constant $b_0$. Using the Leibniz formula, (\ref{asspicc}) implies also the bound
\begin{equation}
\label{bound2}
\|\tau_j\|_{C^1} \leq b_1 ( \delta^{n-j} + c \delta^{n-1-j}),
\end{equation}
for a suitable constant $b_1$. The estimates on the morphism $\mathscr{F}$ in (\ref{asspicc}) give analogous bounds for the dual morphism $\mathscr{F}^{\vee}$, i.e.
\[
\|\mathscr{F}^{\vee}\|_{C^0} \leq b_2 \delta, \qquad  \|\mathscr{F}^{\vee}\|_{C^1} \leq b_2 c,
\]
for a suitable constant $b_2$. Then the Leibniz formula together with (\ref{bound1}) and (\ref{bound2}) yield
\[
\bigl\| \di (\mathscr{F}^{\vee} [\tau_j] ) \bigr\|_{C^0} \leq \|\mathscr{F}^{\vee}\|_{C^1} \|\tau_j\|_{C^0} + \|\mathscr{F}^{\vee}\|_{C^0} \|\tau_j\|_{C^1} \leq b_0 b_2 c\delta^{n-j} + b_1 b_2 ( \delta^{n-j} + c \delta^{n-1-j}) \delta.
\]
The second bound in (\ref{bound1}) and the above one show that, by choosing $\delta$ small enough, the $C^0$-norm of both $\di\tau_j$ and $\di (\mathscr{F}^{\vee} [\tau_j] )$ can be made arbitrarily small. By definition of the densities $p_j$, this implies that we can find a positive number $\delta$, depending on $c$, such that (\ref{asspicc}) implies
\[
\|p_j\|_{C^0} < \epsilon \qquad \forall j=1,\dots,n-1.
\]
This concludes the proof.
\end{proof}
We conclude this subsection by reproducing the computations leading to identity (\ref{forvolume}).
\begin{lem}
	\label{computation}
	Assume that $\beta\in \Omega^1(M)$ has the form $\beta = S \alpha_0 + \eta+\di f$, 
	where $S,f\in \Omega^0(M)$ and $\eta\in \Omega^1(M)$ is such that $\imath_{R_{\alpha_0}} \eta=0$. Then the identity (\ref{forvolume}) holds.
\end{lem}

\begin{proof}
	We set
	\[
	\gamma:=S\alpha_0+\eta,
	\]
	so that $\beta=\gamma+\di f$. Then $\di\gamma=\di\beta$ and
	\[
	\beta\wedge \di\beta^{n-1}=\gamma\wedge\di\gamma^{n-1}+\di\bigl(f\di\gamma^{n-1}\bigl).
	\]
	By Stokes Theorem
	\begin{equation}\label{e:ab}
	\int_M\beta\wedge\di\beta^{n-1}=\int_M\gamma\wedge\di\gamma^{n-1}
	\end{equation}
	and we will now compute the right-hand side of this equality. The differential of $\gamma$ is the two-form
	\[
	\di\gamma = \di S \wedge \alpha_0 + S \di\alpha_0 + \di\eta,
	\]
	and its $(n-1)$-th wedge power is the $(2n-2)$-form
	\[
	\begin{split}
	\di\gamma^{n-1} &= (n-1) \di S \wedge \alpha_0 \wedge ( S \di\alpha_0 + \di\eta)^{n-2} + ( S \di\alpha_0 + \di\eta)^{n-1} \\ &= (n-1) \di S \wedge \alpha_0 \wedge \sum_{j=0}^{n-2} \binom{n-2}{j} S^j \di\alpha_0^j \wedge \di\eta^{n-2-j} +  \sum_{j=0}^{n-1} \binom{n-1}{j} S^j \di\alpha_0^j \wedge \di\eta^{n-1-j} \\ &=  \sum_{j=0}^{n-2}\frac{n-1}{j+1}  \binom{n-2}{j} \di(S^{j+1}) \wedge \alpha_0 \wedge \di\alpha_0^j \wedge \di\eta^{n-2-j} +  \sum_{j=0}^{n-1} \binom{n-1}{j} S^j \di\alpha_0^j \wedge \di\eta^{n-1-j} \\ &=  \sum_{j=0}^{n-2} \binom{n-1}{j+1} \di(S^{j+1}) \wedge \alpha_0 \wedge \di\alpha_0^j \wedge \di\eta^{n-2-j} +  \sum_{j=0}^{n-1} \binom{n-1}{j} S^j \di\alpha_0^j \wedge \di\eta^{n-1-j}.
	\end{split}
	\]
	Wedging this form with $\gamma$ we obtain the $(2n-1)$-form
	\begin{equation}
	\label{lunga1}
	\begin{split}
	\gamma \wedge \di\gamma^{n-1} = & \sum_{j=0}^{n-1} \binom{n-1}{j} S^{j+1} \alpha_0 \wedge \di\alpha_0^j \wedge \di\eta^{n-1-j} \\ & + \sum_{j=0}^{n-2} \binom{n-1}{j+1} \di(S^{j+1}) \wedge \alpha_0 \wedge \di\alpha_0^j \wedge \eta \wedge \di\eta^{n-2-j} \\ & + \sum_{j=0}^{n-1} \binom{n-1}{j} S^j \di\alpha_0^j \wedge \eta \wedge  \di\eta^{n-1-j}.
	\end{split}
	\end{equation}
	The forms with $j$ different from $n-1$ in the first sum above can be rewritten as
	\[
	\alpha_0 \wedge \di\alpha_0^j \wedge \di\eta^{n-1-j} = \di\alpha_0^{j+1} \wedge \eta \wedge \di\eta^{n-2-j} - \di ( \alpha_0 \wedge \di\alpha_0^j \wedge \eta \wedge \di\eta^{n-2-j} ) .
	\]
	Therefore, the first sum in (\ref{lunga1}) can be rewritten as
	\[
	\begin{split}
	\mbox{first sum in (\ref{lunga1})} &= S^n \alpha_0 \wedge \di\alpha_0^{n-1} + \sum_{j=0}^{n-2} \binom{n-1}{j} S^{j+1} \di\alpha_0^{j+1} \wedge \eta \wedge \di\eta^{n-2-j} \\
	& \quad - \sum_{j=0}^{n-2} \binom{n-1}{j} S^{j+1} \di( \alpha_0 \wedge \di\alpha_0^j \wedge \eta \wedge \di\eta^{n-2-j})  \\ &= S^n \alpha_0 \wedge \di\alpha_0^{n-1} + \sum_{j=1}^{n-1} \binom{n-1}{j-1} S^j \di\alpha_0^j \wedge \eta  \wedge \di\eta^{n-1-j}  \\ & \quad - \sum_{j=0}^{n-2} \binom{n-1}{j} S^{j+1} \di ( \alpha_0 \wedge \di\alpha_0^j \wedge \eta \wedge \di\eta^{n-2-j}).
	\end{split}
	\]
	By plugging the above expression into (\ref{lunga1}) and by summing the first sum of the formula above with the third sum in (\ref{lunga1}), from which we isolate the term with $j=0$, we obtain the identity
	\[
	\begin{split}
	\gamma \wedge \di\gamma^{n-1} &= S^n \alpha_0 \wedge \di\alpha_0^{n-1} + \sum_{j=1}^{n-1} \left( \binom{n-1}{j-1} + \binom{n-1}{j} \right)  S^j \di\alpha_0^j \wedge \eta \wedge \di\eta^{n-1-j}  \\ & \quad +\eta\wedge \di\eta^{n-1} - \sum_{j=0}^{n-2} \binom{n-1}{j} S^{j+1} \di( \alpha_0 \wedge \di\alpha_0^j \wedge \eta \wedge \di\eta^{n-2-j}) \\& \quad +  \sum_{j=0}^{n-2}  \binom{n-1}{j+1} \di(S^{j+1}) \wedge \alpha_0 \wedge \di\alpha_0^j \wedge \eta \wedge \di\eta^{n-2-j} .
	\end{split}
	\]
	Now we examine the first sum in the above expression. The coefficient of its $j$-th term is $\binom{n}{j}$, by the addition formula for binomial coefficients, and the term with $j=n-1$ vanishes, because both $\eta$ and $\di\alpha_0$ vanish on $R_{\alpha_0}$. By incorporating the term $\eta\wedge \di\eta^{n-1}$ into this sum, we get the identity
	\[
	\begin{split}
	\gamma \wedge \di\gamma^{n-1} &= S^n \alpha_0 \wedge \di\alpha_0^{n-1} + \sum_{j=0}^{n-2} \binom{n}{j}  S^j \di\alpha_0^j \wedge \eta \wedge \di\eta^{n-1-j}  \\ & \quad - \sum_{j=0}^{n-2} \binom{n-1}{j} S^{j+1} \di( \alpha_0 \wedge \di\alpha_0^j \wedge \eta \wedge \di\eta^{n-2-j}) \\& \quad +  \sum_{j=0}^{n-2}  \binom{n-1}{j+1}\di(S^{j+1}) \wedge \alpha_0 \wedge \di\alpha_0^j \wedge \eta \wedge \di\eta^{n-2-j} .
	\end{split}
	\]
	We now integrate over $M$ and use Stokes theorem when integrating the second sum. We obtain:
	\[
	\begin{split}
	\int_M \gamma \wedge \di\gamma^{n-1} &= \int_M \Bigl[ S^n \alpha_0 \wedge \di\alpha_0^{n-1} + \sum_{j=0}^{n-2} \binom{n}{j} S^j \di\alpha_0^j \wedge \eta \wedge \di\eta^{n-1-j}  \\ & \quad + \sum_{j=0}^{n-2} \left( \binom{n-1}{j} + \binom{n-1}{j+1} \right) \di(S^{j+1}) \wedge \alpha_0 \wedge \di\alpha_0^j \wedge \eta  \wedge \di\eta^{n-2-j}\Bigl] \end{split}
	\]
	By using again the addition formula for binomial coefficients and by shifting the indices in both sums we find the identity
	\[
	\begin{split}
	\int_M \gamma \wedge \di\gamma^{n-1} &=  \int_M \Bigl[ S^n \alpha_0 \wedge \di\alpha_0^{n-1} + \sum_{j=1}^{n-1} \binom{n}{j-1} S^{j-1} \di\alpha_0^{j-1} \wedge \eta \wedge \di\eta^{n-j}  \\ & \quad + \sum_{j=1}^{n-1} \binom{n}{j} \di(S^j) \wedge \alpha_0 \wedge \di\alpha_0^{j-1} \wedge \eta  \wedge \di\eta^{n-1-j}\Bigl],
	\end{split}
	\]
	that is precisely (\ref{forvolume}) thanks to \eqref{e:ab}.
\end{proof}

\section{The systolic inequality}
\label{systolicsec}

The first aim of this section is to put together Theorem \ref{main2}, Proposition \ref{mainprop1} and Proposition \ref{mainprop2} to prove the local systolic maximality of Zoll contact forms of Theorem \ref{main1}. We follow the argument that we already sketched in the Introduction. 

\begin{proof}[Proof of Theorem \ref{main1}]
Let $C>0$ be an arbitrary constant. Let $\alpha$ be a contact form on $M$ such that $\|\alpha-\alpha_0\|_{C^2}<\delta_0$, where $\delta_0$ is given by Theorem \ref{main2}. Then, we can find a diffeomorphism $u: M \rightarrow M$ such that 
\[
u^* \alpha = S \alpha_0 + \eta + \di f,
\]
where $S\in \Omega^0(M)$ is invariant under the flow of $R_{\alpha_0}$, $f\in \Omega^0(M)$, and $\eta\in \Omega^1(M)$ satisfies 
\[
\imath_{R_{\alpha_0}} \eta = 0 , \qquad \imath_{R_{\alpha_0}} \di \eta = \mathscr{F}[\di S],
\]
for a suitable endomorphism $\mathscr{F}: T^*M \rightarrow T^*M$. Moreover, the bounds 
\begin{equation}\label{e:boundsforthm3}
\max\Big\{ \|S-1\|_{C^{k+1}},\ \|\eta\|_{C^k},\ \|\di\eta\|_{C^{k}},\ \|\mathscr{F}\|_{C^k} \Big\}\leq \omega_k\big(\|\alpha-\alpha_0\|_{C^{k+2}}\big).
\end{equation}
hold for every $k\geq0$. 

We set $\beta:= u^* \alpha$ and observe that it suffices to prove the systolic inequality for $\beta$ because both the volume and the minimal period of Reeb orbits are invariant under diffeomorphisms:
\[
\mathrm{vol}(M,\beta) = \mathrm{vol}(M,\alpha), \qquad T_{\min}(\beta) = T_{\min}(\alpha).
\]
We apply Proposition \ref{mainprop2} to $\beta$ and find functions $p_j:M\to\R$ for $j=1,\ldots,n-1$ with zero average with respect to the volume form $\alpha_0 \wedge \di \alpha_0^{n-1}$ such that
\begin{equation}
\label{ilvolume}
\mathrm{vol}(M,\beta)=\int_Mp(x,S(x))\, \alpha_0\wedge\di\alpha_0^{n-1},
\end{equation}
where $p:M\times\R\to\R$ is defined as 
\[
p(x,s)=s^n+\sum_{j=1}^{n-1}p_j(x)s^j.
\] 
Assume now that $\|\alpha-\alpha_0\|_{C^3}<C$. The last three bounds in (\ref{e:boundsforthm3}) for $k=1$ yield
\begin{equation}
\label{c1bounds}
\max\Big\{\|\eta\|_{C^1},\ \|\di\eta\|_{C^1},\ \|\mathscr{F}\|_{C^1} \Big\}\leq \omega_1(C). 
\end{equation}
Take now $c=1+\omega_1(C)$ and $\epsilon=\frac{n}{2(2^n-n-1)}$ in Proposition \ref{mainprop2} and obtain a corresponding $\delta>0$ such that for every $j=1,\dots,n-1$
\begin{equation}\label{e:pjest}
\max\Big\{\|\eta\|_{C^0},\ \|\di\eta\|_{C^0},\ \|\mathscr{F}\|_{C^0} \Big\}<\delta\quad\Longrightarrow\quad \|p_j\|_{C^0} < \frac{n}{2(2^n-n-1)}.
\end{equation}
We now choose $\delta_C$ such that $\omega_0(\delta_C)<\min\{1/2,\delta\}$, so that for $\|\alpha-\alpha_0\|_{C^2}<\delta_C$ we get
\[
\|S-1\|_{C^0}\leq\omega_0(\delta_C)< 1/2,\qquad \max\Big\{\|\eta\|_{C^0},\ \|\di\eta\|_{C^0},\ \|\mathscr{F}\|_{C^0} \Big\}\leq \omega_0(\delta_C)<\delta
\]
thanks to \eqref{e:boundsforthm3}.
Our choice of $\epsilon$ shows that for every $x\in M$ the function $s\mapsto p(x,s)$ is strictly monotonically increasing on the interval $[1/2,+\infty)$. Indeed, for every $x\in M$ and $s\geq 1/2$ we have
\[
\begin{split}
\frac{\partial p}{\partial s} (x,s) &= n s^{n-1} + \sum_{j=1}^{n-1} j p_j(x) s^{j-1} = s^{n-1} \left( n + \sum_{j=1}^{n-1} j p_j(x) \frac{1}{s^{n-j}} \right) \\ &\geq  s^{n-1} \left( n - \sum_{j=1}^{n-1} j 2^{n-j} \|p_j\|_{C^0}\right) \geq s^{n-1} \left( n -  \max_{j\in \{1,\dots,n-1\}} \|p_j\|_{C^0}   \sum_{j=1}^{n-1} j 2^{n-j} \right)\\&=s^{n-1} \left( n - 2 (2^n-n-1) \max_{j\in \{1,\dots,n-1\}} \|p_j\|_{C^0} \right),
\end{split}
\]
and the latter quantity is strictly positive because of (\ref{e:pjest}). 

In particular,  the function $s\mapsto p(x,s)$ is strictly monotonically increasing on the interval $[\min S,\max S]$, which is contained in $[1/2,3/2]$, and (\ref{ilvolume}) yields the inequality
\begin{equation}
\label{ines1}
\mathrm{vol}(M,\beta) =  \int_M p(x,S(x))\, \alpha_0 \wedge \di\alpha_0^{n-1} \geq \int_M p(x,\min S)\, \alpha_0 \wedge \di\alpha_0^{n-1},
\end{equation}
with equality if and only if $S(x)\equiv \min S$, which happens exactly when $S$ is constant. Since the functions $p_j$ have zero average, the latter quantity equals
\[
\int_M p(x,\min S)\, \alpha_0 \wedge \di\alpha_0^{n-1} = \int_M (\min S)^n \alpha_0 \wedge \di\alpha_0^{n-1} = (\min S)^n \, \mathrm{vol}(M,\alpha_0).
\]
By Proposition \ref{mainprop1}, the Reeb flow of $\beta$ has a closed orbit of period $(\min S) T_{\min} (\alpha_0)$. Therefore, $T_{\min}(\beta) \leq (\min S) T_{\min} (\alpha_0)$, and we deduce the inequality
\begin{equation}
\label{ines3}
\mathrm{vol}(M,\beta) \geq (\min S)^n \, \mathrm{vol}(M,\alpha_0)\geq \frac{T_{\min}(\beta)^n}{T_{\min}(\alpha_0)^n} \, \mathrm{vol}(M,\alpha_0),
\end{equation}
which can be rewritten as
\begin{equation}
\label{systolicineq}
\rho_{\mathrm{sys}} (M,\beta) :=\frac{T_{\min}(\beta)^n}{\mathrm{vol}(M,\beta)} \leq \frac{T_{\min}(\alpha_0)^n}{\mathrm{vol}(M,\alpha_0)}=\rho_{\mathrm{sys}} (M,\alpha_0).
\end{equation}
If equality holds in (\ref{systolicineq}), then it must hold also in (\ref{ines1}) and hence $S$ is constant. By Proposition \ref{mainprop1} $\beta$ is Zoll. Conversely, assume that $\beta$ is Zoll. Then all of its orbits have the same minimal period. By Proposition \ref{mainprop1} $S$ is constant. In this case, the inequalities in (\ref{ines1}) and in (\ref{ines3}) are equalities. Therefore, (\ref{systolicineq}) is an equality. This concludes the proof of the theorem.
\end{proof}

We conclude this section by discussing a lower bound for the maximal period of ``short'' periodic orbits that can be proven by an easy modification of the argument described above. 

Recall that $\sigma_{\mathrm{prime}}(\alpha)$ denotes the prime spectrum of the contact form $\alpha$, i.e.\ the set of the periods of all its non-iterated closed Reeb orbits. Denote by $T_0$ the common period of the orbits of the Zoll contact form $\alpha_0$ and fix some number $\tau > T_0$. By Corollary \ref{vpina} we can find a $C^2$-neighborhood $\mathscr{U}_{\tau}$ of $\alpha_0$ in the space of contact forms on $M$ such that for every $\alpha\in \mathscr{U}_{\tau}$ the set  $\sigma_{\mathrm{prime}}(\alpha)$ has non-empty intersection with the interval $(0,\tau]$. Therefore, the function
\[
T_{\max}(\alpha,\tau) := \max(  \sigma_{\mathrm{prime}}(\alpha)\cap (0,\tau])
\]
is well defined on $\mathscr{U}_{\tau}$. Then an easy modification of the above proof allows us to show the following lower bound for $T_{\max}(\alpha,\tau)$.

\begin{thm}
\label{diathm}
Let $\alpha_0$ be a Zoll contact form on a closed manifold $M$ with orbits of period $T_0$ and let $\tau> T_0$. Then for all $C>0$ there exists $\delta_{\tau,C}>0$ such that the $C^3$-neighborhood 
\[
 \mathscr N_{\tau,C}:=\Big\{\alpha\in \Omega^1(M)  \Big|\ \Vert\alpha-\alpha_0\Vert_{C^2}<\delta_{\tau,C},\ \Vert\alpha-\alpha_0\Vert_{C^3}<C \Big\}
 \]
 of $\alpha_0$ is contained in $\mathscr{U}_{\tau}$ and for every $\alpha \in  \mathscr N_{\tau,C}$ we have
\[
\frac{T_0^n}{\mathrm{vol}(M,\alpha_0)} \leq \frac{T_{\max}(\alpha,\tau)^n}{\mathrm{vol}(M,\alpha)},
\]
with equality if and only if $\alpha$ is Zoll. \hfill\qed 
\end{thm}

Indeed, in order to get this bound it is enough choose $\delta_{\tau,C}$ so small that
\[
\|S-1\|_{C^0} < \frac{\tau}{T_0}  -1,
\]
which implies that a circle at which $S$ achieves its maximum is a closed orbit of $R_{\beta}$ of period less than $\tau$, and to replace (\ref{ines1}) by the inequality
\[
\mathrm{vol}(M,\beta) =  \int_M p(x,S(x))\, \alpha_0 \wedge \di\alpha_0^{n-1} \leq \int_M p(x,\max S)\, \alpha_0 \wedge \di\alpha_0^{n-1},
\]
which is an equality if and only if $S$ is constant. 

\section{Convex domains}
\label{viterbo-sec}

Endow $\R^{2n}$ with coordinates $(x_1,y_1,\dots,x_n,y_n)$, with the Liouville one-form
\[
\lambda_0 := \frac{1}{2} \sum_{j=1}^{n} (x_j \, \di y_j - y_j \, \di x_j),
\]
and with the symplectic form
\[
\omega_0 = \di\lambda_0 = \sum_{j=1}^{n} \di x_j \wedge \di y_j.
\] 
We shall use also the standard identification $\R^{2n} \cong \C^n$ given by $z_j = x_j + i y_j$. The restriction of $\lambda_0$ to the unit sphere $S^{2n-1}\subset \R^{2n}$ is denoted by $\alpha_0$. Its Reeb flow is the Hopf flow
\[
(t,z) \mapsto e^{2i t} z, \qquad \forall (t,z) \in \R \times S^{2n-1},
\]
all of whose orbits are closed with period $\pi$. The contact volume of $(S^{2n-1},\alpha_0)$ is $\pi^n$, so $\alpha_0$ is Zoll with systolic ratio 1.

By starshaped smooth domain we mean here an open set of the form
\[
A_f:= \{ r z \mid z\in S^{2n-1}, \; 0 \leq r < f(z) \},
\]
where $f: S^{2n-1} \rightarrow \R$ is a smooth positive function. With this notation, the unit ball $B^{2n}$ of $\R^{2n}$ is the set $A_1$.
The $C^k$-distance of the starshaped smooth domains $A_f$ and $A_g$ is by definition the $C^k$-distance of the smooth functions $f$ and $g$ on $S^{2n-1}$. 

The one-form $\lambda_0$ restricts to a contact form $\alpha_{A_f}$ on the boundary of the starshaped domain $A_f$, and the radial projection
\[
\rho: S^{2n-1} \rightarrow \partial A_f, \qquad z \mapsto f(z) z,
\]
pulls this contact form back to the contact form $f^2 \alpha_0$ on $S^{2n-1}$:
\[
\rho^* \bigl( \alpha_{ A_f} \bigr) = f^2 \alpha_0.
\]
Therefore, this pull-back is $C^k$-close to $\alpha_0=\alpha_{B^{2n}}$ whenever $A_f$ is $C^k$-close to $B^{2n}$. The Reeb vector field of $\alpha_{A_f}$ is a non-vanishing section of the characteristic line bundle of the hypersurface $\partial A_f$, which is defined as the kernel of $\omega_0|_{TA_f}$, or equivalently as the line bundle $iN \partial A_f$, where $N \partial A_f$ denotes the normal bundle of $\partial A_f$ in $\C^n$.

The aim of this section is to prove Propositions \ref{EHZ=cyl} and \ref{symplectic} from the introduction, thus concluding the proof of the perturbative case of Viterbo's conjecture stated in Corollary \ref{cor2}. Both proofs make use of generating functions. We briefly recall here the kind of generating functions that we are going to use and some results about them.

The symplectic vector space $(\C^n \times \C^n, \omega_0 \oplus -\omega_0)$ can be identified with the cotangent bundle $T^* \C^n = \C^n \times (\C^n)^*$ by the linear symplectomorphism
\begin{equation}
\label{linearsymp}
\Phi: \C^n \times \C^n \rightarrow T^* \C^n, \qquad (z,Z) \mapsto (q,p):= \left( \frac{z+Z}{2}, i (z-Z) \right).
\end{equation}
Here, $T^* \C^n$ is endowed with its standard symplectic structure $\di p\wedge \di q$, where $q\in \C^n$, $p\in (\C^n)^*$, and the dual space $(\C^n)^*$ is identified with $\C^n$ by the standard Euclidean product on $\C^n$. The linear symplectomorphism $\Phi$ maps the diagonal $\Delta$ of $\C^n\times \C^n$ to the zero-section of $T^* \C^n$. It is an explicit linear realization of the Weinstein tubular neighborhood theorem for the Lagrangian submanifold $\Delta$ of $(\C^n \times \C^n, \omega_0 \oplus -\omega_0)$.

The graph of a symplectomorphism $\varphi: \C^n \rightarrow \C^n$ is a Lagrangian submanifold of $(\C^n \times \C^n, \omega_0 \oplus -\omega_0)$ and hence gets mapped to a Lagrangian submanifold of $T^* \C^n$ by $\Phi$.  If we assume that
\[
\| \di\varphi(z) - \mathrm{id} \| < 2 \qquad \forall z \in \C^n,
\]
by the Banach fixed point theorem the equation 
\[
q = \frac{z+ \varphi(z)}{2}
\]
can be solved uniquely for $z$, and hence $\Phi(\mathrm{graph} \;\varphi)$ is the graph of a smooth map from $\C^n$ to $(\C^n)^*$, which by the Lagrangian condition is a closed one-form on $\C^n$. A closed one-form on $\C^n$ is necessarily exact, and we deduce the existence of a smooth function $S: \C^n \rightarrow \R$ such that
\begin{equation}
\label{genfun}
i ( z - \varphi(z)) = \nabla S \left( \frac{z+\varphi(z)}{2} \right) \qquad \forall z\in \C^n.
\end{equation}
The function $S$ is called generating function for $\varphi$ and is uniquely defined up to an additive constants. A simple bootstrap argument shows that
\begin{equation}
\label{boundoS}
\|\nabla S\|_{C^k} \leq \omega_k ( \|\varphi - \mathrm{id} \|_{C^k}) \qquad \forall k\geq 0,
\end{equation}
for suitable moduli of continuity $\omega_k$. In particular, $S$ is smooth if $\varphi$ is smooth.

Conversely, if $S: \C^n \rightarrow \R$ is a smooth function such that
\[
\|\nabla^2 S(z) \| < 2 \qquad \forall z \in \C^n,
\]
the Banach fixed point theorem implies that equation (\ref{genfun}) uniquely determines a map $\varphi: \C^n \rightarrow \C^n$, which is smooth because of the smooth dependence of the fixed point in the parametric Banach fixed point theorem. More precisely, a simple bootstrap argument shows that
\begin{equation}
	\label{boundophi}
	\|\varphi - \mathrm{id} \|_{C^k} \leq \omega_k (\|\nabla S\|_{C^k} ) \qquad \forall k\geq 0,
\end{equation}
for suitable moduli of continuity $\omega_k$. Since $\Phi(\mathrm{graph} \;\varphi)$ is the graph of the closed one-form $\nabla S$, the map $\varphi$ is a symplectomorphism. Its inverse $\varphi^{-1}:\C^n\rightarrow\C^n$ is obtained by solving the equation
\[
i(\varphi^{-1}(w)-w)=\nabla S\left( \frac{\varphi^{-1}(w)+w}{2} \right) \qquad \forall w\in \C^n.
\]

The proof of Proposition \ref{EHZ=cyl} builds on the following lemma, in which $\gamma_0$ denotes the closed curve
\begin{equation}
\label{gamma0}
\gamma_0 : \R/\Z \rightarrow \C^n, \qquad \gamma_0(t) := (e^{2\pi i t}, 0,\dots,0),
\end{equation}
along which the one-form $\lambda_0$ has integral $\pi$.

\begin{lem}
\label{spostami}
There exists $\delta>0$ such that if $\gamma: \R/\Z \rightarrow \C^n$ is a smooth closed curve with $\|\gamma-\gamma_0\|_{C^2} < \delta$ and
\begin{equation}
\label{stessointegrale}
\int_{\gamma} \lambda_0 = \int_{\gamma_0} \lambda_0 = \pi,
\end{equation}
then there exists a compactly supported symplectomorphism $\varphi: \C^{n} \rightarrow \C^{n}$ such that $\varphi(\gamma(t))=\gamma_0(t)$ for every $t\in \R/\Z$ and
\begin{equation}
\label{klein}
\|\varphi - \mathrm{id} \|_{C^k} \leq \omega_k(\|\gamma-\gamma_0\|_{C^{k+1}}), \qquad \forall k\geq 0,
\end{equation}
for a suitable sequence of moduli of continuity $\omega_k$.
\end{lem}

\begin{proof}
We shall construct the symplectomorphism $\varphi$ by means of a suitable generating function $S: \C^n \rightarrow \R$ as in (\ref{genfun}). The curve
\[
\gamma_1: \R/\Z \rightarrow \C^n, \qquad \gamma_1(t) := \frac{\gamma(t) + \gamma_0(t)}{2},
\]
is an embedding when $\|\gamma-\gamma_0\|_{C^1}$ is small enough. By (\ref{genfun}), the condition $\varphi(\gamma(t))=\gamma_0(t)$ requires us to prescribe the gradient of $S$ on the image of $\gamma_1$ as follows:
\begin{equation}
\label{vuolsi}
\nabla S (\gamma_1(t)) = i ( \gamma(t) - \gamma_0(t)) \qquad \forall t\in \R/\Z.
\end{equation}
The necessary condition for being able to find a function $S$ satisfying the above identity is
\begin{equation}
\label{condizione}
\int_{\R/\Z}  \big[ i ( \gamma(t) - \gamma_0(t))\big] \cdot \gamma_1'(t)\, dt = 0.
\end{equation}
This condition holds thanks to the assumption (\ref{stessointegrale}) and turns out to be sufficient, as we now show. Thanks to (\ref{condizione}), the function
\[
s(t) := \int_0^t  \big[ i ( \gamma(\tau) - \gamma_0(\tau)) \big] \cdot \gamma_1'(\tau) \, d\tau,
\]
is 1-periodic and hence defines a smooth function on $\R/\Z$. Note that
\begin{equation}
\label{bds}
\|s\|_{C^k} \leq \omega_k( \|\gamma-\gamma_0\|_{C^k}) \qquad \forall k\geq 1,
\end{equation}
for suitable moduli of continuity $\omega_k$.

Denote by $P_1: \C^n \rightarrow \C$ and $P_2: \C^n \rightarrow \C^{n-1}$ the projections that are associated to the splitting 
$\C^n = \C \times \C^{n-1}$. Choose $\delta_0>0$ small enough so that for every smooth curve $\gamma: \R/\Z \rightarrow \C^n$ with $\|\gamma-\gamma_0\|_{C^1} < \delta_0$ the  projection $P_1 \circ \gamma_1$ is an embedding into $\C\setminus \{0\}$ which is everywhere transverse to the rays emanating from the origin. Then, if $\|\gamma - \gamma_0\|_{C^1} < \delta_0$, the map
\[
\psi:(0,+\infty) \times \R/\Z \times \C^{n-1} \rightarrow (\C \setminus \{0\})  \times \C^{n-1}, \qquad (r,t,w) \mapsto r \,\gamma_1(t) + w,
\]
is a diffeomorphism satisfying
\begin{equation}
\label{diffeovic}
\|\psi-\psi_0\|_{C^k((0,r_0] \times \R/\Z \times \C^{n-1})} \leq \omega_k(\|\gamma-\gamma_0\|_{C^k}),
\end{equation}
for every $r_0>0$, where $\psi_0$ denotes the diffeomorphism
\[
\psi_0(r,t,w) :=  r \,\gamma_0(t) + w.
\]
Let 
\[
\chi_1 : (0,+\infty) \rightarrow \R, \qquad \chi_2 : [0,+\infty) \rightarrow \R,
\]
be smooth compactly supported functions such that
\[
\chi_1(r) = 1 \quad \forall r\in \left[{\textstyle \frac{1}{2}, \frac{3}{2} } \right], 
\qquad \chi_2(r) = 1 \quad \forall r\in [0,1].
\]
The identity
\[
S( r \,\gamma_1(t) + w ) = \chi_1(r) \chi_2(|w|) \bigl( s(t) + (r-1) \big[ i (\gamma(t) - \gamma_0(t)) \big] \cdot \gamma_1(t) + \big[ i P_2 (\gamma(t) - \gamma_0(t))\big]\cdot w \bigr)
\]
defines a smooth compactly supported function $S: \C^n \rightarrow \R$. Differentiating the above identity with respect to $t$, $r$ and $w$ at $r=1$ and $w=0$ we obtain the formulas
\[
\begin{split}
\nabla S(\gamma_1(t)) \cdot \gamma_1'(t) &=  s'(t) = \big[ i ( \gamma(t) - \gamma_0(t)) \big] \cdot \gamma_1'(t), \\
\nabla S(\gamma_1(t)) \cdot \gamma_1(t) &=  \big[ i (\gamma(t) - \gamma_0(t)) \big] \cdot \gamma_1(t), \\
P_2 \nabla S(\gamma_1(t))  &= i P_2 (\gamma(t) - \gamma_0(t)),
\end{split}
\]
which are equivalent to (\ref{vuolsi}). Moreover, the above formula for $S$, (\ref{bds}) and (\ref{diffeovic}) imply the bounds
\begin{equation}
\label{bdS}
\|S\|_{C^k} \leq \omega_k ( \|\gamma-\gamma_0\|_{C^k}) \qquad \forall k\geq 1,
\end{equation}
where we have replaced the moduli of continuity $\omega_k$ by possibly larger ones. Let $\delta\leq \delta_0$ be a positive number such that $\|\gamma-\gamma_0\|_{C^2}< \delta$ implies that $\|\nabla^2 S\|_{C^0} < 2$. If $\|\gamma-\gamma_0\|_{C^2}< \delta$, from what we have seen above we deduce that the identity (\ref{genfun}) defines a smooth symplectomorphism $\varphi: \C^n \rightarrow \C^n$, which is compactly supported because $S$ is compactly supported. The identities (\ref{genfun}) and (\ref{vuolsi}) imply
\[
\varphi(\gamma(t)) = \gamma_0(t)  \qquad \forall t\in \R/\Z.
\]
Finally, the bounds (\ref{klein}) follow from (\ref{boundophi}) and  (\ref{bdS}).
\end{proof}

\begin{rem}
The above lemma is a local specialization of the following global result: The group of compactly supported symplectomorphisms of $\R^{2n}$ acts transitively of the set of smooth embedded closed curves $\gamma: \R/\Z \rightarrow \R^{2n}$ with action $\int_{\gamma} \lambda_0 = a$, for any given number $a\neq 0$. This can be proved either by generating functions or by Moser's homotopy argument, after showing that the set of smooth embedded closed curves of action $a$ is connected. The above proof by generating functions allows us to get the bound (\ref{klein}) without further loss of derivatives.
\end{rem}

We are now ready to prove Proposition \ref{EHZ=cyl} from the Introduction, stating that the Ekeland--Hofer--Zehnder capacity and the cylindrical capacity coincide on smooth convex domains that are $C^3$-close enough to the ball $B^{2n}$.

\begin{proof}[Proof of Proposition \ref{EHZ=cyl}]
Let $\epsilon>0$ be such that any diffeomorphism $\varphi: \R^{2n} \rightarrow \R^{2n}$ satisfying $\|\varphi - \mathrm{id}\|_{C^2} < \epsilon$ maps any starshaped smooth domain $C$ whose $C^2$-distance from $B^{2n}$ is smaller than $\epsilon$ to a smooth starshaped domain which is still convex. The existence of such a positive number $\epsilon$ follows from the fact that $S^{2n-1}$ is positively curved. 

Up to rescaling, it is enough to prove the identity
\begin{equation}
\label{equ}
c_{\mathrm{EHZ}}(C) = c_{\mathrm{out}}(C)
\end{equation}
for all convex smooth domains $C$ that are $C^3$-close enough to $B^{2n}$ and satisfy $T_{\min}(\alpha_C)=\pi$. Let $C$ be such a domain and $\zeta:\R/\pi \Z \rightarrow \partial C$ be a closed Reeb orbit of $\alpha_C$ of minimal period.  We claim that $\zeta$ is $C^3$-close to some Reeb orbit $\zeta_0(t) = e^{2it} z$, $z\in S^{2n-1}$, of $\alpha_0=\alpha_{B^{2n}}$. Indeed, writing $C=A_f$ for some smooth real function $f: S^{2n-1}\rightarrow \R$ which is $C^3$-close to the constant 1 and denoting by $\rho:S^{2n-1} \rightarrow \partial C$ the radial projection, we have that the contact form $\rho^* \alpha_C = f^2 \alpha_0$ is $C^3$-close to $\alpha_0$. This implies that the respective Reeb vector fields are $C^2$ close. By a standard argument involving Gronwall's lemma, $\rho^{-1}\circ \zeta$, which is a $\pi$-periodic closed Reeb orbit of $\rho^*\alpha_C$, is $C^3$-close to $\zeta_0(t)= e^{2it} z$ with $z=\rho^{-1}(\zeta(0))$. By applying the map $\rho: S^{2n-1} \rightarrow \partial C$, which is $C^3$-close to the inclusion $S^{2n-1} \hookrightarrow \C^n$, we conclude that $\zeta$ is $C^3$-close to $\zeta_0$, as claimed.

Up to applying a unitary automorphism of $\C^n$, we may assume that $z=(1,0,\dots,0)$, so that $\zeta(t) = \gamma_0(t/\pi)$, where $\gamma_0$ is the curve defined in (\ref{gamma0}). Thanks to Lemma \ref{spostami}, there is a symplectomorphism $\varphi_C: \C^n \rightarrow \C^n$ mapping $\zeta(\R/\pi \Z)$ to $\gamma_0(\R/\Z)$ which is $C^2$-close to the identity.

We now fix the $C^3$-neighborhood $\mathscr{B}$ of $B^{2n}$ in such a way that every $C\in \mathscr{B}$ has $C^2$-distance smaller than $\epsilon$ from $B^{2n}$ and  the symplectomorphism $\varphi_C$ satisfies $\|\varphi_C-\mathrm{id}\|_{C^2} < \epsilon$. By the above choice of $\epsilon$, for every $C\in \mathscr{B}$ the set
$C':= \varphi_C(C)$ is a convex starshaped smooth domain such that $\gamma_0(\R/\Z)$ is a Reeb orbit of $\alpha_{C'}$ of minimal period $T_{\min}(\alpha_{C'})=\pi$. In particular,  the normal direction to $\partial C'$ at the point $\gamma_0(t)$ is the line  $i \gamma_0'(t) \R$, which is contained in $\C\times \{0\}$, so the tangent space of $\partial C'$ at $\gamma_0(t)$ has the form $\gamma_0'(t)\R \oplus \C^{n-1}$. By convexity, $C'$ is then contained in the cylinder $Z=B^2\times \C^{n-1}$. So we have
\[
c_{\mathrm{out}}(C') \leq \pi  =  T_{\min}(\alpha_{C'}) = c_{\mathrm{EHZ}} (C'),
\]
which implies (\ref{equ}) for $C'$, as $c_{\mathrm{out}}\geq c_{\mathrm{EHZ}}$ is always true. From the fact that $C$ is symplectomorphic to $C'$, we deduce that (\ref{equ}) holds also for $C$.
\end{proof}

We conclude this section by proving Proposition \ref{symplectic} from the Introduction, namely the fact that smooth convex domains that are $C^3$-close enough to the ball $B^{2n}$ and whose Reeb flow on the boundary is Zoll are symplectomorphic to balls.

\begin{proof}[Proof of Proposition \ref{symplectic}]
Consider a starshaped domain $C\subset \R^{2n}=\C^n$ with $\alpha_C=\lambda_0|_{\partial C}$ Zoll in a $C^3$-neighborhood $\mathscr{B}$ of $B^{2n}$, whose size will be determined along the proof. Up to rescaling, it is enough to consider the case in which the Reeb orbits of $\alpha_C$ have period $\pi$. If $C=A_f$ is $C^3$-close to the ball $B^{2n}=A_1$, then the contact form
\[
\rho^*(\alpha_C) = f^2 \alpha_0
\]
on $S^{2n-1}$ is $C^3$-close to the standard contact form $\alpha_0= \alpha_{B^{2n}}$, all of whose orbits are closed with period $\pi$. Here, $\rho: S^{2n-1} \rightarrow \partial C$ denotes the radial projection $z \mapsto f(z)z$. 

Since $\alpha_C$ is Zoll with all orbits of period $\pi$, so is $\rho^*(\alpha_C)$. Proposition \ref{rigid-zoll} implies that if $\mathscr{B}$ is small enough then there exists a diffeomorphism $v: S^{2n-1} \rightarrow S^{2n-1}$ that is $C^1$-close to the identity and satisfies
\[
v^* (\rho^*(\alpha_C)) = \alpha_0. 
\]
The diffeomorphism $\psi:= \rho\circ v: S^{2n-1} \rightarrow \partial C$ is $C^1$-close to the inclusion $S^{2n-1} \hookrightarrow \C^n$ and satisfies
\[
\psi^*( \lambda_0|_{\partial C}) = \lambda_0|_{S^{2n-1}}.
\]
In particular, up to reducing if necessary the size of $\mathscr{B}$, we may assume that
\begin{equation}
\label{psivici}
\left| \frac{z+\psi(z)}{2} \right| \geq \frac{1}{2} \qquad \forall z\in S^{2n-1}.
\end{equation}
We now extend $\psi$ to a positively 1-homogeneous map on the whole $\C^{n}$ by mapping $rz$ with $r>0$ and $z\in S^{2n-1}$ to $r\psi(z)$. This extension is still denoted by
\[
\psi: \C^n \rightarrow \C^n.
\]
It is continuous on $\C^n$, smooth on $\C^n \setminus \{0\}$, maps $B^{2n}$ onto $C$ and satisfies
\[
\psi^* \lambda_0 = \lambda_0 \qquad \mbox{on } \C^n \setminus \{0\}.
\]
In particular, it is a symplectomorphism of $\C^n \setminus \{0\}$ onto itself. In order to conclude, we just need to smoothen it near the origin, by keeping it a symplectomorphism. Such a smoothing can be performed using generating functions, as we shall now explain.

The graph of the symplectomorphism $\psi|_{\C^n\setminus \{0\}}$ is a Lagrangian submanifold of $(\C^n \times \C^n, \omega_0 \oplus -\omega_0)$, and hence gets mapped into a Lagrangian submanifold of $T^* \C^n$ by the map $\Phi$ from (\ref{linearsymp}). The map $\di\psi$ is $0$-homogeneous and is arbitrarily $C^0$-close to the identity, provided that $\mathscr{B}$ is sufficiently small. This implies that if $\mathscr{B}$ is small enough the graph of the symplectomorphism $\psi|_{\C^n\setminus \{0\}}$ is mapped to the image of a positively 1-homogeneous Lagrangian section of $T^* (\C^n \setminus \{0\})$, that is, to the graph of a positively 1-homogeneous closed one-form on $\C^n\setminus \{0\}$. Since $\C^n\setminus \{0\}$ is simply connected (we are assuming that $n>1$, because this proposition is trivially true for $n=1$), this one-form is the differential of a positively 2-homogeneous smooth function $S: \C^n \setminus \{0\} \rightarrow \R$. From (\ref{linearsymp}) we deduce that the symplectomorphism $\psi$ satisfies
\[
i ( z - \psi(z)) = \nabla S \left( \frac{z+\psi(z)}{2} \right) \qquad \forall z\in \C^n\setminus \{0\}.
\]
The Hessian of $S$ is positively $0$-homogeneous and is $C^0$-small, see (\ref{boundoS}). The function $S$ extends continuously to the origin by setting $S(0)=0$, but this extension is in general not smooth. In order to smoothen it, choose a smooth function $\sigma: [0,+\infty) \rightarrow [0,1]$ such that $\sigma(r)=0$ for all $r$ sufficiently small and  $\sigma(r)=1$ for every $r\geq 1/2$. We then define a smooth real function $\tilde{S}$ on $\C^n$ by
\[
\tilde{S}(z) := \sigma(|z|) S(z) \qquad \forall z\in \C^n.
\]
The Hessian $\nabla^2 \tilde{S}$ of $\tilde{S}$ is $C^0$-small when the one of $S$ is $C^0$-small, so up to reducing the size of $\mathscr{B}$ we can assume that
\[
\| \nabla^2 \tilde{S}\|_{C^0} < 2.
\]
As discussed at the beginning of this Section, this implies that the identity
\[
i ( z - \varphi(z)) = \nabla \tilde{S} \left( \frac{z+\varphi(z)}{2} \right) \qquad \forall z\in \C^n,
\]
defines a smooth symplectorphism $\varphi: \C^n \rightarrow \C^n$. Since $\nabla \tilde{S}(z) = \nabla S(z)$ for every $z\in \C^n$ with $|z|\geq 1/2$, inequality (\ref{psivici}) implies that $\varphi(z)=\psi(z)$ for every $z\in S^{2n-1}$. We conclude that $\varphi$ is a symplectomorphism of $\C^n$ mapping $B^{2n}$ onto $C$.
\end{proof}

\section{Shadows of symplectic balls}
\label{shadowssec}

In this section, we wish to prove Corollary \ref{cor3} from the Introduction. Before starting with the proof, we need to discuss the linear symplectic non-squeezing theorem. 

The vector space $\R^{2n}$ is endowed with the standard symplectic form $\omega_0$, with the standard Euclidean product and with the standard complex structure, which is $\omega_0$-compatible and allows us to identify $\R^{2n}$ with $\C^n$. If $0\leq k \leq n$, we have the inclusions
\[
\mathrm{Gr}_{k}(\C^n) \subset \mathrm{Gr}_{2k}(\R^{2n},\omega_0) \subset  \mathrm{Gr}_{2k}(\R^{2n})
\]
of the Grassmannian of complex $k$-subspaces into the Grassmannian of symplectic $2k$-subspaces, and of the latter into the Grassmannian of all real $2k$-subspaces. The smallest and the largest Grassmannians are compact, while the symplectic Grassmannian is an open neighborhood of $\mathrm{Gr}_{k}(\C^n)$ in  $\mathrm{Gr}_{2k}(\R^{2n})$.

The Wirtinger inequality states that
\[
|\omega_0^k[v_1,\dots,v_{2k}]| \leq k! \, |v_1 \wedge \dots \wedge v_{2k}|
\]
for all $2k$-uples of vectors in $\R^{2n}$ and, in the case of linearly independent vectors, the equality holds if and only if the vectors $v_1,\dots, v_{2k}$ span a complex subspace. Therefore, the formula
\begin{equation}
\label{wirtfun}
w(V) := \frac{|\omega_0^k[v_1,\dots,v_{2k}]|}{k! \,|v_1 \wedge \dots \wedge v_{2k}|},
\end{equation}
where $v_1,\dots,v_{2k}$ denotes a basis of $V$, defines a non-negative function on $\mathrm{Gr}_{2k}(\R^{2n})$ which is strictly positive precisely on $\mathrm{Gr}_{2k}(\R^{2n},\omega_0)$ and achieves its maximum 1 precisely at $\mathrm{Gr}_{k}(\C^n)$: For every $V\in \mathrm{Gr}_{2k}(\R^{2n})$ there holds
\begin{equation}
\label{wirtinger}
\begin{split}
w(V) &\geq 0, \mbox{ with equality if and only if } V\notin \mathrm{Gr}_{2k}(\R^{2n},\omega_0), \\
w(V) &\leq 1, \mbox{ with equality if and only if }  V\in \mathrm{Gr}_{k}(\C^n).
\end{split}
\end{equation}
The function $w$ is invariant under unitary transformations.

Given $V\in \mathrm{Gr}_{2k}(\R^{2n},\omega_0)$, we denote by $P_V$ the linear projection onto $V$ along the symplectic orthogonal of $V$. The linear symplectic non-squeezing theorem can be stated in the following way, where $B^{2n}$ denotes the unit ball in $\R^{2n}$.

\begin{thm}
\label{linnonsque}
For every element $V\in \mathrm{Gr}_{2k}(\R^{2n},\omega_0)$ and every linear symplectomorphism $\Phi: \R^{2n} \rightarrow \R^{2n}$, we have
\begin{equation}
\label{volombra}
\mathrm{vol}(P_V \Phi(B^{2n}),\omega_0^k|_V) = \frac{\pi^k}{w(\Phi^{-1}(V))}.
\end{equation}
In particular:
\begin{enumerate}[(i)]
\item $\mathrm{vol}(P_V \Phi(B^{2n}),\omega^k_0|_V)\geq \pi^k$ with equality if and only if $\Phi^{-1}(V) \in \mathrm{Gr}_{k}(\C^n)$. In the equality case, the set $P_V \Phi(B^{2n})$ is linearly symplectomorphic to a $2k$-ball of radius 1:
\begin{equation}
\label{casecplx}
P_V \Phi(B^{2n}) = \Phi(B^{2n} \cap \Phi^{-1} (V)).
\end{equation}
\item The function $V\mapsto \mathrm{vol}(P_V \Phi(B^{2n}),\omega_0^k|_V)$ is coercive on $\mathrm{Gr}_{2k}(\R^{2n},\omega_0)$.
\end{enumerate}
\end{thm}

The proof of the above theorem can be obtained by an easy modification of the proof of \cite[Theorem 1]{am13}, but for the reader's convenience we include a full proof at the end of this section. See also \cite{ddgp19} for another approach to the linear non-squeezing theorem: There, statement (i) above is proven by showing that $P_V \Phi(B^{2n})$ always contains a symplectic $2k$-ball of radius 1, by using the Williamson symplectic diagonalization and Schur complements.  

We can now proceed with the proof of Corollary \ref{cor3}.

\begin{proof}[Proof of Corollary \ref{cor3}]
We use the notation $\mathrm{Symp}(\R^{2n})$ for the space of (nonlinear) symplectomorphisms $\varphi:(\R^{2n},\omega_0)\to(\R^{2n},\omega_0)$ and consider the function
\[
f: \mathrm{Symp}(\R^{2n} ) \times \mathrm{Gr}_{2k}(\R^{2n},\omega_0) \to (0,+\infty),\qquad f(\varphi,V):=\mathrm{vol}(P_V (\varphi(B^{2n})), \omega_0^k|_V).
\]
This function is continuous with respect to the $C^0_{\mathrm{loc}}$-topology on $\mathrm{Symp}(\R^{2n})$ and the standard topology of $\mathrm{Gr}_{2k}(\R^{2n},\omega_0)$.

We fix a linear symplectomorphism $\Phi: \R^{2n} \rightarrow \R^{2n}$. In order to prove Corollary \ref{cor3}, it suffices to find a $C^3_{\mathrm{loc}}$-neighborhood $\mathscr W_0$ of $\Phi$ so that
\begin{equation}\label{e:cor3}
f(\varphi,V)\geq \pi^k,\qquad \forall\,(\varphi,V)\in \mathscr{W}_0 \times \mathrm{Gr}_{2k}(\R^{2n},\omega_0).
\end{equation}
Denote by 
\[
\mathscr{G}_0 := \Phi \bigl( \mathrm{Gr}_k(\C^n) \bigr)
\]
the set of $V\in  \mathrm{Gr}_{2k}(\R^{2n},\omega_0)$ such that $\Phi^{-1}(V)$ is a complex linear subspace. By the compactness of the complex Grassmannian $ \mathrm{Gr}_{k}(\C^n)$, $\mathscr{G}_0$ is a compact subset of $\mathrm{Gr}_{2k}(\R^{2n},\omega_0)$.

If $V$ belongs to $\mathscr{G}_0$, then we have the identity
\[
P_V \Phi (B^{2n}) = \Phi( B^{2n} \cap \Phi^{-1}(V))
\]
by statement (i) of Theorem \ref{linnonsque}, so composing $\Phi$ with a unitary map $U_V : \R^{2k} \rightarrow \Phi^{-1} (V)$ we obtain a linear symplectomorphism 
$\Psi_V : \R^{2k} \rightarrow V$ such that
\[
\Psi_V(B^{2k})=P_V \Phi (B^{2n}).
\]
Let now $\mathscr{B}$ be the $C^3$-neighborhood of $B^{2k}$ in the space of smooth convex bodies in $\R^{2k}$ given by Corollary \ref{main2} from the Introduction. The set $\mathscr B$ has the property that
\begin{equation}
\label{quaeB}
c_{\mathrm{EHZ}}(C)^k \leq \mathrm{vol}(C,\omega_0^k) \qquad \forall C\in \mathscr{B}.
\end{equation}
For all $V_0 \in\mathscr G_0$ there holds 
\[
\Psi_{V_0}^{-1} P_{V_0} \Phi (B^{2n}) =B^{2k},
\]
and hence there exists an open neighborhood $\mathscr V_{V_0}$ of $V_0$ in $\mathrm{Gr}_{2k}(\R^{2n},\omega_0)$ and a $C^3_\mathrm{loc}$-neighborhood $\mathscr W_{V_0}$ of $\Phi$ in $\mathrm{Symp}(\R^{2n})$ such that
\begin{equation}
\label{zeroconcl}
\Psi_{V_0}^{-1}  P_V (\varphi(B^{2n}))  \in \mathscr{B} \qquad \forall\, (\varphi,V) \in \mathscr{W}_{V_0} \times \mathscr{V}_{V_0}.
\end{equation}
By the compactness of $\mathscr G_0$ there are finitely many $V_1,\ldots,V_N$ in $\mathscr G_0$ such that the set 
\[
\mathscr{V}:= \bigcup_{i=1}^N \mathscr{V}_{V_i}
\]
is an open neighborhood of $\mathscr{G}_0$. If $\mathscr{W}_0$ is any $C^3_{\mathrm{loc}}$-neighborhood of $\Phi$ in $\mathrm{Symp}(\R^{2n})$ which is contained in $\bigcap_{i=1}^N \mathscr{W}_{V_i}$, then by (\ref{zeroconcl}) we obtain the following statement: For every $(\varphi,V)$ in $\mathscr{W}_0\times \mathscr{V}$ there exists a linear symplectomorphism $\Psi: \R^{2k} \rightarrow V$ such that
\begin{equation}
\label{primaconcl}
\Psi^{-1} P_V (\varphi(B^{2n}))  \in \mathscr{B}. 
\end{equation}
If $(\varphi,V)$ belongs to $\mathscr{W}_0\times \mathscr{V}$ then we find, thanks to (\ref{quaeB}), (\ref{primaconcl}) and the fact that both the volume and the EHZ-capacity are invariant by the symplectomorphism $\Psi$, the inequality
\[
f(\varphi,V)=\mathrm{vol}(P_V(\varphi(B^{2n})),\omega_0^k|_V) \geq c_{\mathrm{EHZ}}(P_V(\varphi(B^{2n})) )^k.
\]
Now we can use the fact that the EHZ-capacity of the linear symplectic projection of a bounded convex domain $C$ is not smaller than the EHZ-capacity of $C$, see e.g.\ \cite[Theorem 4.1 (v)]{ama15}, and we obtain 
\[
c_{\mathrm{EHZ}}(P_V(\varphi(B^{2n})) )^k  \geq c_{\mathrm{EHZ}}(\varphi(B^{2n}))^k = c_{\mathrm{EHZ}}(B^{2n})^k = \pi^k.
\]
Putting the last two inequalities together, we have shown the inequality in \eqref{e:cor3} on $\mathscr W_0 \times \mathscr V$:
\begin{equation}\label{e:cor31}
f(\varphi,V)\geq \pi^k,\qquad \forall\,(\varphi,V)\in  \mathscr W_0 \times \mathscr V .
\end{equation}
Let us now consider the set
\[
\widehat{\mathscr V}:=\{V\in \mathrm{Gr}_{2k}(\R^{2n},\omega_0)\ |\ f(\Phi,V)\leq (2\pi)^k \},
\]
which is compact thanks to statement (ii) in Proposition \ref{linnonsque}. Let us shrink $\mathscr W_0$ so that the implication
\begin{equation*}
\varphi\in\mathscr W_0\quad\Longrightarrow\quad \varphi(B^{2n})\supset \frac{1}{2}\Phi(B^{2n})
\end{equation*}
holds. If $(\varphi,V)\in \mathscr W_0 \times \widehat{\mathscr V}^c$, then this implication yields
\[
\begin{split}
f(\varphi,V) &=\mathrm{vol}(P_V(\varphi(B^{2n})),\omega_0^k|_V) \geq \mathrm{vol}(P_V(\tfrac{1}{2}\Phi(B^{2n})),\omega_0^k|_V) =\mathrm{vol}(\tfrac{1}{2}P_V(\Phi(B^{2n})),\omega_0^k|_V)\\ &=\frac{1}{2^k}\mathrm{vol}(P_V(\Phi(B^{2n})),\omega_0^k|_V)= \frac{1}{2^k} f(\Phi,V) >\pi^k.
\end{split}
\]
Thus, we have shown the inequality in \eqref{e:cor3} on $\mathscr W_0\times\widehat{\mathscr V}^c$:
\begin{equation}\label{e:cor32}
f(\varphi,V)> \pi^k,\qquad \forall\,(\varphi,V)\in  \mathscr W_0 \times \widehat{\mathscr V}^c .
\end{equation}
Since $\widehat {\mathscr V}\setminus\mathscr V$ is compact and $f>\pi^k$ on $\{\Phi\} \times ( \widehat {\mathscr V}\setminus\mathscr V)$, up to shrinking the neighborhood $\mathscr{W}_0$ of $\Phi$ we may assume that
\begin{equation}\label{e:cor33}
f(\varphi,V)> \pi^k,\qquad \forall\,(\varphi,V)\in  \mathscr W_0 \times (\widehat{\mathscr V}\setminus\mathscr V) .
\end{equation}
Inequalities \eqref{e:cor31}, \eqref{e:cor32} and \eqref{e:cor33} yield the desired lower bound \eqref{e:cor3}. This concludes  the proof of Corollary \ref{cor3} from the Introduction.
\end{proof}

We end this section by proving Theorem \ref{linnonsque}.

\begin{proof}[Proof of Theorem \ref{linnonsque}]
We first consider the special instance in which $V$ is a complex subspace. In this case, the symplectic projector $P_V$ is orthogonal, and hence symmetric. We denote by $A$ the surjective linear map
\[
A:= P_V \Phi : \R^{2n} \rightarrow V.
\]
Then 
\begin{equation}
\label{Aperp}
P_V \Phi(B^{2n})  = A(B^{2n}) = A(B^{2n}\cap (\ker A)^{\perp}),
\end{equation}
where $\perp$ denotes the Euclidean orthogonal complement, and the Euclidean volume $\mathrm{vol}_{2k}$ of  this set can be expressed by the formula
\begin{equation}
\label{volsur}
\mathrm{vol}_{2k}(A(B^{2n})) = \omega_{2k} \sqrt{ \det(A^T A|_{(\ker A)^{\perp}})},
\end{equation}
where $A^T: V \rightarrow \R^{2n}$ denotes the transpose of $A$ with respect to the Euclidean product and $\omega_{2k} = \pi^k/k!$ is the volume of the unit $2k$-ball.

Note that, denoting by $J$ the standard complex structure of $\R^{2n}$ and using the fact that $V$ is complex and $P_V$ is symmetric, we have
\begin{equation}
\label{nucleo}
(\ker A)^{\perp} = A^T (V)= \Phi^T P_V (V) = \Phi^T (V)= \Phi^T J (V) = J \Phi^{-1} (V),
\end{equation}
where the last equality follows from the fact that the automorphism $\Phi$ is symplectic. Let $v_1,\dots,v_{2k}$ be a basis of $(\ker A)^{\perp}$ with
\[
|v_1 \wedge \dots \wedge v_{2k} |=1.
\]
By (\ref{volsur}) we have the chain of identities
\begin{equation}
\label{chainide}
\begin{split}
\left( \frac{\mathrm{vol}_{2k}(A(B^{2n}))}{\omega_{2k}} \right)^2 &= \det (A^T A|_{(\ker A)^{\perp}}) = | A^T A v_1 \wedge \dots \wedge A^T A v_{2k} | \\ &= \frac{1}{k! \, w((\ker A)^{\perp})} |\omega^k_0[ A^TA v_1,\dots,A^T A v_{2k}]|,
\end{split}
\end{equation}
From (\ref{nucleo}) and from the fact that $J$ is unitary we obtain
\begin{equation}
\label{e4w}
w((\ker A)^{\perp}) = w( J \Phi^{-1} (V)) = w( \Phi^{-1} (V)).
\end{equation}
Moreover, since $P_V$ is symmetric, there holds
\[
A^T A  = \Phi^T P_V \Phi = \Phi^T A,
\]
and hence, using the fact that $\Phi^T$ is symplectic and $V=A(\R^{2n})$ is complex, we obtain
\begin{equation}
\label{svil}
\begin{split}
 |\omega_0^k &[ A^TA v_1,\dots,A^T A v_{2k}]| =  |\omega^k_0[ \Phi^TA v_1,\dots,\Phi^T A v_{2k}]| =  |\omega^k_0[ A v_1,\dots, A v_{2k}]| \\ &= k! \, |A v_1 \wedge \dots \wedge A v_{2k} | = \frac{k!}{\omega_{2k}} \mathrm{vol}_{2k}(A(B^{2n}\cap (\ker A)^{\perp})) = \frac{k!}{\omega_{2k}} \mathrm{vol}_{2k}(A(B^{2n})), 
 \end{split}
\end{equation}
where in the last identity we have used (\ref{Aperp}). The identities (\ref{chainide}), (\ref{e4w}) and (\ref{svil}) give us the following formula for the Euclidean volume of $A(B^{2n})$:
\[
\mathrm{vol}_{2k}(A(B^{2n})) = \frac{\omega_{2k}}{w(\Phi^{-1}(V))} = \frac{\pi^k}{k! \, w(\Phi^{-1}(V))}.
\]
As $V$ is complex, we deduce the desired identity for the $\omega^k_0$-volume of $P_V \Phi(B^{2n}) = A(B^{2n})$:
\[
\mathrm{vol}(P_V \Phi(B^{2n}),\omega_0^k|_V) = k! \, \mathrm{vol}_{2k}(A(B^{2n})) = \frac{\pi^k}{w(\Phi^{-1}(V))}.
\]
The case of a general symplectic subspace $V\in \mathrm{Gr}_{2k}(\R^{2n},\omega_0)$ can be deduced from the above case as follows. Choose an $\omega_0$-compatible scalar product on $\R^{2n}$ such that the projector $P_V$ is orthogonal, and denote by $\tilde{B}^{2n}$ and $\tilde{J}$ the corresponding unit ball and $\omega_0$-compatible complex structure, which satisfies $\tilde{J}(V)=V$. Let $\Psi: (\R^{2n},\omega_0,\tilde{J}) \rightarrow (\R^{2n},\omega_0,J)$ be a symplectic and complex linear isomorphism. Then $\Psi$ is unitary from $(\R^{2n},\tilde{J})$ to $(\R^{2n},J)$, and hence $\Psi(\tilde{B}^{2n})=B^{2n}$. By applying (\ref{volombra}) to the complex subspace $V$ of $(\R^{2n},\tilde{J})$ and to the linear symplectomorphism $\Phi\Psi$ we obtain
\[
\mathrm{vol}(P_V \Phi(B^{2n}),\omega^k_0|_V) = \mathrm{vol}(P_V \Phi\Psi(\tilde{B}^{2n}),\omega^k_0|_V) = \frac{\pi^k}{\tilde{w} ( \Psi^{-1} \Phi^{-1} (V))} =  \frac{\pi^k}{w (\Phi^{-1} (V))},
\]
where $\tilde{w}$ denotes the function (\ref{wirtfun}) on the symplectic and complex vector space $(\R^{2n},\omega_0,\tilde{J})$ and in the last equality we have used again the fact that $\Psi$ is unitary. This proves the identity (\ref{volombra}) in general. 

The first part of statement (i) and statement (ii) are now immediate consequences of this identity and (\ref{wirtinger}). There remains to show that if $\Phi^{-1}(V)$ is a complex linear subspace, then identity (\ref{casecplx}) holds. This identity can be deduced from (\ref{Aperp}) by the following chain of equalities:
\[
\begin{split}
P_V \Phi(B^{2n}) &= \Phi ( B^{2n} \cap ( \ker P_V \Phi)^{\perp} ) = \Phi (B^{2n} \cap (\Phi^{-1} (\ker P_V))^{\perp}) = \Phi( B^{2n} \cap \Phi^T ((\ker P_V)^{\perp}) ) \\ &=  \Phi( B^{2n} \cap  \Phi^T J (V))=  \Phi( B^{2n} \cap J \Phi^{-1} (V)) = \Phi( B^{2n} \cap \Phi^{-1} (V)).
\end{split}
\]
Here, the fact that the subspace $\Phi^{-1}(V)$ is complex has been used in the last equality.
\end{proof}

\appendix

\section{Appendix: Estimates for differential forms}\label{appbound}
In this appendix we exhibit the proofs of Lemma \ref{boundsonforms} and \ref{l:estreeb}.
\subsection{Proof of Lemma \ref{boundsonforms}}
Let $B$, $B'$ and $B''$  be open balls in $\R^d$ such that 
\[
\overline{B''} \subset B' \subset \overline{B'} \subset B
\] 
and let
	\[
	\varphi_i : B \rightarrow U_i \subset M,  \qquad i=1,\dots,N, 
	\]
	be diffeomorphisms such that $\varphi_i$ and $\varphi_i^{-1}$ have bounded derivatives of every order and the open sets $U_i'':= \varphi_i(B'')$ cover $M$. Set $U'_i:= \varphi_i(B')$. Since $U_i'$ and $U_i''$ have compact closure in $U_i$ and $U_i'$, respectively,  
	we can find a positive number $r$ such that any map $u: M \rightarrow M$ with $\mathrm{dist}_{C^0}(u,\mathrm{id})<r$ satisfies
	\[
	u(U_i') \subset U_i, \qquad u^{-1}(U_i'') \subset U_i'.
	\] 
	Fix a smooth partition of unity $\{\rho_i\}_{i=1,\dots,N}$ subordinated to the open cover $\{U_i''\}_{i=1,\dots,N}$. Let $\alpha\in \Omega^j(M)$ and let $u: M \rightarrow M$ be a smooth map with $\mathrm{dist}_{C^0}(u,\mathrm{id})<r$. Then the $i$-th summand in
	\[
	u^*\alpha = \sum_{i=1}^N u^*(\rho_i \alpha)
	\]
	satisfies
	\[
	\mathrm{supp} \, u^* (\rho_i \alpha) \subset u^{-1} ( \mathrm{supp}\, \rho_i \alpha) \subset u^{-1}(U''_i) \subset U_i'.
	\]
	By means of the coordinate system $\varphi_i$, $\rho_i \alpha$ can be seen as a smooth $j$-form $\beta_i$ on $\R^d$ supported in $B''$ and the restriction of $u$ to $U_i'$ as a smooth map $v_i: B' \rightarrow B$ with bounded derivatives of every order such that $v_i^* \beta_i$ is compactly supported in $B'$.
	
	This localization argument allows us to reduce the proof of Lemma \ref{boundsonforms} to the following statement: For every $\beta\in \Omega^j(\R^d)$ with compact support and every smooth map $v: B'\rightarrow \R^d$ with bounded derivatives of every order we have
	\begin{eqnarray}
	\label{bdof1loc}
	\|v^* \beta\|_{C^k(B')} & \lesssim &  \|\beta\|_{C^k} \|\di v\|_{C^k(B')}^j ( 1 + \|\di v\|_{C^{k-1}(B')}^{k}), \\
	\label{bdof2loc}
	\|v^* \beta - \beta\|_{C^k(B')} & \lesssim & \|\beta\|_{C^{k+1}} \|v-\mathrm{id}\|_{C^{k+1}(B')} ( 1 + \|\di v\|_{C^k(B')}^{k+j}),
	\end{eqnarray}
	for every $k\geq 0$, where for $k=0$ the undefined term $\|\di v\|_{C^{k-1}(B')}$ in (\ref{bdof1loc}) is set to be zero. Indeed, there holds
	\[
	\begin{split}
	\|u^* \alpha\|_{C^k} &\leq \sum_{i=1}^N \|u^* (\rho_i \alpha)\|_{C^k} =  \sum_{i=1}^N \|u^* (\rho_i \alpha)\|_{C^k(U_i')}  \lesssim \sum_{i=1}^N \|v_i^* \beta_i\|_{C^k(B')} \\ &\!\!\!\stackrel{\eqref{bdof1loc}}{\lesssim} \sum_{i=1}^N \|\beta_i\|_{C^k} \|\di v_i\|_{C^k(B')}^j ( 1 + \|\di v_i\|_{C^{k-1}(B')}^{k}) \\ &\lesssim \sum_{i=1}^N \|\rho_i \alpha\|_{C^k} \|\di u\|_{C^k(U_i')}^j ( 1 + \|\di u\|_{C^{k-1}(U_i')}^{k}) \\ &\leq  \|\di u\|_{C^k(M)}^j ( 1 + \|\di u\|_{C^{k-1}(M)}^{k}) \sum_{i=1}^N \|\rho_i \alpha\|_{C^k},
	\end{split}
	\]
	and inequality (\ref{bdof1}) in the statement of Lemma \ref{boundsonforms} follows from the fact that the quantity $\sum_{i=1}^N \|\rho_i \alpha\|_{C^k}$ is a norm on $\Omega^j(M)$ that is equivalent to $\|\alpha\|_{C^k}$. Similarly, we get
	\[
	\begin{split}
	\|u^* \alpha - \alpha\|_{C^k} &\leq \sum_{i=1}^N \|u^* (\rho_i \alpha) - \rho_i \alpha\|_{C^k(U_i')} \lesssim \sum_{i=1}^N \|v_i^* \beta_i - \beta_i\|_{C^k(B')} \\ &\!\!\!\stackrel{\eqref{bdof2loc}}{\lesssim} \sum_{i=1}^N \|\beta_i\|_{C^{k+1}} \|v_i-\mathrm{id}\|_{C^{k+1}(B')}  ( 1 + \|\di v_i\|_{C^{k}(B')}^{j+k}) \\ &\lesssim \sum_{i=1}^N \|\rho_i \alpha\|_{C^{k+1}} \mathrm{dist}_{C^{k+1}(U_i')} (u,\mathrm{id})  ( 1 + \|\di u\|_{C^{k}(U_i')}^{j+k})\\ & \leq  \mathrm{dist}_{C^{k+1}(M)} (u,\mathrm{id}) ( 1 + \|\di u\|_{C^{k}(M)}^{j+k}) \sum_{i=1}^N \|\rho_i \alpha\|_{C^{k+1}} \\ &\lesssim \mathrm{dist}_{C^{k+1}(M)} (u,\mathrm{id}) ( 1 + \|\di u\|_{C^{k}(M)}^{j+k})  \|\alpha\|_{C^{k+1}} ,
	\end{split}
	\]
	proving inequality (\ref{bdof2}) in the statement of Lemma \ref{boundsonforms}.
	
	There remains to prove (\ref{bdof1loc}) and (\ref{bdof2loc}). We first deal with (\ref{bdof1loc}) in the case $j=0$, i.e.\ $\beta: \R^d \rightarrow \R$ is a compactly supported smooth real function, and argue inductively on $k$. In this case, (\ref{bdof1loc}) holds trivially for $k=0$, and we assume that it holds for a certain integer $k\geq 0$. We denote the standard basis of $\R^d$ by $\{e_j\}_{j=1,\dots,d}$ and multi-indices and partial derivatives by
	\[
	p = \sum_{j=1}^d p_j e_j, \qquad |p|:= \sum_{j=1}^d p_j, \qquad \partial^p = \partial_{x_1}^{p_1} \cdots \partial_{x_d}^{p_d},
	\]
	where the $p_j$'s are non-negative integers. If $|p|=k+1$, then we can write $p=q+e_j$ with $|q|=k$ and find
	\[
	\partial^p (v^* \beta) = \partial^q \partial_{x_j} (\beta\circ v) = \partial^q \sum_{i=1}^d ((\partial_{x_i} \beta)\circ v) \, \partial_{x_j} v_i = \sum_{i=1}^d \sum_{r+s=q} \binom{q}{r} \partial^r ((\partial_{x_i} \beta)\circ v)\, \partial^s \partial_{x_j} v_i,
	\]
	where the generalized binomial coefficient $\binom{q}{r}$ is the product of the binomial coefficients $\binom{q_i}{r_i}$ and 
	$v_i$ denotes the $i$-th component of $v$. From this identity and from the inductive assumption applied to the functions $\partial_{x_i} \beta$ we obtain
	\[
	\begin{split}
	\|\partial^p (v^* \beta)\|_{C^0(B')} &\leq \sum_{i=1}^d  2^k \|\partial_{x_i} \beta \circ v\|_{C^k(B')} \|\di v\|_{C^k(B')} \\ &\lesssim \sum_{i=1}^d \|\partial_{x_i} \beta\|_{C^k} ( 1 + \|\di v\|_{C^{k-1}(B')}^k ) \|\di v\|_{C^k(B')} \\ &\leq \|\beta\|_{C^{k+1}} ( \|\di v\|_{C^k(B')}  + \|\di v\|_{C^k(B')}^{k+1}).
	\end{split}
	\]
	Using again the inductive assumption, we deduce that the $C^{k+1}$-norm of $v^* \beta$ has the upper bound 
	\[
	\begin{split}
	\|v^* \beta\|_{C^{k+1}(B')} &= \|v^* \beta\|_{C^k(B')} + \sum_{|p|=k+1} \|\partial^p (v^* \beta)\|_{C^0(B')} \\ &\lesssim \|\beta\|_{C^k} ( 1 + \|\di v\|_{C^{k-1}(B')}^k ) +   \|\beta\|_{C^{k+1}} ( \|\di v\|_{C^k(B')}  + \|\di v\|_{C^k(B')}^{k+1}) \\ &\leq \|\beta\|_{C^{k+1}} ( 1 + \|\di v\|_{C^k(B')}^k + \|\di v\|_{C^k(B')} + \|\di v\|_{C^k(B')}^{k+1}) \\ &\lesssim   \|\beta\|_{C^{k+1}(B')} ( 1 + \|\di v\|_{C^k(B')}^{k+1}).
	\end{split}
	\]
	This concludes the proof of (\ref{bdof1loc}) for $j=0$. 
	
	The bound (\ref{bdof1loc}) for higher order forms follows from the case of functions by writing each smooth $j$-form as sum of the elementary $j$-forms  
	\[
	\beta = f \, \di x_{i_1} \wedge \cdots \wedge \di x_{i_j}
	\]
	and by using the identity
	\[
	v^* \beta = f\circ v \sum_{h\in \{1,\dots,d\}^j} (\partial_{x_{h_1}} v_{i_1}\cdot \ldots\cdot \partial_{x_{h_j}} v_{i_j}) \, \di x_{h_1} \wedge \cdots \wedge \di x_{h_j}.
	\]
	Now we prove the bound (\ref{bdof2loc}), starting again from the case of a function $\beta\in \Omega^0(\R^d)$. We have
	\begin{equation}
	\label{conve}
	\beta(v(x)) - \beta(x) = \sum_{i=1}^d (v_i(x)-x_i) \int_0^1 \partial_{x_i} \beta \big( t v(x) + (1-t) x\big)\, \di t.
	\end{equation}
	From (\ref{bdof1loc}) for $j=0$ we deduce
	\[
	\begin{split}
	\|\partial_{x_i} \beta( t v + (1-t)\mathrm{id}) \|_{C^k(B')} &\lesssim \|\partial_{x_i} \beta\|_{C^k} ( 1 + \|t \di v + (1-t) \mathrm{id} \|_{C^{k-1}(B')}^k) \\ &\lesssim \|\beta\|_{C^{k+1}} ( 1 + \|\di v\|_{C^{k-1}(B')}^k).
	\end{split}
	\]
	We can now estimate the $C^k$-norm of (\ref{conve}) using the above bound and the fact that the $C^k$ norm of a product is bounded by the product of the $C^k$ norms:
	\begin{equation}
	\label{meglio}
	\|\beta\circ v - \beta\|_{C^k(B')} \lesssim \|\beta\|_{C^{k+1}} \|v-\mathrm{id}\|_{C^k(B')} ( 1 + \|\di v\|_{C^{k-1}(B')}^k),
	\end{equation}
	which is a stronger version of (\ref{bdof2loc}) for $j=0$. If $j\geq 1$ and $\beta$ is the elementary $j$-form 
	\[
	\beta = f \, \di x_{i_1} \wedge \cdots \wedge \di x_{i_j},
	\]
	we have
	\begin{equation}
	\label{llaide}
	v^* \beta - \beta = (f\circ v - f) \di v_{i_1} \wedge \cdots \wedge \di v_{i_j} + f ( \di v_{i_1} \wedge \cdots \wedge \di v_{i_j} - \di x_{i_1} \wedge \cdots \wedge \di x_{i_j}) .
	\end{equation}
	By writing $v_i(x) = x_i + w_i(x)$ we can expand the term $\di v_{i_1} \wedge \cdots \wedge \di v_{i_j}$ and get the bound
	\begin{equation}
	\label{llaine}
	\|\di v_{i_1} \wedge \cdots \wedge \di v_{i_j} - \di x_{i_1} \wedge \cdots \wedge \di x_{i_j}\|_{C^k(B')} \lesssim \|\di v-\mathrm{id}\|_{C^k(B')} ( 1 + \|\di v - \mathrm{id}\|_{C^k(B')}^{j-1}).
	\end{equation}
	The identity (\ref{llaide}) and the estimate (\ref{llaine}), together with the bound (\ref{meglio}) applied to the function $f$, imply
	\[
	\begin{split}
	\|v^* \beta - \beta\|_{C^k(B')} &\leq \|f\circ v - f\|_{C^k(B')} \|\di v\|_{C^k(B')}^j \\ & \quad + \|f\|_{C^k} \| \di v_{i_1} \wedge \cdots \wedge \di v_{i_j} - \di x_{i_1} \wedge \cdots \wedge \di x_{i_j}\|_{C^k(B')} \\ 
	&\lesssim \|f\|_{C^{k+1}} \|v-\mathrm{id}\|_{C^k(B')} ( 1 + \|\di v\|_{C^{k-1}(B')}^k) \|\di v\|_{C^k(B')}^j \\ &\quad + \|f\|_{C^k} \|\di v-\mathrm{id}\|_{C^k(B')} ( 1 + \|\di v - \mathrm{id}\|_{C^k(B')}^{j-1}) \\ &
	\lesssim \|f\|_{C^{k+1}} \|v-\mathrm{id}\|_{C^{k+1}(B')} ( 1 + \|\di v\|_{C^k(B')}^{k+j}).
	\end{split}
	\]
	By adding up over all elementary forms we obtain (\ref{bdof2loc}).
\subsection{Proof of Lemma \ref{l:estreeb}}

Let  $J:\Lambda^{2n-2}M\to TM$ be the vector bundle isomorphism that is the inverse of the map
\[
X\mapsto\iota_X(\alpha_0\wedge \di\alpha_0^{n-1}),
\]
and consider the bundle map
\[
K : \Lambda^2 M \to TM, \qquad K(\beta) = J(\beta^{n-1}).
\]
If $\alpha$ is a contact form on $M$, then $\alpha \wedge \di\alpha^{n-1} = f\, \alpha_0 \wedge \di\alpha_0^{n-1}$ for some non-vanishing function $f\in \Omega^0(M)$, and we have
\[
\imath_{K(\di\alpha)} \bigl( \alpha \wedge \di\alpha^{n-1} \bigr) = f \,\imath_{K(\di\alpha)} \bigl( \alpha_0 \wedge \di\alpha_0^{n-1} \bigr) = f\, \di\alpha^{n-1}.
\] 
From the identity 
\[
\imath_{R_{\alpha}} \bigl( \alpha \wedge \di\alpha^{n-1} \bigr) = \di\alpha^{n-1}
\]
we then obtain that the non-vanishing vector field $K(\di\alpha)$ is parallel to $R_{\alpha}$ and hence we find the following formula for the Reeb vector field of $\alpha$:
\begin{equation}
\label{formR}
R_{\alpha} = \frac{K(\di\alpha)}{\alpha(K(\di\alpha))}.
\end{equation}
Since $K(\di \alpha_0) = R_{\alpha_0}$, for every integer $k\geq0$ the map $K$ satisfies
\begin{equation}
\label{aaa0}
\|K(\di \alpha) - R_{\alpha_0}\|_{C^k} = \|K(\di \alpha) - K(\di \alpha_0)\|_{C^k} \leq \omega_k( \|\di \alpha - \di\alpha_0\|_{C^k}) \qquad \forall \alpha\in \Omega^1(M), 
\end{equation}
for a suitable modulus of continuity $\omega_k$. We deduce the estimates
\[
\begin{split}
\|\alpha(K(\di\alpha)) - 1 \|_{C^k} &\leq \|\alpha(K(\di \alpha)) - \alpha(R_{\alpha_0})\|_{C^k} + \|\alpha(R_{\alpha_0}) - \alpha_0(R_{\alpha_0}) \|_{C^k} \\ &\leq \|\alpha\|_{C^k} \|K(\di \alpha) - R_{\alpha_0}\|_{C^k} + \|\alpha-\alpha_0\|_{C^k} \|R_{\alpha_0}\|_{C^k} \\ & \leq  \|\alpha\|_{C^k} \omega_k( \|\di \alpha-\di \alpha_0\|_{C^k}) + \|\alpha - \alpha_0\|_{C^k}  \|R_{\alpha_0}\|_{C^k},
\end{split}
\]
from which we obtain the bounds
\begin{equation}
\label{aaa1}
\|\alpha(K(\di\alpha)) - 1 \|_{C^k} \leq \omega_k' (\|\alpha-\alpha_0\|_{C^{k+1}}) \qquad \forall k\geq 0,
\end{equation}
for a suitable sequence of moduli of continuity $\omega_k'$. 

Let $\delta>0$ be such that $\omega_0'(\delta)\leq 1/2$. If the contact form $\alpha$ on $M$ satisfies $\|\alpha-\alpha_0\|_{C^1}<\delta$ then (\ref{aaa1}) implies that $\alpha(K(\di \alpha))$ is uniformly bounded away from zero and we have bounds 
\begin{equation}
\label{aaa2}
\Bigl\|\frac{1}{\alpha(K(\di\alpha))} - 1 \Bigr\|_{C^k} \leq \omega_k''(\|\alpha-\alpha_0\|_{C^{k+1}}) \qquad \forall k\geq 0,
\end{equation}
for a suitable sequence of moduli of continuity $\omega_k''$. The desired estimate for the $C^k$-norm of $R_{\alpha} - R_{\alpha_0}$ now follows from (\ref{formR}), (\ref{aaa0}) and (\ref{aaa2}).

\section{Appendix: Bottkol's theorem}
\label{appbottkol}
\subsection{The statement of the theorem}

Let $M$ be a smooth closed manifold and $X_0$ a smooth vector field on $M$ all of whose orbits are periodic with the same minimal period $T_0$. The flow $\phi_{X_0}^t$ of $X_0$ induces a free $S^1$-action on $M$. 

We fix a Riemannian metric $g$ on $M$ such that the diffeomorphisms $\phi_{X_0}^t$ are isometries for all $t\in \R$. In order to construct a metric with this property, it is enough to start from any metric on $M$ and average it on the orbits of $\phi_{X_0}$. 

For every integer $k\geq 0$, we denote by $\mathfrak{X}^k(M)$ the vector space of $C^k$ vector fields on $M$ endowed with the $C^k$-norm induced by $g$. The symbol $\mathfrak{X}(M)$ denotes the space of smooth vector fields on $M$.

Let now $U\in \mathfrak{X}^0(M)$ be a continuous vector field. First, we can average $U$ on the orbits of $X_0$, producing the following $\phi_{X_0}$-invariant vector field:
\[
\overline{U}(x) := \frac{1}{T_0} \int_0^{T_0} \di\phi_{X_0}^{-t} [ U(\phi_{X_0}^t(x))]\, \di t.
\]

Second, we can define a continuous map
\[
u: M \rightarrow M, \qquad u = \exp\circ\, U,
\]
where 
\[
\exp: TM \rightarrow M
\]
denotes the exponential mapping associated to the metric $g$. The map $u$ is $C^k$ if $U$ is $C^k$ and is a $C^k$ diffeomorphism if $U$ is $C^k$ and $C^1$-small.

Third, for every $x\in M$ we denote by
\[
P(U)_x : T_x M \rightarrow T_{u(x)} M
\]
the linear map that is induced by the Jacobi fields along the geodesic $t\mapsto \exp(tU(x))$, $t\in [0,1]$, vanishing at $t=0$:
\[
P(U)_x v := \frac{\di}{\di s}\Bigr|_{s=0}   \exp_x(U(x)+sv) = \di^v \exp(U(x))[v],
\]
where $\di^v\exp$ denotes the vertical differential of the exponential map:
\[
\di^v \exp : TM \cong T^v TM \rightarrow TM.
\]
The map $P(U)_x$ is an isomorphism provided that $\|U\|_{C^0} < r_{\mathrm{inj}}$, where $r_{\mathrm{inj}}$ denotes the injectivity radius of $(M,g)$.
The aim of this appendix is to discuss the proof and some consequences of the following result.

\begin{thm}
	\label{bottkol}
	There exists $\delta>0$ such that for every $X\in \mathfrak{X}(M)$ with $\|X-X_0\|_{C^1} < \delta$ there is a pair of vector fields $U,V\in \mathfrak{X}(M)$ and a smooth function $h: M \rightarrow \R$ such that:
	\begin{enumerate}[(i)]
		\item $P(U) V = \di u [X_0] - h X \circ u$, where $u=\exp\circ\, U$;
		\item $\overline{U}=0$;
		\item $\mathcal{L}_{X_0} V =0$;
		\item $g(V,X_0)=0$;
		\item $\mathcal{L}_{X_0} h =0$.
	\end{enumerate}
	Moreover, for every integer $k\geq 1$ we have the bound
	\begin{equation}
	\label{boundsappe}
	\max\{ \|U\|_{C^k}, \|\mathcal{L}_{X_0} U\|_{C^k},  \|V\|_{C^k}, \|h-1\|_{C^k} \} \leq \omega_k(\|X-X_0\|_{C^k}),
	\end{equation}
	for some modulus of continuity $\omega_k$.
\end{thm}

Under the stronger assumption that $X$ is $C^2$-close to $X_0$, the existence of $C^1$ vector fields $U$, $V$ and of a $C^1$ function $h$ satisfying (i)-(v) was proven by Bottkol in \cite[Theorem 1 and Lemma A]{bot80}, building on ideas of Weinstein and Moser from \cite{wei73a,wei73b,mos76}. Actually, Bottkol's setting is more general: The flow of the vector field $X_0$ is $T_0$-periodic only on a submanifold of $M$ that satisfies a suitable non-degeneracy assumption. The fact that the $C^2$-closeness assumption can be replaced by $C^1$-closeness by adapting an argument from \cite{mos76} is explicitly observed in \cite{bot80}. 

Up to reducing the positive number $\delta$ in the above theorem, we can assume that $U$ is sufficiently $C^1$-small so that $u=\exp \circ\, U$ is a diffeomorphism. In this case, condition (i) can be rewritten as
\begin{equation}
\label{inew}
h \, u^* X = X_0 - \mathscr{Q}[ V],
\end{equation}
where
\[
\mathscr{Q} := \di u^{-1} \circ P(U)
\]
is a linear automorphism of $TM$ lifting the identity. Note that the bounds on $U$ from (\ref{boundsappe}) imply for all $k\geq 1$ that
\begin{equation}
\label{boundsappe2}
\max\Big\{ \mathrm{dist}_{C^{k}}(u,\mathrm{id}),\  \|\mathscr{Q}-\mathrm{id}\|_{C^{k-1}},\  \mathrm{dist}_{C^{k}}(\di u\circ \mathscr{Q},\mathrm{id}) \Big\} \leq \omega_k(\|X-X_0\|_{C^{k}}),
\end{equation}
for suitable moduli of continuity $\omega_k$. In this way, we have obtained the formulation of Bottkol's theorem that we stated as Theorem \ref{tbot} and used in the proof of the normal form.

\subsection{Application to the existence of short periodic orbits}

Theorem \ref{bottkol} can be used to prove the existence of closed orbits for vector fields $X$ that are $C^1$-close to the vector field $X_0$, all of whose orbits are closed and have the same minimal period $T_0$. Indeed, denote by 
\[
\pi: M \rightarrow B
\]
the projection onto the quotient induced by the free $S^1$-action given by the flow of $X_0$. Conditions (iii) and (iv) in Theorem \ref{bottkol} imply that there is a smooth vector field $\widehat{V}$ on $B$ such that $\di\pi[V] = \widehat{V}\circ \pi$ and $V(x)=0$ if and only if $\widehat{V}(\pi(x))=0$. Let $b\in B$ be a zero of $\widehat{V}$. Then $V$ vanishes on the circle $\pi^{-1}(b)$, and (\ref{inew}) implies that $u^*X$ is parallel to $X_0$ along this circle. Therefore, $\pi^{-1}(b)$ is a closed orbit of $u^*X$ of period $T=h(x) T_0$, where $x$ is any point on $\pi^{-1}(b)$ (by (v), $h$ is constant on $\pi^{-1}(b)$). We conclude that the original vector field $X$ has the periodic orbit $u(\pi^{-1}(b))$, which is close to $\pi^{-1}(b)$ and has period $T$ close to $T_0$.

Therefore, any zero of the vector field $\widehat{V}$ on $B$ corresponds to a closed orbit of $X$ that bifurcates from the manifold of closed orbits of $X_0$ and has period close to $T_0$. In particular, if the Euler characteristic of $B$ does not vanish, $X$ must have closed orbits of this kind.

On the other hand, it is well known that, under the above assumptions on $X_0$, the following fact holds true: For every $\epsilon>0$ there exists $\rho>0$ such that if $\|X-X_0\|_{C^1}<\rho$ then all non-iterated closed orbits of $X$ have period that is either contained in the interval $(T_0-\epsilon,T_0+\epsilon)$ or larger than $1/\epsilon$, see \cite[Corollary 1]{ban86}. It is then natural to ask whether the zeroes of the vector field $\widehat{V}$ actually detect all the closed orbits of $X$ with period in the interval $(T_0-\epsilon,T_0+\epsilon)$. The next result says that this is indeed true, provided that $X$ is $C^2$-close to $X_0$.

\begin{prop}
\label{Vvedetutto}
For every $\epsilon>0$ there exists $\rho>0$ such that if $\|X-X_0\|_{C^2}<\rho$ then the following facts hold:
\begin{enumerate}[(i)]
\item All the non-iterated closed orbits of $X$ have period that is either contained in the interval $(T_0-\epsilon,T_0+\epsilon)$, or is larger than $1/\epsilon$.
\item The closed orbits of $X$ with period in the interval $(T_0-\epsilon,T_0+\epsilon)$ are precisely the curves of the form $u(\pi^{-1}(b))$, where $b$ is a zero of the vector field $\widehat{V}$.
\end{enumerate}
\end{prop}

\begin{proof}
Assume, without loss of generality, that $T_0=1$.
As recalled above, statement (i) can be derived from \cite[Corollary 1]{ban86}. However, we will deduce both (i) and (ii) simultaneously from Bottkol's Theorem \ref{bottkol}. Assume that 
\begin{equation}
\label{VV1}
\|X-X_0\|_{C^2} < \rho,
\end{equation}
for some positive number $\rho$ whose size will be specified along the proof. Using the notation introduced above, we have by (\ref{boundsappe}) and (\ref{boundsappe2})
\begin{equation}
\label{VV2}
\max\{\|V\|_{C^2}, \|h-1\|_{C^2}, \|\mathscr{Q} - \mathrm{id}\|_{C^1}\} \leq \omega_2(\rho).
\end{equation}
In particular, if $\rho$ is small enough we have
\begin{equation}
\label{VV3}
\|h-1\|_{C^0} < \epsilon.
\end{equation}
Set $Y:= u^*X$, so that (\ref{inew}) gives us
\begin{equation}
\label{VV4}
Y  =  \frac{1}{h} X_0 - \frac{1}{h} \mathscr{Q} [V],
\end{equation}
and together with (\ref{VV2}) we obtain the bound
\begin{equation}
\label{VV5}
\|Y-X_0\|_{C^1} \leq \omega(\rho),
\end{equation}
for a suitable modulus of continuity $\omega$.

Since the diffeomorphism $u$ conjugates the flows of $Y$ and $X$, it suffices to prove the following fact: For every closed orbit $\gamma: \R/T \Z \rightarrow M$ of $Y$ of period $T\leq 1/\epsilon$ there is a point $b\in B$ with $\widehat{V}(b)=0$ such that
\[
\gamma(\R/T \Z) = \pi^{-1}(b). 
\]
Indeed, if this is the case then $V$ vanishes along $\gamma$, and the identity (\ref{VV4}) implies that $T$ agrees with the (constant) value of $h$ on $\pi^{-1}(b)$. Thus, the bound $|T-1|< \epsilon$ follows from inequality (\ref{VV3}). 

Let us prove the fact stated above. The upper bound $1/\epsilon$ on $T$ guarantees that if $Y$ is $C^0$-close enough to $X_0$ then $\gamma(\R/T \Z)$ remains close to some fiber $\pi^{-1}(b)$ of the $S^1$-bundle $\pi : M \rightarrow B$. Therefore, by choosing $\rho$ small enough we may assume that $\gamma(\R/T \Z)$ is contained in a trivializing neighborhood $\pi^{-1}(B_0)$ for such a bundle. We can then identify $B_0\subset B$ with an open set of $\R^{d-1}$, where $d=\dim M$, and $\pi^{-1}(B_0)$ with the product
\[
B_0 \times \T \subset \R^{d-1} \times \T
\]
in such a way that $\pi$ is the projection onto the first factor and $X_0 = \partial_{\theta}$, where $\theta$ denotes the variable in $\T:= \R/\Z$. By this identification, the closed orbit $\gamma$ has components
\[
\gamma(t) = (\beta(t),\theta(t)) \in B_0 \times \T.
\]
By projecting the equation
\begin{equation}
\label{VV6}
\gamma' = Y(\gamma)
\end{equation}
onto $\R^{d-1}$ we obtain the following equation for $\beta: \R/T \Z \rightarrow B_0$
\begin{equation}
\label{VV7}
\beta'(t) = A(t) [ V(\beta(t)) ],
\end{equation}
where $A$ is the closed path of linear mappings
\[
A(t) :=  - \frac{1}{h(\beta(t))} \pi \circ \mathscr{Q}_{\gamma(t)} : \R^d \rightarrow \R^{d-1}.
\]
In (\ref{VV7}) we have used that $\mathcal L_{X_0}V=0$. From (\ref{VV2}), (\ref{VV5}) and (\ref{VV6}) we deduce that the path $A$ is $C^1$-close to the constant path $-\pi$, and hence
\begin{equation}
\label{VV8}
\max \{\|A+\pi\|_{C^0}, \|A'\|_{C^0} \} \leq \omega_1(\rho),
\end{equation}
for a suitable modulus of continuity $\omega_1$. We denote by $E$ the vector bundle over $B_0 \times \T$ whose fibers are the $(d-1)$-dimensional $g$-orthogonal complements of $\R \partial_{\theta}$. If $\rho$ is small enough, (\ref{VV8}) implies that $A(t)$ maps the fiber $E_{\gamma(t)}$ of $E$ at $\gamma(t)$ isomorphically onto $\R^{d-1}$ and, if we denote by
\[
A(t)^{-1} : \R^{d-1} \rightarrow E_{\gamma(t)} 
\]
the inverse of this restriction, we have
\begin{equation}
\label{VV9}
\|A^{-1}\|_{C^0} \leq c,
\end{equation}
for a suitable positive number $c$. By differentiating (\ref{VV7}) with respect to $t$ we obtain
\[
\beta'' = A' [V(\beta)] + A \circ \di V(\beta) [\beta'] = A' \circ A^{-1} [\beta'] + A \circ \di V(\beta) [\beta'],
\]
where in the second equality we have used (\ref{VV7}) again and the fact that $V$ takes values in the vector bundle $E$. The above identity, together with (\ref{VV2}),  (\ref{VV8}) and (\ref{VV9}) yields
\begin{equation}
\label{VV10}
|\beta''| \leq \omega_2(\rho) |\beta'|,
\end{equation}
for a suitable modulus of continuity $\omega_2$. Since $T\leq 1/\epsilon$, the Poincar\'e inequality applied to the map $\beta': \R/T \Z \rightarrow \R^{d-1}$, which has vanishing integral, gives us
\begin{equation}
\label{VV11}
\|\beta'\|_{L^2(\R/T\Z)} \leq \frac{T}{2\pi} \|\beta''\|_{L^2(\R/T\Z)} \leq \frac{1}{2\pi \epsilon}  \|\beta''\|_{L^2(\R/T\Z)}.
\end{equation}
If we choose $\rho$ so small that $\omega_2(\rho)$ is less than $2\pi \epsilon$, (\ref{VV10}) and (\ref{VV11}) force $\beta'$ to be identically zero. Therefore, $\beta(t)=b$ for every $t\in \R/T\Z$, for some $b\in B_0\subset B$. Equation (\ref{VV7}) implies that $V$ vanishes on $\pi^{-1}(b)$, and we conclude that
\[
\gamma(\R/T\Z) = \pi^{-1}(b),
\]
where $b$ is a zero of $\widehat{V}$, as we wished to prove.
\end{proof}

\subsection{The proof}

The proof of Theorem \ref{bottkol} we exhibit here is different from Bottkol's one: We obtain the triplet $(U,V,h)$ with low regularity properties by a rather straightforward application of the inverse mapping theorem, building on an idea we learned from Kerman, see \cite[Proposition 3.4]{ker99}, and then we prove its smoothness, together with the bounds (\ref{boundsappe}), by a standard argument that appears, for instance, in \cite{mos76}.

Without loss of generality, we assume that the period of the flow $\phi_{X_0}$ is 1 and we denote by $\T:= \R/\Z$ the 1-torus. We introduce the following space of continuous vector fields that are continuously differentiable along $X_0$ and have vanishing average on the orbits of $X_0$:
\[
\mathscr{U} := \{ U \in \mathfrak{X}^0(M) \mid \mathcal{L}_{X_0} U \mbox{ exists and is continuous on $M$, } \overline{U}=0\}.
\]
The norm
\[
\|U\|_{\mathscr{U}} := \|U\|_{C^0} + \|\mathcal{L}_{X_0} U\|_{C^0}
\]
turns $\mathscr{U}$ into a Banach space. We denote by $\mathscr{U}_{\mathrm{inj}}$ the open subset of $\mathscr{U}$ consisting of those vector fields $U\in \mathscr{U}$ such that $\|U\|_{C^0} < r_{\mathrm{inj}}$.

We consider also the following space of continuous vector fields that are orthogonal to $X_0$ and $\phi_{X_0}$-invariant:
\[
\mathscr{V} := \{ V \in \mathfrak{X}^0(M) \mid g(V,X_0) = 0 \mbox{ and } V(\phi^t_{X_0}(x)) = \di\phi^t_{X_0}(x)[V(x)] \; \forall x\in M, \; t\in \T\}.
\]
This is a closed linear subspace of $\mathfrak{X}^0(M)$, and hence a Banach space with the $C^0$-norm. Finally, we consider the following space of continuous $\phi_{X_0}$-invariant real functions on $M$:
\[
\mathscr{H} := \{ h\in C^0(M) \mid h(\phi_{X_0}^t(x)) = h(x) \; \forall x\in M, \; t\in \T\}.
\] 
The space $\mathscr{H}$ is also Banach with the $C^0$-norm.

When evaluated at $x\in M$, the identity (i) in the statement of Theorem \ref{bottkol} is an equality of vectors in $T_{u(x)} M$. When $U\in \mathscr{U}_{\mathrm{inj}}$ we can rearrange this equality  as an identity for vector fields on $M$ by applying the inverse of the isomorphism $P(U)_x: T_x M \rightarrow T_{u(x)} M$ to both sides. We obtain the identity
\[
V=  P(U)^{-1}  \di u[X_0] - h P(U)^{-1} X(u).
\]
This shows that the triplet $(U,V,h)$ we are looking for is a zero of the following map
\[
\Phi_X : \mathscr{U}_{\mathrm{inj}} \times \mathscr{V} \times \mathscr{H} \rightarrow \mathfrak{X}^0(M), \qquad \Phi_X(U,V,h) = P(U)^{-1} \di u[X_0] - h P(U)^{-1} X(u) - V,
\]
where we are setting as usual $u:= \exp\circ\,U$. Indeed, this map is well-defined because
\[
\di u[X_0] = \di^h \exp(U)[X_0] + \di^v \exp(U)[\nabla_{X_0} U]
\]
and $\nabla_{X_0} U = \mathcal{L}_{X_0} U +\nabla_U X_0$ is a continuous vector field. Here, $\di^h\exp(U)$ and $\di^v\exp(U)$ denote the horizontal and vertical derivatives of the map $\exp:TM\to M$ at the point $U$.
The usual facts about composition operators imply that $\Phi_X$ is continuously differentiable. Moreover, the map
\[
(X,U,V,h) \mapsto \Phi_X(U,V,h),  \qquad \mbox{respectively} \qquad  (X,U,V,h) \mapsto \di\Phi_X(U,V,h)
\]
is continuous from 
\[
\mathfrak{X}(M) \times \mathscr{U}_{\mathrm{inj}} \times \mathscr{V} \times \mathscr{H},
\]
where the space $\mathfrak{X}(M)$ is given the $C^1$-topology, into $\mathfrak{X}^0(M)$, respectively into the space of bounded operators from $\mathscr{U} \times \mathscr{V} \times \mathscr{H}$ to $\mathfrak{X}^0(M)$ endowed with the operator norm.

The map $\Phi_{X_0}$ sends $(0,0,1)$ to $0$. Moreover, after some computations one gets the formula
\[
\di\Phi_{X_0}(0,0,1)[(U,V,h)] = \nabla_{X_0} U - \nabla_U X_0 - h X_0 - V =\mathcal{L}_{X_0} U - h X_0 - V
\]
for the differential of $\Phi_{X_0}$ at $(0,0,1)$. In the next lemma we show that this operator is an isomorphism.

\begin{lem}
	\label{linear}
	The linear operator
	\[
	\mathscr{U} \times \mathscr{V} \times \mathscr{H} \rightarrow \mathfrak{X}^0(M), \qquad (U,V,h) \mapsto \mathcal{L}_{X_0} U - h X_0 - V,
	\]
	is an isomorphism.
\end{lem}

\begin{proof}
	We must prove that for every vector field $W\in \mathfrak{X}^0(M)$ there exists a unique triplet $(U,V,h)$ in $\mathscr{U}\times \mathscr{V} \times \mathscr{H}$ such that
	\begin{equation}
	\label{linear1}
	\mathcal{L}_{X_0} U = h X_0 + V + W.
	\end{equation}
	By the definition of the Lie derivative, the equation
	\[
	\mathcal{L}_{X_0} U = Y
	\]
	is equivalent to the integral formulation
	\[
	U( \phi_{X_0}^t(x) ) = \di\phi_{X_0}^t(x) \left[ U(x) + \int_0^t \di\phi_{X_0}^{-s} \bigl[ Y (\phi_{X_0}^s(x)) \bigr]\, \di s \right],
	\]
	for every $x\in M$ and every $t\in\T$. Using that the vector fields $X_0$ and $V$ and the function $h$ are $\phi_{X_0}$-invariant, equation (\ref{linear1}) can then be rewritten as
	\begin{equation}
	\label{linear2}
	U( \phi_{X_0}^t(x) ) = \di\phi_{X_0}^t(x) \left[ U(x) + t\, h(x) X_0(x) + t \,V(x) + \int_0^t \di\phi_{X_0}^{-s} \bigl[ W (\phi_{X_0}^s(x)) \bigr]\, \di s \right].
	\end{equation}
	Since the flow of $X_0$ gives us a free action of $\T$, the above formula defines a continuous vector field $U$ on $M$ if and only if the term in square brackets equals $U(x)$ for $t=1$, i.e.\ if and only if
	\begin{equation}
	\label{linear3}
	h(x) X_0(x) + V(x) = - \int_0^1 \di\phi_{X_0}^{-s} \bigl[ W (\phi_{X_0}^s(x)) \bigr]\, \di s = - \overline{W}(x).
	\end{equation}
	Given $W\in \mathfrak{X}_0(M)$, the above equation uniquely defines a real number $h(x)$ and a vector $V(x)$ in $T_x M$ that is orthogonal to $X_0(x)$. The fact that the averaged vector field $\overline{W}$ is continuous and $\phi_{X_0}$-invariant implies that the vector field $V$ and the function $h$ that are defined by (\ref{linear3}) are also continuous and $\phi_{X_0}$-invariant, and hence belong to $\mathscr{V}$ and $\mathscr{H}$, respectively.
	
	Thanks to (\ref{linear3}), equation (\ref{linear2}) becomes
	\begin{equation}\label{e:UU}
	U( \phi_{X_0}^t(x) ) = \di\phi_{X_0}^t(x) \left[ U(x) - t\, \overline{W}(x) + \int_0^t \di\phi_{X_0}^{-s} \bigl[ W (\phi_{X_0}^s(x)) \bigr]\, \di s \right].
	\end{equation}
	Therefore, the condition $\overline{U}=0$ reads 
	\[
	\begin{split}
	0 &= \overline{U}(x) = \int_0^1 \left( U(x) - t \,\overline{W}(x) + \int_0^t \di\phi^{-s}_{X_0}  \bigl[ W (\phi_{X_0}^s(x)) \bigr]\, \di s \right)\, \di t \\ &= U(x) - \frac{1}{2} \overline{W}(x) +  \int_0^1 \left( \int_0^t \di\phi_{X_0}^{-s} [W(\phi_{X_0}^s(x))] \, \di s \right) \, \di t \\ & = U(x) - \frac{1}{2} \overline{W}(x) + \overline{W}(x) - \int_0^1 t \, \di\phi_{X_0}^{-t} [W(\phi_{X_0}^t(x))] \, \di t \\ & = U(x) - \int_0^1 \left( t - \frac{1}{2} \right) \di\phi_{X_0}^{-t} [W(\phi_{X_0}^t(x))] \, \di t,
	\end{split}
	\]
where we have integrated by parts. The above equation determines $U$ uniquely:
	\begin{equation}
	\label{linear4}
	U(x) =\int_0^1 \left( t - \frac{1}{2} \right) \di\phi_{X_0}^{-t} [W(\phi_{X_0}^t(x))] \, \di t.
	\end{equation}
This formula defines a continuous vector field $U$ which has zero average and satisfies \eqref{e:UU}. Thus, $U$ and the pair $(V,h)\in \mathscr{V} \times \mathscr{H}$ that is defined by (\ref{linear3}), form the unique solution of (\ref{linear1}). This equation implies that $\mathcal{L}_{X_0} U$ is continuous, so the vector field $U$ belongs to $\mathscr{U}$.
\end{proof}

The regularity properties of $\Phi_X$ discussed above and the invertibility of $\di\Phi_{X_0}(0,0,1)$ allow us to apply the parametric inverse mapping theorem and to conclude that there is a positive number $\delta$ and an open neighborhood $\mathscr{N}$ of $(0,0,1)$ in $\mathscr{U}_{\mathrm{inj}} \times \mathscr{V} \times \mathscr{H}$ such that for every $X\in \mathfrak{X}(M)$ with $\|X-X_0\|_{C^1}<\delta$ the restriction of $\Phi_X$ to $\mathscr{N}$ is a $C^1$ diffeomorphism onto an open neighborhood of $0$ in $\mathfrak{X}^0(M)$. In particular, if $\|X-X_0\|_{C^1}<\delta$ then there exists a unique $(U,V,h)\in \mathscr{N}$ such that 
\[
\Phi_X(U,V,h)=0.
\]
Moreover, the inverse of $\Phi_X|_{\mathscr{N}}$ depends continuously on $X\in \mathfrak{X}(M)$ with respect to the $C^1$ topology and hence
\begin{equation}
\label{zerobd}
\max\{ \|U\|_{\mathscr{U}}, \|V\|_{C^0}, \|h-1\|_{C^0}\} \leq \omega_0 (\|X-X_0\|_{C^1})
\end{equation}
for a suitable modulus of continuity $\omega_0$.
Up to reducing the size of $\delta$ and $\mathscr{N}$, we may also assume that
\begin{equation}
\label{bdoninv}
\|\di\Phi_X(U,V,h)^{-1}\| \leq c \qquad \forall (U,V,h)\in \mathscr{N}, \; \forall X\in \mathfrak{X}(M) \mbox{ with } \|X-X_0\|_{C^1} < \delta,
\end{equation}
for a suitable positive number $c$.

There remains to prove that $U$, $V$ and $h$ are smooth, and that the bounds (\ref{boundsappe}) hold. Indeed, smooth zeros of $\Phi_X$ satisfy the conditions (i)-(v) of Theorem \ref{bottkol} and hence the following lemma concludes the proof of this theorem.

\begin{lem}\label{l:lemmabounds}
	The maps  $U$, $V$ and $h$ are smooth and for every integer $k\geq 1$ we have the bound
	\begin{equation}\label{lemmabounds}
\max\{ \|U\|_{C^k}, \|\mathcal{L}_{X_0} U\|_{C^k}, \|V\|_{C^k}, \|h-1\|_{C^k} \} \leq \omega_k(\|X-X_0\|_{C^k}),
\end{equation}
for some modulus of continuity $\omega_k$.
\end{lem}

\begin{proof}
	Since the matter is local, it is enough to consider the special case in which $M$ is a torus $\T^d$ and $X_0$ is the constant vector field $\partial_{x_1}$. In this case, $\mathcal{L}_{X_0} U$ is just $\partial_{x_1} U$. In order to simplify the notation we set
	\[
	W:= (U,\partial_{x_1} U,V,h) : \T^d \rightarrow \R^{3d+1}
	\]
	so that the map $\Phi_X$ becomes a multiplication operator of the form
	\[
	\Phi_X(U,V,h)  = \varphi_X(\,\cdot\,,W),
	\]
	for a suitable smooth map
	\[
	\varphi_X : \T^d \times \R^{3d+1}  \rightarrow \R^d.
	\]
	Note that the differential of order $k$ of $\varphi_X$ depends on the derivatives up to order $k$ of the smooth vector field $X$. Differentiation yields the identity
	\begin{equation}
	\label{rido}
	\di\Phi_X(U,V,h) [ (\widehat{U}, \widehat{V}, \widehat{h})] = \di_2 \varphi_X(\,\cdot\,,W)[\widehat{W}],
	\end{equation}
	where 
	\[
	\widehat{W}:= (\widehat{U},\partial_{x_1} \widehat{U},\widehat{V},\widehat{h}).
	\]
	We denote by $\tau_y$ the translation operator by the vector $y\in \R^d$:
	\[
	(\tau_y W)(x) := W(x+y) \qquad \forall x\in \T^d.
	\]
	The fact that $(U,V,h)$ is a zero of $\Phi_X$ implies that
	\[
	\varphi_X(\,\cdot\,,W)=0.
	\]
	A first order expansion for $y\rightarrow 0$ then gives us
	\[
	\begin{split}
	0 &= \varphi_X(\,\cdot+y,\tau_y W) - \varphi_X(\,\cdot\,,W)\\ & = \varphi_X(\,\cdot+y,\tau_y W) - \varphi_X(\,\cdot\,,\tau_y W) + \varphi_X(\,\cdot\,,\tau_y W)- \varphi_X(\,\cdot\,,W) \\ &= \di_1 \varphi_X(\,\cdot\,,\tau_y W)[y] + o(|y|) + \di_2 \varphi_X(\,\cdot\,,W)[\tau_y W - W] + o(|\tau_y W - W|)\\ &=\di_1 \varphi_X(\,\cdot\,,W)[y] + o(|y|) + \di_2 \varphi_X(\,\cdot\,,W)[\tau_y W - W] + o(|\tau_y W - W|),
	\end{split}
	\]
	where we have used that $\di_1 \varphi_X(\,\cdot\,,\tau_y W)[y]-\di_1 \varphi_X(\,\cdot\,,W)[y]=o(|y|)$. Using identity (\ref{rido}) this can be reformulated in the following way
	\[
	\begin{split}
	\di\Phi_X(U,V,h) &[ (\tau_y U- U,\tau_y V-V,\tau_yh-h)] = - \di_1 \varphi_X(\,\cdot\,,W)[y] + o(|y|) \\ &+ o(\|\tau_y U - U\|_{\mathscr{U}} + \|\tau_y V-V\|_{C^0} + \|\tau_y h-h\|_{C^0}).
	\end{split}
	\]
	By applying the inverse of the operator $\di\Phi_X(U,V,h)$ we find
	\[
	\begin{split}
	(\tau_y U- U,\tau_y V-V,\tau_yh-h) = &- \di\Phi_X(U,V,h)^{-1} \di_1 \varphi_X(\,\cdot\,,W)[y] + o(|y|) \\ &+ o(\|\tau_y U - U\|_{\mathscr{U}} + \|\tau_y V-V\|_{C^0} + \|\tau_y h-h\|_{C^0}),
	\end{split}
	\]
which shows that the maps $U$, $V$ and $h$ are of class $C^1$ with
\begin{equation}
\label{ildiff}
(\partial_{x_i} U,\partial_{x_i} V,\partial_{x_i} h) = - \di\Phi_X(U,V,h)^{-1} \partial_{x_i}\varphi_X(\cdot,W) \qquad\forall\,i=1,\ldots,d,
\end{equation}
where $\partial_{x_i}\varphi_X:=\di_1\varphi_X[\partial_{x_i}]$. This also shows that $\partial_{x_i} U$ belongs to $\mathscr{U}$, meaning that $\partial_{x_1} \partial_{x_i}U$ exists and is continuous.

If we set $w_0:= (0,1)\in \R^{3d} \times \R$, the fact that $\Phi_X(0,0,1)=0$ reads
	\[
	\varphi_{X_0}(x,w_0) = 0 \qquad \forall x\in \T^d,
	\]
	and hence
	\begin{align*}
	|\partial_{x_i} \varphi_X(x,W(x))| &\leq | \partial_{x_i} \varphi_X(x,W(x)) - \partial_{x_i} \varphi_{X_0}(x,W(x))|\\
	&\quad + |\partial_{x_i} \varphi_{X_0}(x,W(x)) - \partial_{x_i} \varphi_{X_0}(x,w_0)|\\
	&\leq \| \partial_{x_i} \varphi_X-\partial_{x_i}\varphi_{X_0}\|_{C^0}+\omega\big(\| W-w_0\|_{C^0}\big)
	\end{align*}
	for some modulus of continuity $\omega$. Since $X\mapsto \di_1\varphi_X$ and $X\mapsto W$ are continuous in the $C^1$-norm of $X$, this inequality implies a bound of the form
	\[
	\|\partial_{x_i} \varphi_X(\,\cdot\,,W)\|_{C^0} \leq \omega(\|X-X_0\|_{C^1})\qquad\forall\,i=1,\ldots,d,
	\]
	for a suitable modulus of continuity $\omega$. This bound, together with (\ref{ildiff}) and (\ref{bdoninv}), gives us a modulus of continuity $\omega$ such that for all $i=1,\ldots,d$
	\[
	\max\{ \|\partial_{x_i}U\|_{\mathscr{U}}, \|\partial_{x_i}V\|_{C^0}, \|\partial_{x_i}h\|_{C^0} \} \leq \omega(\|X-X_0\|_{C^1}).
	\]
	The above inequality and (\ref{zerobd}) imply the case $k=1$ in (\ref{lemmabounds}).  By bootstrapping the above argument we obtain that $U$, $V$ and $h$ are smooth and satisfy (\ref{lemmabounds}) for all $k\geq1$. 
\end{proof}


\providecommand{\bysame}{\leavevmode\hbox to3em{\hrulefill}\thinspace}
\providecommand{\MR}{\relax\ifhmode\unskip\space\fi MR }
\providecommand{\MRhref}[2]{%
  \href{http://www.ams.org/mathscinet-getitem?mr=#1}{#2}
}
\providecommand{\href}[2]{#2}

\end{document}